\documentclass{cpamart1}     
\startingpage{1}

\authorheadline{J.-F. Babadjian, G.A. Francfort}
\titleheadline{Continuity equation in Hencky plasticity}

\usepackage{amsthm,amssymb,color}
\usepackage{latexsym,esint}
\usepackage{mathrsfs}

\usepackage{graphicx,pgf,tikz}
\usetikzlibrary{decorations}
\tikzstyle xyax=[thin]
\tikzstyle mlin=[thick]
\tikzstyle slin=[]

\usepackage[T1]{fontenc}
\usepackage{mathptmx}

\newtheorem{theorem}{Theorem}[section]
\newtheorem{lemma}[theorem]{Lemma}

\newtheorem{proposition}[theorem]{Proposition}
\theoremstyle{definition}
\newtheorem{definition}[theorem]{Definition}
\theoremstyle{remark}
\newtheorem{remark}[theorem]{Remark}

\newcommand{\res}{\mathop{\hbox{\vrule height 7pt width .5pt depth 0pt
\vrule height .5pt width 6pt depth 0pt}}\nolimits}

\definecolor{purple}{rgb}{0.8,0.01,0.7}

\newcommand{\N}{\mathbb N}
\newcommand{\R}{\mathbb R}

\newcommand{\dist}{{\rm dist}}

\DeclareMathOperator{\Div}{div}
\newcommand{\wto}{\rightharpoonup}
\newcommand{\e}{\varepsilon}

\newcommand\p{\varphi}

\newcommand{\LL}{{\mathcal L}}
\newcommand{\HH}{{\mathcal H}}
\newcommand{\M}{{\mathcal M}}
\newcommand{\C}{{\mathcal C}}

\newcommand\h{\HH^1}

\let\O=\Omega

\newcommand\ds{\displaystyle}

\newcommand{\bpartial}{{\boldsymbol{\partial}}}

\renewcommand*\contentsname{}

\begin{document}                        


\title{Continuity equation and characteristic flow for scalar
Hencky plasticity}

\author{Jean-Fran\c{c}ois Babadjian}{Universit\'e Paris-Saclay}

\author{Gilles A. Francfort}{Universit\'e Paris-Nord \& Courant Institute}






\begin{abstract}
 {\scriptsize We investigate uniqueness issues for a   continuity equation arising out of the simplest model for plasticity,  Hencky plasticity. The associated system is of the form $\rm{ curl\;}(\mu\sigma)=0$ where $\mu$ is a nonnegative measure and $\sigma$ a two-dimensional divergence free unit vector field. After establishing the Sobolev regularity of that field, we provide a precise description of all possible geometries of the characteristic flow, as well as of the associated solutions.}
\end{abstract}

\maketitle



 {\scriptsize\tableofcontents}


\section{Introduction}

\subsection{The mathematical subtext}\label{sec:subtext}

It is by now well established that {the solutions $\mu$ of} the continuity equation
$$\Div({\mu b})=0,$$
for some given vector field $b:\R^N \to \R^N$, are closely related to the notion of characteristics, that is to the solutions $X$ of the ordinary differential equation
$$\begin{cases}
\ds\frac{dX}{ds}{(s)}= {b}(X(s)), \quad s\ge 0,\\[2mm]
X(0)=x.
\end{cases}$$
When $b$ is a Lipschitz continuous vector field, the Cauchy-Lipschitz theorem for ODE's provides a complete picture of the solutions. For less regular $b$'s, the theory of  regular Lagrangian flows initiated by R. DiPerna and P.-L. Lions \cite{DPL} and pursued  notably in  \cite{Amb,CD} has had tremendous success in handling such problems  in the context of Hamiltonian flows, {\it i.e.}, under the additional assumption that $\Div b$ is well controlled, e.g. in $L^\infty(\R^N)$.

We propose  to investigate a closely related question {in a two-dimensional setting where $b=\sigma^\perp$ is the $\pi/2$-rotation of some $\sigma \in L^\infty(\R^2;\R^2) \cap H^1_{\rm loc}(\R^2;\R^2)$. In that setting, the} continuity equation {takes} the form
\begin{equation}\label{eq.cont.plas}
{\rm curl\;}(\mu \sigma)=0 \quad \Longleftrightarrow \quad \Div (\mu\sigma^\perp)=0.
\end{equation}
The field $ \sigma$ under consideration in \eqref{eq.cont.plas} is a divergence free field such that $\Div \sigma^\perp=-{\rm curl\; \sigma}$ only belongs to $L^2_{\rm loc}(\R^2)$. As a consequence we do not control the so-called compressibility constant. The available theoretical tools developed in \cite{DPL,Amb,CD} cannot produce the kind of uniqueness obtained in e.g. \cite[Corollary 2.10]{CD}. Note that here the ODE defining the characteristic flow follows a gradient flow structure, rather than that of a Hamiltonian flow.

As  will be discussed in Subsection  \ref{sec:mech}, the existence of a nonnegative measure $\mu$ which solves \eqref{eq.cont.plas} (in a sense that will be specified in Section \ref{sec:Hencky}) is secured. However we {have no information about uniqueness}. The problem we will detail in Subsection  \ref{sec:mech} and Section \ref{sec:Hencky}  exhibits additional structure. First it lives in a bounded domain $\O$ of $\R^2$ and the  divergence free field $\sigma$ belongs to  $L^\infty(\O;\R^2) \cap H^1_{\rm loc}(\O;\R^2)$. Further,  $\mu$ is  a nonnegative bounded Radon measure supported on  $\overline\O$ and $\mu\sigma$ is well-defined as a bounded Radon measure supported on  $\overline\O$. Consequently,
there is $u \in BV(\O)$ such that $Du=\mu\sigma$. That function is further assigned {a prescribed} exterior trace $w$ on $\partial\O$. Finally, it might be so that $|\sigma|\equiv 1$ on an open subdomain of $\O$; see the example discussed in Subsection \ref{sec:mech}.
This motivates our choice of studying \eqref{eq.cont.plas} under the following assumptions:
\begin{equation}\label{eq:eq.for.sigma}\begin{cases}
\sigma\in H^1_{\rm loc}(\O;\R^2),\\
|\sigma|=1\quad\mbox{a.e. in }\O_p\mbox{ open convex subset of }\O,\\
\Div\sigma=0 \quad\mbox{in }\O.
\end{cases}
\end{equation}
From the standpoint of $\sigma$ the setting is a particular case  of that expounded upon in \cite{JOP}. The additional information here is that $\sigma\in H^1_{\rm loc}(\O;\R^2)$ and it allows us to provide a very detailed description of the characteristics which, restricted to $\O_p$, are
straight lines in the direction of $\sigma^\perp$ (a  constant field along those lines) in Section \ref{sec:rig}.

But even an intimate knowledge of the characteristics does not yield  any
uniqueness result for the solution $\mu$ to the continuity equation
\eqref{eq.cont.plas}.  In our setting we prove in Section \ref{sec:u} that {the associated function} $u$ remains constant along the characteristic lines as suggested by  the formal computation
$$
\frac{d}{ds} u(x+s\sigma^\perp\!(x))\!=\!Du(x+s\sigma^\perp\!(x))\cdot \sigma^\perp\!(x) \!=\!Du(x+s\sigma^\perp\!(x))\cdot \sigma^\perp\!(x+s\sigma^\perp\!(x))\!=\!0
$$
since $Du=\mu \sigma.$ This also is not  sufficient to claim uniqueness
of the
solution $\mu$ to the continuity equation, most notably because, as explained in Remark \ref{rem:strict-conv}, $\O_p$ cannot coincide with $\O$ so
that we  do not know how to relate the values of $u$ on $\O_p$ to the boundary values of $u$, this independently of whether or not the internal trace of $u$ on $\partial \O$ coincides with the given external trace $w$.

The full results are given in Theorem \ref{thm:main-results}. They are expressed in a slightly different language, that of plasticity because, as will become clear in the next subsection, our main motivation derives from issues of uniqueness of the plastic strain in Von Mises plasticity. The
connection with  hyperbolicity \`a la  \eqref{eq.cont.plas} is uncovered in Subsection \ref{sec:mech}.

\subsection{The specific context}\label{sec:mech}

When departing from a completely reversible behavior, fluid mechanics  essentially follows a unique path, that of viscosity. In its simplest manifestation, Euler equations cede the ground to Navier-Stokes equations which become the template for classical fluid behavior.   {Even non-Newtonian
fluids usually exhibit viscosity, although one that may depend on a variety of kinematic or internal variables. When it comes to solids, while} elasticity is the universally adopted reversible behavior, the irreversibility palette is much richer. This is so because solid mechanics encodes geometry and not only flow. As for  fluids, viscosity is one expression of dissipation, leading to various kinds of viscoelastic models which, by the way, are mathematically much easier to handle   in the case of solids. But many other kinds of dissipative behaviors may occur, together with, or separate from viscosity. Their essential distinguishing feature is rate-independence: Material response is, up to rescaling, impervious  to the loading rate. In that class, the best-established behavior is plasticity, and, within plasticity, Von Mises plasticity. While other rate-independent behaviors are still a modeling challenge, Von Mises plasticity can
be thought of as the solid equivalent to Navier-Stokes, that is an admittedly simplistic model that however contains key ingredients for explaining much of the underlying physics at the macroscopic level.

Of course, because of geometry, this assertion should be nuanced: Von Mises plasticity is a perfectly sound model as long as deformations are small, that is as long as the kinematics of the deformation does not result in large changes of shape. Models for large deformations are not completely settled at present, even in the absence of irreversibility. In spite of major advances in the past 40 years spearheaded by the work of J.M. Ball \cite{Ball}, finite elasticity is far from a complete theory, while finite plasticity is a  minefield.

Von Mises plasticity, also called Prandtl-Reuss elasto-plasticity with a Von Mises yield criterion, consists in a  system of time dependent equations below. There, we denote by $\O$ the {three-dimensional} domain under consideration and, for simplicity, place ourselves in a quasi-static setting, that is in the absence of inertia. We further assume homogeneity  and take all material parameters to be identically $1$ (with the right units).

 The displacement field $u(t):\O\to \R^3$ is constrained by a time-dependent Dirichlet boundary condition $u(t)=w(t)$ on $\partial\O_d$, a relatively open subset of $\partial\O$, which constitutes the Dirichlet part of the boundary. The associated linearized strain $Eu(t):=\frac12 (\nabla u(t)+\nabla u(t)^T)$ is additively decomposed into the elastic strain $e(t)$ (a $3 \times 3$ symmetric matrix) and the plastic strain $p(t)$ (a trace-free $3 \times 3$ symmetric matrix), {\it i.e.},
$$
Eu(t)=e(t)+p(t), \mbox{ with } {\rm tr}(p(t)) =0.
$$

 We assume, for simplicity, that the only driving mechanism, besides the imposed displacement $w(t)$, is a surface load $g(t)$;  there are no body
 loads.  With our assumptions, the Cauchy stress $\sigma(t)$,  is simply $$\sigma(t)= e(t).$$
It is in quasi-static equilibrium, {\it i.e.},
$$
\Div \sigma(t)=0 \mbox{ in }\O,\quad \sigma(t)\nu=g(t) \mbox{ on } \partial\O_n:=\partial\O\setminus\overline{\partial\O_d}
$$
($\nu$  is the outer normal to $\partial\O_n$), while its deviatoric part
$\sigma_D(t):=\sigma(t)-\frac13 {\rm tr\,}(\sigma(t)) {\rm Id}$ satisfies the Von Mises yield criterion,
$$
|\sigma_D(t,x)|\le\sqrt{\frac{2}{3}} \mbox{ at every point $x\in \O$}.
$$

The deviatoric stress  $\sigma_D(t)$ and the plastic strain rate $\dot
p(t)$ are related, at every point $x\in \O$, through the so-called flow rule
\begin{equation}\label{eq:fr-na}
\dot p(t,x) =\lambda(t,x) \sigma_D(t,x), \mbox{ with }
\begin{cases}
\lambda(t,x)\ge 0,\\
\lambda(t,x)=0 \;\mbox{ if } |\sigma_D(t,x)|<\sqrt{\frac{2}{3}}.
\end{cases}
\end{equation}
In other words, whenever the (deviatoric part of the) stress reaches the boundary of its admissible set,  the plastic strain should flow in the direction normal to that set.

\medskip

The  modern mathematical treatment of Von Mises plasticity  finds its roots in the work of  P.-M. Suquet \cite{suquet81}, later completed by various works (see {\it e.g.} \cite{T, kohn.temam,A2, anzellotti84,AnzLuc1987}).  That work was revisited some 20 years later by G. Dal Maso, A. De Simone and M. G. Mora \cite{DMDSM} within the framework of  the variational theory of rate independent evolutions popularized by A. Mielke (see {\it e.g.}\;\cite{mielke05}). The basic tenet there  is  that the evolution can be viewed as a time-parameterized set of minimization problems for the sum of the elastic energy and of the add-dissipation.  The minimizers  should also be such that an energy conservation statement, amounting to a kind of Clausius-Duhem inequality, is satisfied throughout the evolution.

In any case, for a Lipschitz bounded domain $\O$ and smooth enough $w$ and $g$ (see e.g. \cite[Remark 2.10]{FG}), the resulting evolutions $t \mapsto (u(t), e(t), p(t))$ are found to live in $AC([0,T]; BD(\O)\times L^2(\O;\R^6)\times \M(\O\cup\partial\O_d;\R^6))$. Here, $BD(\Omega)$ stands for the space of functions of bounded deformation, {\it i.e.},  integrable vector fields $v:\Omega\to {\R^3}$ whose distributional symmetrized gradient $Ev=\frac12 (Dv+Dv^T)$ is a bounded measure in $\O$ and $\M(\O\cup\partial\O_d;\R^6)$ stands for the space of {$\R^6$-valued} bounded Radon measures on $\O\cup\partial\O_d$ (see \cite{FG} under those
conditions). Further, uniqueness of $e(t)$, hence of $\sigma(t)$ is guaranteed. Such however is not the case for $p(t)$, hence for $u(t)$. The first example of  non-uniqueness was presented in \cite[Section 2.1]{suquet81} while \cite[Section 10]{demyanov} introduces the first examples of uniqueness. In those references the setting is essentially 1D.  To our knowledge, the only examples of determination of uniqueness or non-uniqueness in 3D can be found in \cite{FGM16, FGM17}. There, the discussion around uniqueness is centered around the equation that the Lagrange multiplier $\lambda$ in \eqref{eq:fr-na} must satisfy when $|\sigma_D(t,x)|=\sqrt{2/3}$.

\medskip

We next formally manipulate the equations at a given fixed time $t$. Indeed, since $\dot p= E\dot u- \dot e$, or still $\lambda \sigma_D= E\dot u- \dot \sigma$, we can use the compatibility equations for symmetrized gradient, that is, for all $1 \leq i,j,k,l\leq 3$,
$$\ds
\frac{\partial^2 (E \dot u)_{ij}}{\partial x_k\partial x_l}+\frac{\partial^2 (E \dot u)_{kl}}{\partial x_i\partial x_j}-\frac{\partial^2 (E \dot u)_{ik}}{\partial x_j\partial x_l}-\frac{\partial^2 (E \dot u)_{il}}{\partial x_j\partial x_k}=0.
$$
We obtain a system of 6 equations, namely,
\begin{multline}\label{eq:sys-hyp}
(\sigma_D)_{ij}\frac{\partial^2\lambda}{\partial x_k\partial x_l}+(\sigma_D)_{kl}\frac{\partial^2\lambda}{\partial x_i\partial x_j}-(\sigma_D)_{ik}\frac{\partial^2\lambda}{\partial x_j\partial x_l}-(\sigma_D)_{il}\frac{\partial^2\lambda}{\partial x_j\partial x_k}\\ (+ \mbox{ terms of lower order in }\lambda)= -({\rm curl}\;{\rm curl}\; ( \dot \sigma))_{ijkl}.
\end{multline}
In the example investigated in \cite{FGM16}, the stress field $\sigma$ is
constant, so that the lower order terms disappear as well as the right-hand side, and the formal manipulations can be justified. We then have to deal with a {\it bona fide} system of  second order linear partial differential equations for the measure $\lambda$ of the form
$$
D\nabla^2 \lambda(t)=0,
$$
with $D$ a constant $6\times 6$ matrix. Whenever the determinant of $D$ is not $0$,  $\nabla^2\lambda(t)\equiv 0$;  then, because of the specific setting in that example, $\lambda(t)$ is $x$-independent, from which it follows that $p(t)$ is an $x$-independent plastic strain $p(t)$. With that
result at hand, $p(t)$ is   easily shown to be unique; the {example} in  \cite{FGM17}, while more intricate,  goes along the same lines. Otherwise, the system reduces to a spatial hyperbolic equation for $\lambda(t)$ and then uniqueness depends on whether the associated characteristics coming out of $\partial\O_d$ fill the whole domain. Roughly speaking, if they do, then uniqueness is obtained. If they don't, then non-uniqueness can be drastic because plastic strains that are as badly behaved as one desires (for example plastic strains supported on Cantor sets) can appear at any time $t$ in the region not reached by those characteristics. So one could then impose a large enough homogeneous boundary condition and the stress field would consequently be constant while a possible plastic strain would be spatially homogeneous. Yet, arbitrary localized non-zero plastic strains could be superimposed at any later time. This is reminiscent of what was recently observed by C. De Lellis and L. Sz\'ekelyhidi when dealing with non-uniqueness in Euler equations (see \cite{DLS} and subsequent works),  with the important caveat that plastic strains, once turned on, cannot be turned off because of dissipation.

 \medskip

This kind of analysis of uniqueness depends on our ability to deal with a
system such as \eqref{eq:sys-hyp}. In the examples already alluded to,  the key observation is the spatial homogeneity of the stress field $\sigma(t)$, a feast that cannot be easily reproduced in a generic problem. Barring this, our toolbox is rather empty.  As a matter of fact, it is impossible to even define possible characteristics in a meaningful manner because of the lack of regularity of the stress field $\sigma(t)$. At best, it
is a locally $H^1$-function while $\lambda(t)$  is a measure.

This is why, in an attempt to simplify the problem, we address in the present paper a scalar-valued version of Von Mises plasticity in the simplest setting where geometry will play its part, that is in 2D (see \cite[Section 3.1]{BMi} for a formal derivation). Furthermore, since time evolution seems to be a complicating feature but not one that uniqueness hinges on, we propose to investigate a time independent (static) version of Von Mises plasticity, that of Hencky plasticity which actually predates evolutionary plasticity \`a la Von Mises. The  system of time-independent equations becomes in its formal version (see \eqref{eq:plast} for a more precise formulation)
$$\begin{cases}
{\rm div}\sigma=0 \quad \text{ in }\O, \\[1mm]
|\sigma|\leq 1 \quad \text{ in }\O,\\[1mm]
Du=\sigma+p \mbox{ in } \O,  \\[1mm]
u=w\text{ on } \partial \O_d,\quad \sigma\cdot\nu=g \text{ on } \partial \O_n,\\[1mm]
p=\lambda\sigma \mbox { in } \O \mbox{ with } \lambda\ge 0 \text{ and }{\lambda(1-|\sigma|)=0}.
\end{cases}$$
Once again existence of $(u,\sigma,p)$ (for a slightly relaxed problem) is {guaranteed}, this time through a straightforward minimization process (see Section \ref{sec:Hencky}). The triplet $(u,\sigma,p)$ belongs to $BV(\O)\times L^2(\O;\R^2)\times \M(\O\cup\partial\O_d;\R^2)$ and $\sigma$ is unique {and actually belongs to  $H^1_{\rm loc}(\O;\R^2)$ (see below).}

\medskip

\begin{figure}[hbtp]
\scalebox{.9}{\begin{tikzpicture}

\fill[color=lightgray]  plot coordinates { (0,-1.04) (3,-1.04) (0,1.96)};

\draw[style=very thick] (0,-1.04) -- (0,2.765) ;
\draw[style=very thick] (3,-1.04) -- (3,-.235)  ;
\draw[style=very thick] (0,2.765) -- (3,-.235);
\draw[style=very thick] (0,-1.04) -- (3,-1.04);

\draw[dashed] (-.5,3.265) -- (4,-1.235);
\draw (-1,3.1) node{\tiny $x+y=\ell$};

\draw (-1,1.04) node{\small $\sigma\cdot \nu=-\frac{1}{\sqrt2}$};
\draw(2.4,1.6) node{\small $u=\frac{a\ell}{\sqrt 2},\;a>1$};
\draw(1.3,-1.5) node{\small $u=\frac{x}{\sqrt 2}$};
\draw(1.3,-1.9) node{\tiny  width $=\mathrm{d<\ell}$};
\draw (4,1) node{$\Omega$};
\draw(4,-0.64) node{\small $\sigma\cdot \nu=\frac{1}{\sqrt2}$};
\draw(-1,.6) node{\tiny  height $=\ell$};
\draw[->] (3.7,1) -- (2.7, 0.15);

\draw(1.5,0.85) node{ $\mathbf U$};

\end{tikzpicture}}
\caption{\small  {Example of non-uniqueness.}}
\label{fig:example}
\end{figure}
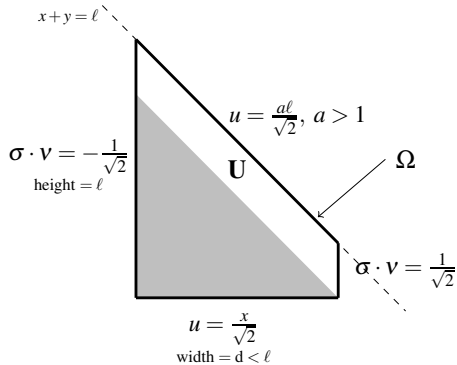

In that setting as well, uniqueness issues are intimately tied to the solving of a first order hyperbolic equation. Let us illustrate this with a very simple example. For $0<d<\ell$, take
$$\O:=\{(x,y) \in \R^2:\; 0<x<d,\; 0<y<\ell-x\},$$
$$\partial\O_d=(0,d)\times\{0\}\cup \{(x,y) \in \R^2:\; 0<x<d,\;x+y=\ell\},$$
and
$$\partial\O_n=(\{0\}\times(0,\ell)) \cup(\{d\}\times(0,\ell-d)).$$
We set
$$\begin{cases}
\ds {w(x,0)=\frac{x}{\sqrt 2}, \quad w(x,\ell-x)=\frac{a\ell}{\sqrt 2}\quad \text{ for all } x\in (0,d)}\\[2mm]
\ds g(0,y)=-\frac{1}{\sqrt 2}\quad \text{ for all }y\in(0,\ell), \\
\ds g(d,y)=\frac{1}{\sqrt 2}\quad \text{ for all } y\in(0,\ell-d),
\end{cases}$$
with $a>1$. It is then easily seen that the unique stress field is given by
{$\sigma(x,y)=\frac{1}{{\sqrt 2}}(1,1)$}, and that $\lambda$ satisfies $\frac{\partial\lambda}{\partial x}-\frac{\partial\lambda}{\partial y}=0$ from which we conclude that $\lambda$ reads as $\zeta(x+y)$, the push forward of a nonnegative bounded Radon measure on $\zeta \in \M(\R)$ by the map $(x,y) \mapsto x+y$. Consequently,
\begin{equation}\label{eq:u-uniq}
u(x,y)=\frac{x+y+Z(x+y)}{\sqrt{2}} \quad \text{ for all }(x,y) \in \O
\end{equation}
for some $Z \in BV(\R)$ with $DZ=\zeta$.

In view of Remark \ref{rmk:jump-flow-rule} below, there can be no jumps of $u$ on $(0,d)\times\{0\}$ since $\sigma\cdot\nu \neq \pm 1$ on that set
($\nu$ outer normal). Thus $\lambda=0$ on $(0,d)\times\{0\}$, which means that $\zeta = 0$ in $(0,d)$.  Using again the boundary condition and
the fact that $u$ does not jump on $(0,d)\times\{0\}$, we obtain
$$Z(t)=0 \quad \text{ for all }t \in (0,d).$$
However, there can be a jump on $\{(x,y) \in \R^2:\; 0<x<d,\;x+y=\ell\}$ since, in that case,
$\sigma\cdot\nu = 1$.
Since $d<\ell$, the part $U:=\{(x,y)\in \O:\; d-x<y<\ell-x\}$ is not traversed by characteristic lines intersecting $(0,d)\times\{0\}$ (see Figure \ref{fig:example}). So, $\lambda(x,y)=\zeta(x+y)$, where $\zeta$ is any nonnegative Radon measure. Because $w=\frac{a\ell}{\sqrt 2}$ on $\{(x,y) \in \R^2:\; 0<x<d,\;x+y=\ell\}$, we must have $Z^+(\ell)=(a-1)\ell$. In conclusion,  $Z$ can be any monotonically increasing  function such that $Z\equiv 0$ on $(0,d)$ and $Z^+(\ell)=(a-1)\ell$. This in turn
 will give rise to  many possible   $u$'s in $U$ through \eqref{eq:u-uniq}.

Note that, for $d=\ell$ the domain degenerates to a triangle, the Neumann boundary condition is only on $\{0\}\times (0,\ell)$ and the solution is unique. It corresponds to $Z\equiv 0$ on $(0,\ell)$ and $Z^+(\ell)=(a-1)\ell$.

\medskip

In the spirit our prior discussion, one should investigate the compatibility equation ${\rm curl\;}Du=0$, that is the continuity equation
$${\rm curl\;}(\mu\sigma)=0.$$
with $\mu:=\mathcal L^2+\lambda$. At this point, we are back to \eqref{eq.cont.plas} and  have to further specialize the setting by assuming that there is an open set $\O_p$  on which $|\sigma|\equiv 1$. This is the case  in the examples that were investigated in \cite{FGM16,FGM17} in a vectorial setting.

\begin{remark}The previous assumption is also related to the so-called slip line theory, widely used in mechanics to get exact solutions to plane-stress rigid-plasticity problems; see \cite{BabFr} for a description of the mathematical framework of rigid plasticity. There, the spatial hyperbolic structure of the stress equations is used in combination with the method of characteristics to construct non trivial solutions to the stress problem and to determine the associated velocities \cite[Section 5.1]{Lub}. \hfill\P
\end{remark}

\begin{remark} As mentioned in Subsection \ref{sec:subtext}, the results of \cite{JOP} will play a decisive role in the analysis. Unfortunately, those are severely limited to the scalar case, that is to a single equation of the form $\Div\sigma=0$. Extending the results to the true two-dimensional Hencky setting where $\sigma$ is a symmetric $2\times 2$-matrix and $\Div\sigma=0$ is a set of two equations would require, in the terminology of \cite{JOP}, to design appropriate entropies as in \cite{DKMO};
see Subsection \ref{stressLip} for details. Unfortunately, we are unable to accomplish such a task at this time.\hfill\P
\end{remark}

As already announced, our results are  briefly described in the next subsection.
\subsection{The results}
Section \ref{stressreg} focuses mainly on the $H^1_{\rm loc}(\O;\R^2)$-regularity of the stress field $\sigma$. That regularity   was
 first demonstrated  in \cite{S}, then in \cite{BF}, in both instances in
the  vectorial setting. This is the object of Theorem \ref{thm:Sobolev}  which offers a short derivation in the current scalar setting using a  Perzyna type approximation of the problem as in \cite{BM} (see \eqref{eq:perzyna}), in lieu of the so-called Norton-Hoff approximation used in \cite{BF}, or of the Kelvin-Voigt visco-elastic approximation used in \cite{S}. It is  our belief that the proof of that result is rather convoluted in those prior works; we strive to give a hopefully more transparent and self-contained proof.

In Subsection \ref{stressLip}, we  quickly revisit the main result in \cite{JOP} adapted to our context, that is the locally Lipschitz regularity of the stress in $\O_p$, provided that this set is open and convex (see Theorem \ref{thm:lip}).  In particular, those results imply that $\sigma$ remains constant along straight lines in $\O_p$ with direction $\sigma^\perp$  which are precisely the characteristics inside $\O_p$ (see Proposition \ref{prop:caracteristic}). This is done by adapting the results
of \cite{JOP} (see also \cite{BP,I}) which use the notion of entropies introduced in \cite{DKMO}. Again, the results of this Subsection are not new. However, the proofs in prior works do not take into account the specificities of the case at hand, that is the {\it a priori} knowledge of the Sobolev regularity of the stress. In Subsection \ref{sec:u}, we prove that any displacement field $u$ must remain constant along the characteristic lines in $\O_p$ (see Theorem \ref{thm:cst-disp}).

In Section \ref{sec:rig} we first consider boundary fans that are solutions that correspond to vortices for $\sigma$ (see Subsection \ref{bdary-fans}). This happens when two distinct characteristic lines intersect on the boundary. Note that such an intersection is impossible inside $\O_p$ because it would contradict the continuity of $\sigma$ at that point. Those
can be anywhere in $\O_p$. {We take their union} and consider the complementary set $\mathscr C$ within $\O_p$. We  show that each connected component of $\mathscr C$ intersects the boundary $\partial\O_p$. If its interior is empty, it is a characteristic line. Otherwise, its intersection with $\partial\O_p$ has one or two connected components. Furthermore, if it has two connected components, then all points in  those are traversed by a characteristic line, whereas if it has only one component, then it might be so that a single line segment within that component is not traversed by any characteristic. This is the object of Theorem \ref{thm:geom-struct} which is our main rigidity result. As far as the stress is
concerned, we show continuity of the stress at all points of the boundary
that are traversed by a characteristic (see Theorem \ref{thm:cont-sigma}). Finally, we demonstrate that, besides boundary fans, exterior fans (that are fans with an apex outside $\overline \O_p$), and areas of constant $\sigma$ (which correspond to parallel characteristic lines), one can also have  areas where the characteristic lines look like a ``continuous''
one parameter family of lines, {\it e.g.} of the form $y=x/t-t$, for $t>0$ (see Paragraph \ref{4.5.3}). We conjecture that those four situations
are the only possible ones for $\sigma$ in the region on which $|\sigma|=1$.

The behavior of any solution field $(\sigma,u)$  along the characteristic
lines in $\O_p$ (see Proposition \ref{prop:caracteristic} and Theorem \ref{thm:cst-disp})  seems beyond  reach for now, even in the scalar-valued setting, absent an additional assumption like the existence of a set with
non empty interior where $|\sigma|=1$. But even in our restrictive setting this result falls short of adjudicating uniqueness of the plastic strain $p$. This is so because the set $P:=\{x\in \O: \; |\sigma|=1\}$ is a closed set in $\O$ while we have to assume that $\O_p$ is a convex open set in the interior of $P$. In particular, we have no systematic way of relating the values of $u$ on the boundary of $\partial\O$ to those on $\partial\O_p$, except for very particular settings {(see Propositions \ref{prop:unique-u-fan} and \ref{prop:unique-C}). {If $\O$ was a convex domain and $P=\O$, then uniqueness could be obtained, at least in the case
of Dirichlet boundary conditions throughout $\partial\O$}. For more details  see Remark \ref{rem:strict-conv}.

\medskip

For the reader's convenience, we concatenate the main results in a unique
Theorem which, in its concision, somewhat hides the hyperbolic nature of the questions that are central to this paper (see Figure \ref{fig:geometry} for an  illustration of the geometric structure of the solutions).

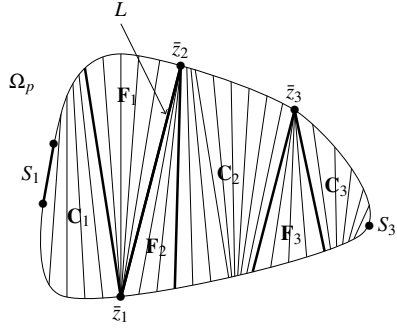
\begin{figure}[hbtp]
\scalebox{.8}{\begin{tikzpicture}

\draw plot[smooth cycle] coordinates {(-1,-1) (4,0) (3,2) (0,3) (-1,2)};

\fill[color=white]  plot coordinates {(-1.3,0.5)  (-1.12,1.5) (-1.5,1.5) (-1.5,0.5)};

\draw[style=very thick] (0,-1.04) -- (-0.6,2.765) ;
\draw (0,-1.04) -- (-0.3,2.95) ;
\draw (0,-1.04) -- (0,3) ;
\draw (0,-1.04) -- (0.3,2.97) ;
\draw (0,-1.04) -- (0.65,2.9) ;
\draw[style=very thick] (0,-1.04) -- (1,2.8) ;
\draw (0,-1.04) node{{\small $\bullet$}};
\draw (0,-1.04) node[below]{\small $\bar z_1$};
\draw (-0.2,2.3) node[right]{\small $\mathbf F_1$};

\draw[style=very thick] (1,2.8)--(0,-1.04)  ;
\draw (1,2.8)--(0.3,-0.98)  ;
\draw (1,2.8)--(0.6,-0.94)  ;
\draw[style=very thick] (1,2.8)--(0.9,-0.89)  ;
\draw (1,2.8) node{{\small $\bullet$}};
\draw (1,2.8) node[above]{\small $\bar z_2$};
\draw (0.6,-0.2) node{\small $\mathbf F_2$};

\draw[style=very thick] (2.9,2.07)--(2.2,-0.61)  ;
\draw (2.9,2.07)--(2.5,-0.54) ;
\draw (2.9,2.07)--(2.8,-0.45) ;
\draw (2.9,2.07)--(3.1,-0.36) ;
\draw[style=very thick] (2.9,2.07)--(3.4,-0.27)  ;
\draw (2.9,2.07) node{{\small $\bullet$}};
\draw (2.9,2.07) node[above]{\small $\bar z_3$};
\draw (2.85,0) node{\small $\mathbf F_3$};

\draw (-0.3,-1.05) -- (-0.7,2.66) ;
\draw (-0.6,-1.06) -- (-0.8,2.51) ;
\draw (-0.9,-1.03) -- (-0.9,2.31) ;
\draw (-1.2,-0.83) -- (-1,2) ;
\draw[style=very thick] (-1.3,0.5) -- (-1.12,1.5) ;
\draw (-0.7,0.3) node{\small $\mathbf C_1$};
\draw  (-1.3,0.5)  node{{\small $\bullet$}};
\draw  (-1.12,1.5) node{{\small $\bullet$}};
\draw (-1.2,1) node[left]{\small $S_1$};

\draw[<-]  (0.75,2) -- (0,3.5);
\draw( 0,3.5) node[above] {\small $L$};

\draw (1.1,2.78)--(1.2,-0.83)  ;
\draw (1.17,2.75)--(1.5,-0.77)  ;
\draw (1.24,2.74)--(1.8,-0.71)  ;
\draw (1.54,2.64)--(1.85,-0.7)  ;
\draw (1.84,2.54)--(1.9,-0.69)  ;
\draw (2.14,2.43)--(1.95,-0.67)  ;
\draw (2.44,2.31)--(2,-0.66)  ;
\draw (2.7,2.18)--(2.1,-0.64)  ;
\draw (1.8,1) node{\small $\mathbf C_2$};

\draw (3.2,1.84)--(3.5,-0.24)  ;
\draw (3.57,1.45)--(3.6,-0.19)  ;
\draw (3.85,1.11)--(3.7,-0.16)  ;
\draw (4.05,0.77)--(3.8,-0.11)  ;
\draw (4.15,0.50)--(3.89,-0.06)  ;
\draw (3.6,0.8) node{\small $\mathbf C_3$};
\draw (4.15,0.15) node{{\small $\bullet$}};
\draw (4.15,0.15) node[right]{\small $S_3$};

\draw (-2,2.5) node[right]{\small $\Omega_p$};
\end{tikzpicture}}
\caption{\small  {An example of geometry with three boundary fans $\mathbf F_1$, $\mathbf F_2$, $\mathbf F_3$ with apexes, respectively, $\bar z_1$, $\bar z_2$, $\bar z_3$, three connected components $\mathbf C_1$, $\mathbf C_2$ and $\mathbf C_3$, and a characteristic line $L$ which is the intersecting characteristic line segment between the (open) fans $\mathbf F_1$ and $\mathbf F_2$. The connected components $\mathbf C_1$ and $\mathbf C_3$ have one characteristic line segment on their boundaries, the characteristic boundary set $S_1$ is a closed line segment, while $S_3$ is a
single point. The connected component $\mathbf C_2$ has two characteristic line segments on its boundary.}}
\label{fig:geometry}
\end{figure}

\begin{theorem}[Main results]\label{thm:main-results}
Assume that $\O$ is a Lipschitz bounded domain of $\R^2$ and that $w\in H^1(\O)$. The minimization problem
$$
\inf \Big\{ \frac12 \int_\O |\sigma|^2\, dx + |p|(\overline \O) : \; (u,\sigma,p) \in \mathcal A_w\Big\}
$$
 with
\begin{multline*}\mathcal A_w:=\{(u,\sigma,p) \in BV(\O) \times L^2(\O;\R^2) \times \M(\overline \O;\R^2):
 \\ Du=\sigma+p \text{ in } \O,\ p=(w-u)\nu \HH^1\text{ on } \partial
\O\}
\end{multline*}
has at least one minimizer $(u,\sigma,p)$. Furthermore, $\sigma$ is unique and
belongs to $H^1_{\rm loc}(\O;\R^2)$.

Assume moreover that there exists a non empty convex open set $\O_p\subset\O$ such that $|\sigma|=1$ a.e. on $\O_p$. Then $\sigma$ is locally Lipschitz on $\O_p$ and remains constant along each open line segment
$$L_x:=(x+\R \sigma^\perp(x)) \cap \O_p, \quad \text{ with }x \in \O_p,$$
called  a {\rm characteristic} line segment. Moreover,  $u$ remains constant along $L_x \cap \O_p$ for all $x \in \O_p \setminus \left(\bigcup_{z \in Z} L_z\right)$ where $Z \subset \O_p$ is an $\HH^1$-negligible set such that $\LL^2\left(\bigcup_{z \in Z}L_z\cap \O_p\right)=0$.

Geometrically, $\O_p$ can be decomposed as the following disjoint union
$$\O_p=\bigcup_{i\in I} \mathbf F_i\cup \bigcup_{\lambda \in \Lambda} (L_{x_\lambda} \cap \O_p) \cup \bigcup_{j\in J}\mathbf C_j,$$
for some countable sets $I$ and $J$, and some (possibly) uncountable set $\Lambda$. For all $i \in I$,
$\mathbf F_i$ is a boundary fan, {\it i.e.},  the intersection of $\O_p$ with an open cone with apex $\bar z_i \in \partial\O_p$ and two characteristic line segments as generatrices; for all $\lambda \in \Lambda$, $L_{x_\lambda}$ is a characteristic line segment passing through $x_\lambda \in \O_p$ and we set $P_\lambda:=L_{x_\lambda}\cap\partial\O_p$ (a set made of two points); and for all $j \in J$,  $\mathbf C_j$ is a convex set,
closed in the relative topology of $\O_p$ and with non empty interior, endowed with one of the following two properties:
\begin{itemize}
\item Either $\partial \mathbf C_j= L_j \cup \Gamma_j$ with $L_j\subset
\O_p$ {an open characteristic} line segment  and $\Gamma_j$ a connected closed set in $\partial\O_p$. In that case, $\Gamma_j=\Gamma_j^1\cup\Gamma_j^2\cup {S_j}$ where $\Gamma_j^1$ and $\Gamma_j^2$ are connected  and
$S_j$ is {a closed line segment (possibly reduced to single point)} that separates $\Gamma_j^1$ and $\Gamma_j^2$. Further each point of $\Gamma_j^1$ (resp. $\Gamma_j^2$) is traversed by a characteristic line segment which will re-intersect $\partial\O_p$ on $\Gamma_j^2$ (resp. $\Gamma_j^1$);
\item Or  $\partial \mathbf C_j= L_j \cup L'_j\cup \Gamma_j\cup \Gamma'_j$ where $L_j$ and $L'_j\subset \O_p$ are open characteristic line segments, while $\Gamma_j$ and $\Gamma'_j$ are two disjoint connected closed sets in $\partial\O_p$. Further each point of $\Gamma_j$ (resp. $\Gamma'_j$) is traversed by a characteristic line segment which will re-intersect
$\partial\O_p$ on $\Gamma'_j$ (resp. $\Gamma_j$). In that case we set $S_j=\emptyset$.
\end{itemize}

Finally, $\sigma$ is continuous on {$\overline \O_p\setminus \big(\bigcup_{\lambda\in\Lambda}P_\lambda \cup \bigcup_{j \in J}S_j \cup \bigcup_{i \in I}\{\bar z_i\}\big)$}.
\end{theorem}

\begin{remark}\label{rem:caution}
Our results only pertain to the case $\partial\O_d=\partial\O$ and $\partial\O_n=\emptyset$. There are no obstacles in treating the more general case of a surface load $g$ on a Neumann part of the boundary $\partial\O_n$. This would simply add
the term  $-\int_{\partial\O_n} g u \,d\mathcal H^1$  in the minimization
problem \eqref{eq:minim} and the term $\int_{\partial\O_n}\varphi g(u-w)\, d\mathcal H^1$ in the definition \eqref{eq:duality} of the duality. However, one would have to spell out a so-called safe-load conditions on $g$ that guarantee existence, as well as add technical conditions on $\partial_{\lfloor \partial\O}\partial\O_d$ (the boundary of the Dirichlet part $\partial\O_d$ in $\partial\O$) (see \cite[Section 6]{FG}). Barring this, all results are local in nature and would not be affected.
With that caveat in mind, all our results, which are local, equally hold in the enlarged setting of a Neumann condition on part of the boundary of the domain.
\hfill\P\end{remark}

As already alluded to at the onset of the introduction, the reader uninterested in the particulars of Hencky plasticity may skip Sections \ref{sec:Hencky}, \ref{stressreg} without prejudice and view our contribution as an investigation of
the continuity equation \eqref{eq.cont.plas} under the assumptions \eqref{eq:eq.for.sigma} on $\sigma$, but with the additional knowledge of the existence of a nonnegative measure-solution $\mu$ such that ${\sigma}\mu=Du$ for some $u\in BV(\O)$.

\section{Notation and preliminaries}\label{sec:not}

The Lebesgue measure in $\R^n$ is denoted by $\mathcal L^n$  and the $s$-dimensional Hausdorff measure by $\mathcal H^s$.

 From here onward the space dimension is set to $2$. If $a$ and $b \in \R^2$, we write $a \cdot b$ for the Euclidean scalar product, and we denote
the norm by $|a|=\sqrt{a \cdot a}$. The open (resp. closed) ball of center $x$ and radius $\rho$ is denoted by $B_\rho(x)$ (resp. $\overline B_\rho(x)$).  If $K \subset \R^2$ is a closed and convex set, we denote by ${\mathscr N}_K(x)=\{\xi \in \R^2 : \; \xi \cdot (y-x) \leq 0 \text{ for
all }y \in {K}\}$ the normal cone to $K$ at $x \in \partial K$, and by ${\mathscr T}_K(x)=\{\zeta \in \R^2 : \; \xi \cdot \zeta \leq 0 \text{ for all }\xi \in {\mathscr N}_K(x)\}$ the tangential cone to $K$ at $x \in \partial K$. If $z=(z_1,z_2) \in \R^2$, we denote by $z^\perp=(-z_2,z_1)$ the rotation of $z$ of an angle $\pi/2$. Given two  vectors $u$ and $v \in \R^2$, we denote by $C(u,v):=\{\alpha u +\beta v : \; \alpha>0, \, \beta>0\}$ the open cone generated by $u$ and $v$ { with apex at the origin}. Also, if  $f:\R \to \overline \R$ is a proper, convex function, we denote by $\bpartial f(x)$ (or $\bpartial f(\cdot)(x)$) the subdifferential of $f$ at the point $x\in \R^2$ which is not empty at all points in the interior of its domain.

\medskip

In all that follows, $\Omega \subset \R^2$ is a bounded and Lipschitz open set. We use standard notation for Lebesgue and Sobolev spaces. We write
$\mathcal{M}(\Omega;\R^2)$ (resp. $\M(\O)$) for the space bounded Radon measures in $\Omega$ with values in $\R^2$ (resp. $\R$), endowed with the
norm $|\mu|(\Omega)$, where $|\mu|\in \mathcal{M}(\Omega)$ is the total variation of the measure $\mu$. The space $BV(\Omega)$ of functions of bounded variation in $\Omega$ is made of all functions $u \in L^1(\Omega)$ such that its distributional gradient $Du\in \mathcal M(\Omega;\R^2)$. Then $BV(\O)\subset L^2(\O)$.

\medskip

Given a map $\sigma:\R^2 \to \R^2$, we set  ${\rm div\,} \sigma:= \frac{\partial\sigma_1}{\partial x_1}+ \frac{\partial\sigma_2}{\partial x_2}$ and denote by ${\rm curl\,}\sigma$  the scalar quantity $\frac{\partial\sigma_2}{\partial x_1}-\frac{\partial\sigma_1}{\partial x_2}(=-\Div \sigma^\perp)$. We denote by $H(\Div,\Omega)$ the Hilbert space of all $\sigma \in L^2(\Omega;\R^2)$
such that $\Div \sigma \in L^2(\Omega)$. We recall that if $\O$ is bounded with Lipschitz boundary and $\sigma \in H(\Div,\Omega)$, its normal trace, denoted by $\sigma\cdot\nu$, is well defined as an element of $H^{-1/2}(\partial \Omega)$. If further $\sigma \in H(\Div,\Omega) \cap L^\infty(\Omega;\R^2)$, then $\sigma\cdot \nu \in L^\infty(\partial \Omega)$ with
$\|\sigma\cdot\nu\|_{L^\infty(\partial \O)}\leq \|\sigma\|_\infty$ (see {\it e.g.}~\cite[Theorem 1.2]{A}). Moreover, according to \cite[Theorem 2.2 (iii)]{CF}, if $\Omega$ is of class $\mathcal C^2$, then for all $\varphi \in L^1(\partial \Omega)$,
\begin{equation}\label{eq:sigmanu}
\lim_{\e \to 0}\int_0^1\int_{\partial \Omega} \big(\sigma(y-\e s \nu(y))\cdot \nu(y) - (\sigma\cdot \nu)(y) \big)\varphi(y)\, d\HH^1(y)\,ds=0,
\end{equation}
where $\nu$ denotes the outer unit normal to $\partial \Omega$.

According to \cite[Definition 1.4]{A}, we define {a generalized notion of duality pairing} between stresses and plastic strains as follows (see also \cite[Section 6]{FG}).
\begin{definition}
\label{def_Sigma_p}
Let $\sigma \in H(\Div,\Omega) \cap L^\infty(\Omega;\R^2)$, $(u,e,p) \in BV(\Omega) \times  L^2(\Omega;\R^2) \times \mathcal{M}(\overline \Omega;\R^2)$ and $w \in H^1(\O)$ be such that $D u = e+p$ in $\O$ and $p=(w-u)\nu \HH^{1}$ on $\partial \O$. We define the distribution $\left[\sigma\cdot p \right]\in \mathcal D'(\R^2)$ by
\begin{multline}\label{eq:duality}\langle\left[\sigma\cdot p \right], \varphi\rangle =  \int_\O \p\sigma\cdot (\nabla w-e)\; dx
+ \int_\O (w-u)\sigma \cdot\nabla\p\; dx\\
+\int_{\Omega} (w-u) (\Div \sigma) \varphi \, dx \quad \text{ for all } \varphi\in \mathcal{C}^{\infty}_c(\R^2).
\end{multline}
\end{definition}

If $\O$ has Lipschitz boundary, approximating $\sigma$ by smooth functions, and using the integration by parts formula in $BV$, one can show that $[\sigma\cdot p]$ is actually a bounded Radon measure supported in $\overline \Omega$ satisfying
\begin{equation}\label{eq:convex-ineq}
|[\sigma\cdot p]| \leq \|\sigma\|_\infty |p| \quad \text{ in } \mathcal M(\overline \Omega)
\end{equation}
and with total mass that obtained by taking $\p\equiv 1$ in \eqref{eq:duality}. Moreover, if $\sigma \in \mathcal C(\O;\R^2)$ we can show that
$$\langle\left[\sigma\cdot p \right], \varphi\rangle =  \int_\O \varphi
\sigma \cdot\, dp=\int_\O \varphi \sigma \cdot \frac{dp}{d|p|} \, d|p|\quad \text{ for all }\varphi\in \C_c(\O),$$
where $\frac{dp}{d|p|}$ stands for the Radon-Nikod\'ym derivative of $p$ with respect to its total variation $|p|$. For all of the above, see \cite[Section 6]{FG} in the vectorial case.

\section{Hencky plasticity}\label{sec:Hencky}

We recall in a few lines the basic tenet of Hencky plasticity, which is the ancestor of modern (small strain) plasticity. In Hencky plasticity (see {\it e.g.} \cite{T}), time is absent, so that the plasticity problem reduces to a problem of statics which can be tackled as a minimization problem. In this paper, we even go further and only address what is usually called the anti-plane shear strain case for which the displacement field (generally a vector-valued field in 3d) is  along the $x_3$-direction and only depends on the planar variables $(x_1,x_2)$.

\medskip

Let $\O$ be a bounded open set in $\R^2$ with Lipschitz boundary and $w \in H^1(\O)$ be a boundary data.
We consider the following minimization problem
\begin{equation}\label{eq:minim}
\inf \Big\{ \frac12 \int_\O |\sigma|^2\, dx + |p|(\overline \O) : \; (u,\sigma,p) \in \mathcal A_w\Big\}
\end{equation}
 with
\begin{multline*}\mathcal A_w:=\{(u,\sigma,p) \in BV(\O) \times L^2(\O;\R^2) \times \M(\overline \O;\R^2): \\
Du=\sigma+p \text{ in } \O, \ p=(w-u)\nu \HH^1\text{ on } \partial \O\}.
\end{multline*}
The direct method in the calculus of variations ensures the existence of minimizers $(u,\sigma,p) \in BV(\O) \times L^2(\O;\R^2) \times \M(\overline \O;\R^2)$. In addition, the stress $\sigma$ is unique and all minimizers $(u,\sigma,p)$ satisfy the following first order conditions
\begin{equation}
\label{eq:plast}
\begin{cases}
{\rm div}\sigma=0 \quad \text{ in }H^{-1}(\O), \\[1mm]
|\sigma|\leq 1 \quad \text{ a.e. in }\O,\\[1mm]
Du=\sigma+p \quad \text{ in }\M(\O;\R^2), \\[1mm]
p=(w-u)\nu \HH^1\text{ in } \M(\partial \O;\R^2),\\[1mm]
|p|=[\sigma\cdot p]  \quad \text{ in }\M(\overline \O).
\end{cases}
\end{equation}

The uniqueness of $\sigma$ is a immediate consequence of the strict convexity of $\sigma\mapsto\int_\O |\sigma|^2\, dx$.  We quickly check \eqref{eq:plast}. Performing variations of the form $(u,\sigma,p)+t(v,\tau,q)$ where $t\in (0,1)$ and $(v,\tau,q) \in \mathcal A_0$ and letting $t \to 0^+$ leads to
\begin{equation}\label{eq:EL}
|p|(\overline \O) \leq \int_\O \sigma\cdot\tau\, dx + |p+q|(\overline \O).
\end{equation}
Choosing first $(v,\tau,q)=\pm (\varphi,\nabla \varphi,0)$, with $\varphi \in \mathcal C^\infty_c(\O)$ arbitrary, as test function in \eqref{eq:EL} gives
$$\int_\O \sigma \cdot \nabla \varphi\, dx=0$$
hence ${\rm div}\sigma=0$  in $H^{-1}(\O)$.
Next, denoting by $p^a=\frac{dp}{d\LL^2}$ the Radon-Nikod\'ym derivative of $p$ with respect to $\LL^2$ and considering $(v,\tau,q)=(0,p^a-\eta,\eta-p^a)$ as test function in \eqref{eq:EL}, where $\eta \in L^1(\O;\R^2)$ is arbitrary, leads to
$$\int_\O |\eta|\, dx \geq \int_\O |p^a|\, dx + \int_\O \sigma\cdot (\eta-p^a)\, dx.$$
Localizing this inequality yields $\sigma(x) \in \bpartial |\cdot|(p^a(x))\subset\bpartial |\cdot|(0)$ for a.e. in $x \in \O$, hence
\begin{equation}\label{eq.str.const}|\sigma|\leq 1 \mbox{ a.e. in }\O.
\end{equation}
It remains to prove the ``flow rule". From \eqref{eq:convex-ineq} and  \eqref{eq.str.const}  the first inequality $|p| \geq [\sigma\cdot p]$ in $\mathcal M(\overline \O)$ holds. To prove the reverse inequality take $(v,\tau,q)=(w-u,\nabla w-\sigma,-p)$ as test function in \eqref{eq:EL}, and use the definition \eqref{eq:duality} of duality. This gives
$$|p|(\overline\O) \leq \int_\O \sigma\cdot (\nabla w-\sigma)\, dx=\langle [\sigma\cdot p],1\rangle.$$
Hence, the nonnegative measure $|p|-[\sigma\cdot p]$ has zero total mass which leads to the flow rule $|p| = [\sigma\cdot p]$ in $\mathcal M(\overline \O)$.

\begin{remark}\label{rmk:jump-flow-rule}
Exactly  as in \cite[Lemma 3.8]{FG},  if $\omega$ is an  open subset of $\O$  with Lipschitz boundary $\Gamma=\partial \omega$ and such that $\overline\omega\subset\O$, then $\sigma\cdot \nu \in L^\infty (\Gamma)$ and
$$[\sigma\cdot p]\res \Gamma=(\sigma\cdot\nu) (u^+-u^-) \HH^1 \res \Gamma,$$
where $u^+$ and $u^-$ are the outer and inner traces of $u$ on $\Gamma$ and $\sigma\cdot\nu$ is the normal trace of $\sigma$ on $\Gamma$. Thus, the flow rule localized on $\Gamma$ reads
$$(\sigma\cdot\nu) (u^+-u^-)=|u^+-u^-| \quad \HH^1\text{-a.e. on }\Gamma.$$
Since by definition $u^+ \neq u^-$ on $J_u$, we infer that $\sigma\cdot\nu=\pm 1$ $\HH^1$-a.e. on $\Gamma \cap J_u$. This applies also if $\HH^1(\partial\omega\cap\partial\O)>0$, replacing $u^+$ by $w$ on that part of
$\partial\omega$.
\hfill\P\end{remark}

\section{{Sobolev regularity of the stress}}\label{stressreg}

This section revolves around the Seregin/Bensoussan-Frehse's Sobolev regularity property of the (unique) stress field $\sigma$ in the minimization
problem \eqref{eq:minim} (or equivalently in the system \eqref{eq:plast}).

\begin{theorem}[Seregin/Bensoussan-Frehse regularity]\label{thm:Sobolev}
The unique stress $\sigma$ such that the triplet $(u,\sigma,p)\in BV(\O) \times L^2(\O;\R^2) \times \M(\overline \O;\R^2)$ is a minimizer of \eqref{eq:minim} belongs to $H^1_{\rm loc}(\O;\R^2)$.
\end{theorem}

\begin{proof}
\noindent{\sf Step 1.\;} In a first step we perform a so-called {Perzyna approximation} of the minimization problem \eqref{eq:minim}. We thus consider, for $\e>0$,
\begin{multline}\label{eq:perzyna}
\inf \Big\{\int_\O \Big( \frac12 |\sigma|^2 + |p|+ \frac{\e}{2} |p|^2\Big) \,dx : \; (u,\sigma,p) \in H^1(\O) \times L^2(\O;\R^2) \times L^2(\O;\R^2)  \\\mbox{ such that }
\nabla u=\sigma+p \text{ a.e. in } \O \text{ and } u=w \; \h\text{-a.e. on } \partial \O\Big\}.
\end{multline}
By strict convexity, there exists a unique minimizing triplet $(u_\e,\sigma_\e,p_\e)$.

Further, for a subsequence (still labeled by $\e$), it is straightforward
to show that
\begin{equation}
\label{eq:bds}
\begin{cases}
u_\e \rightharpoonup u \mbox{ weak-* in } BV(\O),\\
\sigma_\e \rightharpoonup \sigma \mbox{ weak in } L^2(\O;\R^2),\\
(\sqrt\e p_\e)_{\e>0} \mbox{ is bounded in } L^2(\O;\R^2),\\
p_\e \rightharpoonup p \mbox{ weakly* in } \M(\overline \O;\R^2),
\end{cases}
\end{equation}
where $(u,\sigma,p) \in BV(\O) \times L^2(\O;\R^2) \times \M(\overline \O;\R^2)$ is a minimizer of \eqref{eq:minim} as can be seen through direct application of an approximation result found in \cite[Theorem 3.5]{Mo}; see also \cite[Proposition 3.3]{BM}.

Now, testing minimality with triplets $(u_\e+tv,\sigma_\e+t\tau, p_\e+tq)$ with $t \in (0,1)$ and $(v,\tau,q) \in H^1_0(\O)\times L^2(\O;\R^2)\times L^2(\O;\R^2)$ where $\nabla v=\tau+q$, it is easily seen -- choosing
either $q\equiv 0$ and $\tau=\nabla v$ with $v \in H^1_0(\O)$ arbitrary, or $v\equiv 0$ and $\tau=-q{\bf 1}_E$ where $E \subset \O$ is measurable and $q \in \R^2$ arbitrary -- that
the minimizing triplet $(u_\e,\sigma_\e,p_\e)$ satisfies the following Euler-Lagrange equations:
\begin{equation}
\label{eq:ELeps}
\begin{cases}
\Div \sigma_\e =0\quad\text{ in }H^{-1}(\O),\\[1mm]
\nabla u_\e= \sigma_\e+p_\e \quad\mbox{ a.e. in }\O,\\[1mm]
u_\e=w \quad \h\text{-a.e. on }\partial \O,\\[1mm]
\sigma_\e-\e p_\e \in \partial H(p_\e) \quad \mbox{ a.e. in }\O,
\end{cases}
\end{equation}
where $H(q):=|q|$ for all $q\in \R^2$.
\begin{remark}
\label{rem:prop-subdif}
Observe that the fourth relation in \eqref{eq:ELeps} reads, by convex duality, as
$$
p_\e\in \bpartial I(\sigma_\e-\e p_\e)  \quad \mbox{ a.e. in }\O,
$$
where $I$ is the indicator function of the closed unit ball, {\it i.e.},

$$I(q):=
\begin{cases}
0 & \text{ if } |q|\le1,\\
+\infty & \text{ if } |q|>1.
\end{cases}$$
Thus
$\nabla u_\e-\e p_\e=p_\e+(\sigma_\e-\e p_\e)\in \bpartial \Psi(\sigma_\e-\e p_\e)
$ where $\Psi:\R^2 \to \R$ is defined by
$$\Psi(q):=\frac12 |q|^2+I(q) \quad \text{ for all } q \in \R^2,$$
or still
\begin{equation}
\label{eq:sudif-rel}
\sigma_\e-\e p_\e =D\Psi^*(\nabla u_\e-\e p_\e)
\end{equation}
where $\Psi^*:\R^2 \to \R$, the convex conjugate of $\Psi$, is given by
$$
\Psi^*(q)=
\begin{cases}
\frac12|q|^2 & \text{ if } |q|\le1,\\[1mm]
|q|-\frac12 & \text{ if } q|>1.
\end{cases}
$$
Remark that $\Psi^*\in \mathcal C^1(\R^2)$ with
$$D\Psi^*(q)=
\begin{cases}
q & \text{ if } |q|\le1,\\[1mm]
{q}/{|q|} & \text{ if }|q|>1,
\end{cases}$$
and that $D\Psi^*\in {\rm Lip}(\R^2) \cap \mathcal C^1(\R^2 \setminus \mathbb S^1)$ with
$$D^2\Psi^*(q)=
\begin{cases}
{\rm Id} & \text{ if }  |q|<1,\\[1mm]\ds
\frac{1}{|q|}{\rm Id}-\frac1{|q|^3} q\otimes q & \text{ if }|q|>1.
\end{cases}$$

Finally, the expression for $D^2\Psi^*$ implies that, for $|q|>1$ and for
all $r \in \R^2$,
$$
D^2\Psi^*(q)[r]= \frac1{|q|}P_{q^\perp}(r)
$$
where $P_{q^\perp}$ is the orthogonal projection onto the linear span of $q^\perp$.

Those projection properties, which are  specific to the Von-Mises criterion,  will be instrumental in the proof of the Sobolev regularity of the stress.
\hfill\P
\end{remark}

For $\e$ fixed, a usual translation argument  yields the classical local elliptic regularity
of the fields {(see {\it e.g.} \cite[Proposition 3.4]{BM} in the vectorial evolution case)}, that is
\begin{equation}
\label{eq:ell-reg}
u_\e\in H^2_{\rm loc}(\O), \quad  \sigma_\e,\; p_\e\in H^1_{\rm loc}(\O;\R^2)
\end{equation}
with corresponding $\e$-dependent bounds.

\medskip

\noindent{\sf Step 2.\;}
Let $k\in \{1,2\}$ and $\p\in \mathcal C^\infty_c(\O)$. Taking $\p^2 \partial_k u_\e$ (which belongs to $H^1_0(\O)$ thanks to \eqref{eq:ell-reg})
as test function for the equation
$$\Div (\partial_k  \sigma_\e)=0 \quad \text{ in }H^{-1}(\O),$$
we obtain
$$
0=\int_\O \p^2 \partial_k \sigma_\e \cdot \partial_k \nabla u_\e\, dx+
\int_\O\partial_k \sigma_\e\cdot \nabla \p^2  \partial_k u_\e\, dx=:I_1+I_2.
$$
{In the sequel, $C_\p$ will stand for a positive constant which may vary from line to line; it may depend on $\p$ and on the bounds coming from \eqref{eq:bds}, but it is independent of $\e$.}

We now rewrite $I_1$ and $I_2$ as follows.
\begin{equation}
\label{eq:I1}
I_1=\int_\O \p^2 |\partial_k \sigma_\e|^2\, dx+\int_\O \p^2\left(\partial_k \sigma_\e-\e\partial_k p_\e\right)\cdot \partial_k p_\e\, dx+\e\int_\O \p^2|\partial_k p_\e|^2\, dx,
\end{equation}
while
\begin{eqnarray}
\label{eq:I2}
|I_2| & \le & 2\int_\O |\partial_k \sigma_\e\cdot\nabla\p|\; |(\sigma_\e)_k \p|\, dx+
\int_\O |\partial_k \sigma_\e| |\nabla \p^2| |(p_\e)_k|\, dx\nonumber\\
& \leq & C_\p \|\sigma_\e\|_{L^2(\O)}\|\p\partial_k \sigma_\e\|_{L^2(\O)}+\int_\O |\partial_k \sigma_\e| |\nabla \p^2| |p_\e|\, dx.
\end{eqnarray}
Using Young's inequality and reassembling \eqref{eq:I1}, \eqref{eq:I2} we
thus get, in view of the bound coming from the second convergence in \eqref{eq:bds},
\begin{multline}
\label{eq:ineq1}
\frac12 \int_\O \p^2 |\partial_k \sigma_\e|^2\, dx+\e\int_\O \p^2|\partial_k p_\e|^2\, dx+\int_\O \p^2\left(\partial_k \sigma_\e-\e\partial_k p_\e\right)\cdot \partial_k p_\e\, dx\\
\leq C_\p +\int_\O |\partial_k \sigma_\e| |\nabla \p^2| |p_\e|\, dx,
\end{multline}
or still, adding and subtracting $\partial_k \sigma_\e-\e \partial_k p_\e$ to $\partial_k p_\e$ in the third integral on the left hand-side of \eqref{eq:ineq1} above and recalling that $\nabla u_\e=\sigma_\e+p_\e$,
\begin{multline}
\label{eq:ineq2}
\frac12 \int_\O \p^2 |\partial_k \sigma_\e|^2\, dx+\e\int_\O \p^2|\partial_k p_\e|^2\, dx\\
+\int_\O \p^2(\partial_k \sigma_\e-\e\partial_k p_\e)\cdot (\partial_k \nabla u_\e-\e \partial_k p_\e)\, dx\\
\leq C_\p+\int_\O \p^2|\partial_k \sigma_\e-\e\partial_k p_\e|^2\, dx+
\int_\O |\partial_k \sigma_\e| |\nabla \p^2| |p_\e|\, dx.
\end{multline}

We now rewrite \eqref{eq:ineq2}  using Remark \ref{rem:prop-subdif}. Since, in view of \eqref{eq:ell-reg},  {$\nabla u_\e-\e p_\e \in H^1_{\rm loc}(\O;\R^2)$, we can apply the generalized chain rule formula from \cite[Theorem 2.1]{MT} to \eqref{eq:sudif-rel}.} We obtain
\begin{equation}
\label{eq:cr}
\partial_k \sigma_\e-\e\partial_k p_\e= D^2\Psi^*(\nabla u_\e-\e p_\e)\left[\partial_k \nabla u_\e-\e \partial_k p_\e\right].
\end{equation}

\begin{remark}\label{rem:mean-lip}
Relation \eqref{eq:cr} has to be understood as follows (see \cite{MT}):
\begin{equation}\label{eq:chainrule}
\partial_k \sigma_\e-\e\partial_k p_\e=
\begin{cases}
\ds \partial_k \nabla u_\e-\e \partial_k p_\e  \quad\text{ if } |\nabla u_\e-\e p_\e| \leq 1,\\
\ds \frac{1}{|\nabla u_\e-\e p_\e|}P_{(\nabla u_\e-\e p_\e)^\perp}(\partial_k \nabla u_\e-\e \partial_k p_\e) \quad \text{ else.}
\end{cases}
\end{equation}
Note that application of \cite[Proposition 2.2]{MT}  to the $\mathcal C^1$-Lipschitz functions
$q\mapsto |q|^2 \mbox{ and } q\mapsto 2|q|-1$ yields
\begin{multline*}\partial_k \nabla u_\e-\e \partial_k p_\e=\frac{1}{|\nabla u_\e-\e p_\e|}P_{(\nabla u_\e-\e p_\e)^\perp}(\partial_k \nabla u_\e-\e \partial_k p_\e) \\ \text{ a.e. on }\{|\nabla u_\e-\e p_\e| = 1\},
\end{multline*}
which implies that \eqref{eq:chainrule} can be changed into
$$\partial_k \sigma_\e-\e\partial_k p_\e=
\begin{cases}
\ds \partial_k \nabla u_\e-\e \partial_k p_\e \quad \text{ if } |\nabla u_\e-\e p_\e| < 1,\\
\ds \frac{1}{|\nabla u_\e-\e p_\e|}P_{(\nabla u_\e-\e p_\e)^\perp}(\partial_k \nabla u_\e-\e \partial_k p_\e) \quad \text{ else.}
\end{cases}$$
\hfill\P
\end{remark}

In view of \eqref{eq:cr}, \eqref{eq:ineq2} reads as
\begin{multline}
\label{eq:ineq3}
\frac12 \int_\O \p^2 |\partial_k \sigma_\e|^2\, dx+\e\int_\O \p^2|\partial_k p_\e|^2\, dx
+ \\\int_\O \p^2D^2\Psi^*(\nabla u_\e-\e p_\e)[\partial_k \nabla u_\e-\e \partial_k p_\e]\cdot (\partial_k \nabla u_\e-\e \partial_k p_\e)\, dx\\
\leq C_\p+\int_\O \p^2|D^2\Psi^*(\nabla u_\e-\e p_\e)[\partial_k \nabla u_\e-\e \partial_k p_\e]|^2\, dx\\
+\int_\O |\partial_k \sigma_\e| |\nabla \p^2| |p_\e|\, dx.
\end{multline}

\medskip

\noindent{\sf Step 3.\;} We next exploit inequality \eqref{eq:ineq3} obtained at the end of Step 2 by splitting $\O$ into three $\e$-dependent subsets as follows,
\begin{eqnarray*}
\O_\e^-&:=&\{x\in \O: \; |\nabla u_\e(x)-\e  p_\e(x)|\le 1\},\\[2mm]
\O^+_{\e>}&:=&\{x\in \O:\; |\nabla u_\e(x)-\e p_\e(x)|> 2\},\\[2mm]
\O^+_{\e<}&:=&\{x\in \O: \; 1<|\nabla u_\e(x)-\e p_\e(x)|\le 2\},\\[2mm]
\O_\e^+&:=& \O^+_{\e<}\cup\O^+_{\e>}.
\end{eqnarray*}
First note that, on $\O_\e^-$, \eqref{eq:chainrule} in Remark \ref{rem:mean-lip} implies that
$$
\partial_k\sigma_\e-\e\partial_k p_\e= D^2\Psi^*(\nabla u_\e-\e p_\e)\left[\partial_k \nabla u_\e-\e \partial_k p_\e\right]=\partial_k \nabla
u_\e-\e\partial_k p_\e.
$$ Consequently, the contributions of the integrals involving the term $D^2\Psi^*(\nabla u_\e-\e  p_\e)$ cancel out on that set in \eqref{eq:ineq3}. Further, writing $p_\e= (\nabla u_\e-\e p_\e)-\sigma_\e+\e p_\e$ and  by virtue of the bounds coming from the second and third convergences in \eqref{eq:bds}, we have that,
$$
\|p_\e\|_{L^2(\O_\e^-\cup \O^+_{\e<})}\le  C
$$
with $C>0$ a constant independent of $\e$ (and actually also independent of $\p$). Using once again Young's inequality we can thus absorb the contribution of the term $\int_{\O^-_\e\cup \O^+_{\e<}} \left|\partial_k\sigma_\e\right| |\nabla \p^2| |p_\e|\, dx$ in  the term $\frac12\int_\O \p^2 |\partial_k\sigma_\e|^2\, dx$ in \eqref{eq:ineq3}  at the possible expense of changing $C_\p$. In lieu of \eqref{eq:ineq3}, we are thus left  with
\begin{multline*}
\frac14\int_\O \p^2 |\partial_k\sigma_\e|^2\, dx+\e\int_\O \p^2|\partial_k p_\e|^2\, dx\\
 +\int_{\O^+_\e} \p^2D^2\Psi^*(\nabla u_\e-\e  p_\e)[\partial_k \nabla u_\e-\e\partial_k p_\e]\cdot (\partial_k \nabla u_\e-\e\partial_k p_\e)\, dx \\
\leq C_\p +\int_{\O^+_\e} \p^2 |D^2\Psi^*(\nabla u_\e-\e p_\e)[\partial_k \nabla u_\e-\e\partial_k p_\e]|^2\, dx\\
+\int_{\O^+_{\e>}} |\partial_k\sigma_\e| |\nabla \p^2| |p_\e|\, dx,
\end{multline*}
or still, in view of  \eqref{eq:cr}, with
\begin{multline*}
\frac14\int_\O \p^2 |\partial_k\sigma_\e|^2\, dx+\e\int_\O \p^2|\partial_k p_\e|^2\, dx\\
+\int_{\O^+_\e} \p^2D^2\Psi^*(\nabla u_\e-\e  p_\e)[\partial_k \nabla u_\e-\e\partial_k p_\e]\cdot (\partial_k \nabla u_\e-\e\partial_k p_\e)\, dx\\
\leq C_\p +\int_{\O^+_\e} \p^2 |D^2\Psi^*(\nabla u_\e-\e p_\e)[\partial_k \nabla u_\e-\e\partial_k p_\e]|^2\, dx\\
+\int_{\O^+_{\e>}}\!\! |D^2\Psi^*(\nabla u_\e-\e p_\e)[\partial_k \nabla
u_\e-\e\partial_k p_\e]| |\nabla \p^2| |p_\e|\, dx\\
+\e\!\int_{\O^+_{\e>}}\!\! |\partial_k p_\e||\nabla \p^2| |p_\e|\, dx.
\end{multline*}

Because of  the third relation in \eqref{eq:bds} and of Young's inequality, the third integral in the right hand-side of the last inequality above
can be controled by the term $\e\int_\O \p^2|\partial_k p_\e|^2\, dx$ in the left hand-side at the expense of changing $C_\p$, {\it i.e.} ,
\begin{multline}
\label{eq:ineq4}
\frac14\int_\O \p^2 |\partial_k\sigma_\e|^2\, dx+\frac{\e}{2} \int_\O \p^2|\partial_k p_\e|^2\, dx\\
+\!\int_{\O^+_\e} \p^2D^2\Psi^*(\nabla u_\e-\e  p_\e)[\partial_k \nabla u_\e-\e\partial_k p_\e]\cdot (\partial_k \nabla u_\e-\e\partial_k p_\e)\; dx\\
\leq C_\p+\int_{\O^+_\e} \p^2 |D^2\Psi^*(\nabla u_\e-\e p_\e)[\partial_k
\nabla u_\e-\e\partial_k p_\e]|^2 \; dx\\
+\int_{\O^+_{\e>}} |D^2\Psi^*(\nabla u_\e-\e  p_\e)[\partial_k \nabla u_\e-\e\partial_k p_\e]| |\nabla \p^2| \big\{|\nabla u_\e-\e p_\e|+|\sigma_\e+\e p_\e|\big\}\, dx,
\end{multline}
where we used $|p_\e|\le|\nabla u_\e-\e p_\e|+|\sigma_\e+\e p_\e|$ in the
last term of the right hand side.

Recalling Remark \ref{rem:prop-subdif} and noting that $P_{q^\perp}(\xi)\cdot \xi=|P_{p^\perp}(\xi)|^2$, \eqref{eq:ineq4} now reads as
\begin{multline*}
\frac14\int_\O \p^2 |\partial_k\sigma_\e|^2\, dx+\frac{\e}{2}\int_\O \p^2|\partial_k p_\e|^2\, dx
\\+ \int_{\O^+_\e} \frac{\p^2}{|\nabla u_\e-\e  p_\e|}|P_{(\nabla u_\e-\e
 p_\e)^\perp}[\partial_k \nabla u_\e-\e\partial_k p_\e]|^2\, dx\\
\leq C_\p+\int_{\O^+_\e} \frac{\p^2}{|\nabla u_\e-\e p_\e|^2}|P_{(\nabla
u_\e-\e  p_\e)^\perp}[\partial_k \nabla u_\e-\e\partial_k p_\e]|^2\, dx\\
+\int_{\O^+_{\e>}} \!\!|P_{(\nabla u_\e-\e p_\e)^\perp}[\partial_k \nabla u_\e-\e\partial_k p_\e]| |\nabla \p^2| \, dx \\
+\int_{\O^+_{\e>}} \frac{|\nabla \p^2|}{|\nabla u_\e-\e p_\e|}|P_{(\nabla u_\e-\e  p_\e)^\perp}[\partial_k \nabla u_\e-\e\partial_k p_\e]| |\sigma_\e+\e p_\e|\, dx.
\end{multline*}
In view of the bounds coming from the second and third relations in \eqref{eq:bds}, application of the Cauchy-Schwarz and Young inequalities to the last term in the inequality above yield, at the possible expense of changing $C_\p$,
\begin{multline}
\label{eq:ineq5}
\frac14\int_\O \p^2 |\partial_k\sigma_\e|^2\, dx+\frac{\e}{2}\int_\O \p^2|\partial_k p_\e|^2\, dx\\
+ \int_{\O^+_\e} \frac{\p^2}{|\nabla u_\e-\e  p_\e|}|P_{(\nabla u_\e-\e  p_\e)^\perp}[\partial_k \nabla u_\e-\e\partial_k p_\e]|^2\, dx\\
\leq C_\p +\int_{\O^+_{\e<}} \frac{\p^2}{|\nabla u_\e-\e p_\e|^2}|P_{(\nabla u_\e-\e  p_\e)^\perp}[\partial_k \nabla u_\e-\e\partial_k p_\e]|^2\,
dx\\
+\int_{\O^+_{\e>}} |P_{(\nabla u_\e-\e  p_\e)^\perp}[\partial_k \nabla u_\e-\e\partial_k p_\e]| |\nabla \p^2| \, dx\\
+\frac32\int_{\O^+_{\e>}} \frac{\p^2}{|\nabla u_\e-\e p_\e|^2}|P_{(\nabla u_\e-\e  p_\e)^\perp}[\partial_k \nabla u_\e-\e\partial_k p_\e]|^2\, dx.
\end{multline}
The second integral in the right hand-side of  inequality \eqref{eq:ineq5} can in turn be estimated as follows with the help, once more, of the Cauchy-Schwarz inequality,
\begin{multline*}
\int_{\O^+_{\e>}} |P_{(\nabla u_\e-\e  p_\e)^\perp}[\partial_k \nabla u_\e-\e\partial_k p_\e]| |\nabla \p^2| \, dx \\
\le \|\nabla u_\e -\e p_\e\|^{1/2}_{L^1(\O)}
\left\{\int_{\O^+_{\e>}}\frac{|\nabla \p^2|^2}{|\nabla u_\e-\e p_\e|}|P_{(\nabla u_\e-\e  p_\e)^\perp}[\partial_k \nabla u_\e-\e\partial_k p_\e]|^2  \, dx\right\}^{1/2}\\
\le C'_\p\left\{\int_{\O^+_{\e>}}\frac{\p^2}{|\nabla u_\e-\e p_\e|}|P_{(\nabla u_\e-\e  p_\e)^\perp}[\partial_k \nabla u_\e-\e\partial_k p_\e]|^2
 \, dx\right\}^{1/2},
\end{multline*}
where the last inequality holds because $\|\nabla u_\e -\e p_\e\|_{L^1(\O)}\le C$ in view of the bounds coming from \eqref{eq:bds}, and $C'_\p>0$ is another $\p$-dependent constant.

Thus inequality \eqref{eq:ineq5} becomes
\begin{multline}
\label{eq:ineq9}
\frac14\int_\O \p^2 |\partial_k\sigma_\e|^2\, dx+\frac{\e}{2}\int_\O \p^2|\partial_k p_\e|^2\, dx\\
+ \int_{\O^+_{\e>}} \frac{\p^2}{|\nabla u_\e-\e p_\e|}|P_{(\nabla u_\e-\e  p_\e)^\perp}[\partial_k \nabla u_\e-\e\partial_k p_\e]|^2\, dx\\
+ \int_{\O^+_{\e<}} \frac{\p^2}{|\nabla u_\e-\e p_\e|}|P_{(\nabla u_\e-\e  p_\e)^\perp}[\partial_k \nabla u_\e-\e\partial_k p_\e]|^2\, dx\\
\leq C_\p+ \frac32 \int_{\O^+_{\e>}} \frac{\p^2}{|\nabla u_\e-\e p_\e|^2}|P_{(\nabla u_\e-\e  p_\e)^\perp}[\partial_k \nabla u_\e-\e\partial_k p_\e]|^2\, dx\\
+\int_{\O^+_{\e<}}\frac{\p^2}{|\nabla u_\e-\e p_\e|^2}|P_{(\nabla u_\e-\e  p_\e)^\perp}[\partial_k \nabla u_\e-\e\partial_k p_\e]|^2\, dx\\
+C'_\p\left\{\int_{\O^+_{\e>}}\frac{\p^2}{|\nabla u_\e-\e p_\e|}|P_{(\nabla u_\e-\e  p_\e)^\perp}[\partial_k \nabla u_\e-\e\partial_k p_\e]|^2\, dx\right\}^{1/2}.
\end{multline}
But $|\nabla u_\e-\e  p_\e|>2$ on $\O^+_{\e>}$, while $|\nabla u_\e-\e  p_\e|^2\ge |\nabla u_\e-\e  p_\e|$  on $\O^+_{\e<}$,  so that, finally, \eqref{eq:ineq9} implies that
\begin{multline}
\label{eq:ineq10}
\frac14\int_\O \p^2 |\partial_k\sigma_\e|^2\, dx+\frac{\e}{2}\int_\O \p^2|\partial_k p_\e|^2\, dx\\
+ {\frac14} \int_{\O^+_{\e>}} \frac{\p^2}{|\nabla u_\e-\e p_\e|}|P_{(\nabla u_\e-\e  p_\e)^\perp}[\partial_k \nabla u_\e-\e\partial_k p_\e]|^2\, dx\\
\leq C_\p+ C'_\p\left\{\int_{\O^+_{\e>}}\frac{\p^2}{|\nabla u_\e-\e  p_\e|}|P_{(\nabla u_\e-\e  p_\e)^\perp}[\partial_k \nabla u_\e-\e\partial_k p_\e]|^2  \, dx\right\}^{\frac12}.
\end{multline}

From a last application of Young's inequality, it is immediately deduced from \eqref{eq:ineq10} that the sequence $(\partial_k \sigma_\e)_{\e>0}$ is bounded in $H^1_{\rm loc}(\O;\R^2)$ independently of $\e$, which implies the desired regularity for $\partial_k \sigma$. The proof of Theorem
\ref{thm:Sobolev} is complete.
\end{proof}

\begin{remark}
\label{rem:h1pp}
Since $\sigma \in H^1_{\rm loc}(\O;\R^2)$,  it admits a precise representative defined $\rm{Cap }_p$-quasi everywhere for any $p<2$ hence $\HH^s$-almost everywhere in $\O$ for any $s>0$ {(see {\it e.g.} \cite[Sections 4.7, 4.8]{EG}).} In particular, $\h$-almost every point in $\O$ is a Lebesgue point for $\sigma$ and satisfies
$$\lim_{\rho \to 0}\frac{1}{\rho^2} \int_{B_\rho(x_0)}|\sigma(y)-\sigma(x)|^2\, dy =0 \quad \text{for $\h$-a.e. $x \in \O$.}$$
In the sequel we will identify $\sigma$ to be its precise representative which is thus defined $\h$-almost everywhere in $\O$ (it is actually defined outside a set of zero Hausdorff dimension).\hfill\P
\end{remark}

\begin{remark}\label{rmk:strong-fr}
Arguing as in \cite{A2,DMDSM,FGM,BM}, it is possible to express the flow rule by means of the quasi-continuous representative of the stress, still
denoted by $\sigma$, which is $|p|$-measurable, in a pointwise sense:
$$\sigma (x) \cdot \frac{dp}{d|p|}(x)=1 \quad \text{for $|p|$-a.e. $x\in \O$}$$
or still
$$p=\sigma |p| \quad \text{ in }\M(\O).$$
In particular, the measure $|p|$ is concentrated in the plastic part of the domain, {\it i.e.},  $|p|(\{x \in \O : \; |\sigma(x)|<1\})=0$.

Also, the additive decomposition of $Du$ implies that
\begin{equation}\label{eq:DU=}
Du=\sigma+p=\sigma + \sigma|p|= \sigma(\mathcal L^2+|p|) \quad \text{ in } \M(\O).
\end{equation}
\hfill\P
\end{remark}

\begin{remark}\label{rem:Leb-pts}
Let $S$ be a segment such that $\overline S \subset \O$, then $\sigma \in
H^{1/2}(S;\R^2)$. In addition, if $x_0 \in S$ is a Lebesgue point of $\sigma$, then
$$\lim_{\rho \to 0}\frac{1}{\rho} \int_{B_\rho(x) \cap S}|\sigma(y)-\sigma(x_0)|^2\, d\h(y) =0.$$
In other words, $x_0$ is also a Lebesgue point for the trace of $\sigma$ on $S$.

\medskip

Indeed, the trace theorem in Sobolev spaces states that $\sigma$ has a trace on $S$, denoted by $\sigma_{|_{S}}$, which belongs to $H^{1/2}(S;\R^2)$. For simplicity, we will assume that $S=(-1,1)\times \{0\}$. Arguing
by approximation, we first observe that, for all $x \in S$,
\begin{equation}\label{eq:fondcalc}
\frac{1}{4\rho^2}\int_{x+(-\rho,\rho)^2} |\sigma(y)-\sigma_{|_{S}}(y_1,0)|^2\, dy \leq C \int_{x+(-\rho,\rho)^2} |\nabla \sigma(y)|^2\, dy \to 0
\end{equation}
as $\rho \to 0$, since $\nabla \sigma \in L^2_{\rm loc}(\O;\mathbb M^{2 \times 2})$.

{Let $x \in S$ be a Lebesgue point of $\sigma$} (as an element of $L^2_{\LL^2}(\O;\R^2)$), {\it i.e.},
\begin{equation}\label{eq:Lebesguepoint}
\lim_{\rho \to 0}\frac{1}{\rho^2} \int_{x+(-\rho,\rho)^2}|\sigma(y)-\sigma(x)|^2\, dy =0
\end{equation}
as well as a Lebesgue point of $\sigma_{|_{S}}$ (as an element of $L^2_{\h}(S;\R^2)$), {\it i.e.},
\begin{equation}\label{eq:Lebesguepoint3}
\lim_{\rho \to 0}\frac{1}{\rho} \int_{B_\rho(x) \cap S}|\sigma_{|_{S}}(y)-\sigma_{|_{S}}(x)|^2\, d\h(y) =0.
\end{equation}
Observe that $\h$-almost every point $x$ in $S$ satisfy these properties.
As a consequence of \eqref{eq:fondcalc} and \eqref{eq:Lebesguepoint}, we
have
\begin{multline*}
\frac{1}{\rho} \int_{B_\rho(x) \cap S}|\sigma_{|_{S}}(y)-\sigma(x)|^2\, d\h(y) =\frac{1}{\rho} \int_{-\rho}^\rho|\sigma_{|_{S}}(y_1,0)-\sigma(x)|^2\, dy_1 \\
\leq \frac{1}{\rho^2} \int_{x+(-\rho,\rho)^2} |\sigma(y)-\sigma_{|_{S}}(y_1,0)|^2\, dy + \frac{1}{\rho^2}\int_{x+(-\rho,\rho)^2}|\sigma(y)-\sigma(x)|^2\, dy \to 0.
\end{multline*}
Using next \eqref{eq:Lebesguepoint3}, we infer that $\sigma(x)=\sigma_{|_{S}}(x)$ which shows that $\sigma=\sigma_{|_{S}}$ $\h$-a.e. on $S$, and that $\sigma \in H^{1/2}(S;\R^2)$.

If now $x_0 \in S$ is only a Lebesgue of $\sigma$, then by \eqref{eq:fondcalc} (which holds for all $x_0 \in S$) and \eqref{eq:Lebesguepoint}, we have similarly
\begin{multline*}
\frac{1}{\rho} \int_{B_\rho(x_0) \cap S}|\sigma(y)-\sigma(x_0)|^2\, d\h(y) =\frac{1}{\rho} \int_{-\rho}^\rho|\sigma(y_1,0)-\sigma(x_0)|^2\, dy_1\\
\leq \frac{1}{\rho^2} \int_{x_0+(-\rho,\rho)^2} |\sigma(y)-\sigma(y_1,0)|^2\, dy + \frac{1}{\rho^2}\int_{x_0+(-\rho,\rho)^2}|\sigma(y)-\sigma(x_0)|^2\, dy \to 0,
\end{multline*}
which completes the proof of the result.
\hfill\P
\end{remark}

\section{{Rigidity properties of the solutions}}\label{sec:rigidity}

The goal of this section is to take advantage of the hyperbolic equations
satisfied by $\sigma$ and $|p|$ in the plastic zone $\{x \in \O : \; |\sigma(x)|=1\}$ in order to derive rigidity properties  of the solutions $\sigma$ and $u$ in that zone. The equations are
\begin{equation}\label{eq:hyperbolicsigma}
{\rm div}\sigma=0, \quad |\sigma|= 1,
\end{equation}
and
\begin{equation}\label{eq:hyperbolicu}
{\rm curl\;} (\sigma(1+|p|))={\rm curl\;} Du= 0.
\end{equation}
We will need the  the following

\medskip

\noindent {\bf Hypothesis (H).} \emph {The set  $\{x \in \O : \; |\sigma(x)|=1\}$ has a nonempty interior. We denote by  $\O_p$    a convex open subset  of that interior.}

\medskip

 Note, for future use, that, in such a setting $\partial\Omega_p$ has Lipschitz boundary (see {\it e.g.} Propositions 2.4.4 and Proposition 2.4.7 in \cite{HP}).

As already discussed in the Introduction, hypothesis (H) is fulfilled in several vectorial examples. It could also be the case in particular  in simple scalar settings, for example when $\O$ itself is a (countable union
of) boundary fans (see Subsection \ref{bdary-fans}).

\subsection{Lipschitz regularity and rigidity of the stress}\label{stressLip}

In this subsection we improve the Sobolev regularity of the stress field in the plastic region. We show that under assumption (H), the stress is actually locally Lipschitz continuous in $\O_p$, and that it is constant along all the characteristic lines
\begin{equation}\label{eq:defL}
L_x:=x+\R \sigma^\perp(x)
\end{equation}
for all $x \in \O_p$,  associated with the hyperbolic conservation law \eqref{eq:hyperbolicsigma} solved by $\sigma$ in $\O_p$. We adopt henceforth the following

\medskip
\noindent {\bf Notation.}
\emph {$L_x$ is the characteristic line that passes through $x$ defined as \eqref{eq:defL}, this for all $x\in\O_p$.}

\medskip

Per hypothesis  (H), the system
$$\begin{cases}
\sigma \in H^1_{\rm loc}(\O;\R^2),\\
{\rm div}\sigma=0, \quad |\sigma|= 1 \quad \text{ a.e. in }\O_p
\end{cases}$$
possesses a solution. The main result of this section is the following
\begin{theorem}[Jabin-Otto-Perthame regularity]\label{thm:lip}
Assume that hypothesis  (H) holds. Let $\omega$ be a bounded and convex open set such that $\overline \omega \subset \O_p$ and $d:=\dist(\omega,\partial \O_p)>0$. Then, for every Lebesgue points $x_0$ and $y_0 \in \omega$ of $\sigma$,
$$|\sigma(x_0)- \sigma(y_0)|\leq \frac1d |x_0-y_0|.$$
In particular, $\sigma$ admits a representative, still denoted by $\sigma$, which is locally Lipschitz in $\O_p$. Moreover, $\sigma$ is constant along all characteristic lines $L_{x_0}\cap \O_p$, for all $x_0 \in \O_p$.
\end{theorem}

The result was explicitly stated in \cite[Theorem 1]{I}. The proof there was  indirect: it used the rigidity result of \cite{JOP} for general $L^\infty$-solutions to the system ${\rm div}\sigma=0$, $|\sigma|= 1$. In
the sequel, we give a more direct proof of that result by exploiting from
the get-go the {\it a priori} knowledge of the $H^1_{\rm loc}(\O;\R^2)$-regularity for $\sigma$. Doing so simplifies several arguments, for example  the existence of traces on lines which becomes a straightforward consequence of the trace theorem in Sobolev spaces). The fundamentals of the proof are unchanged. As in \cite{I,JOP,BP}, the argument rests on the notion of entropies introduced in \cite{DKMO} and on the convexity of the domain, so that we do not claim  originality in this Subsection.

\begin{definition}
A function $\Phi \in \mathcal C^\infty_c(\R^2;\R^2)$ is called an entropy
if, for all $z \in \R^2$,
$$z\cdot [D\Phi(z)z^\perp]=0, \quad \Phi(0)=0, \quad D\Phi(0)=0.$$
\end{definition}

According to \cite[Lemma 2.3]{DKMO}, for all entropies $\Phi$, we have $\Div[\Phi(\sigma)]=0$ a.e. in $\O_p$, or still,
\begin{equation}\label{eq:entropy}
\int_{\O_p} \Phi(\sigma)\cdot \nabla \varphi\, dx =0
\end{equation}
for all test functions $\varphi \in \mathcal C^\infty_c(\O_p)$.
Following  \cite{DKMO}, we introduce the following family of {parameterized} generalized entropies:
$$\Phi^{(\xi)}(z):=
\begin{cases}
|z|^2\xi & \text{ if }z \cdot \xi>0,\\[1mm]
0 & \text{ if }z \cdot \xi\leq 0,
\end{cases}$$
where $\xi \in \mathbb S^1$. According to \cite[Lemma 2.5]{DKMO}, there exists a sequence $(\Phi_n)_{n \in \N}$ of entropies in $\mathcal C^\infty_c(\R^2;\R^2)$ which is locally uniformly bounded, and such that $\Phi_n \to \Phi^{(\xi)}$ pointwise in $\R^2$.
Specializing \eqref{eq:entropy} to $\Phi_n$ and passing to the limit as $n \to +\infty$ with the help of Lebesgue's dominated convergence theorem,

$$\int_{\O_p} \Phi^{(\xi)}(\sigma)\cdot \nabla \varphi\, dx =0$$
for all $\xi\in \mathbb S^1$ and all $\p\in \mathcal C^\infty_c(\O_p)$.
Now following \cite{JOP}, we introduce, for a.e. $x \in \O_p$ and for any
$\xi \in \mathbb S^1$,
$$\chi(x,\xi)=
\begin{cases}
1 & \text{ if }\sigma(x)\cdot \xi>0,\\
0 & \text{ if }\sigma(x)\cdot \xi\leq 0.\\
\end{cases}$$
The considerations above establish that
$$\Div[\xi \chi(\cdot,\xi)]=0 \quad{ \text{ in }\mathcal D'(\O_p)}$$
for all $\xi \in \mathbb S^1$.
This can be rewritten as a so-called {\it kinetic formulation} as follows:
$$D_\xi \chi(\cdot,\xi)=0 \quad{ \text{ in }\mathcal D'(\O_p)}.$$

The previous kinetic formulation entails the ordering property below whose proof  is a direct adaptation of \cite[Proposition 3.1]{JOP} (see also \cite[Corollary 4.7]{BP}).

\begin{proposition}\label{prop:order}
Assume that hypothesis  (H) holds. Let $x_0$ and $y_0 \in \O_p$ be two Lebesgue points of $\sigma$. Then
$$\sigma(x_0)\cdot(y_0-x_0) >0 \quad \Longrightarrow\quad \sigma(y_0)\cdot(y_0-x_0)\geq 0,$$
and
$$\sigma(x_0)\cdot(y_0-x_0) <0 \quad \Longrightarrow\quad \sigma(y_0)\cdot(y_0-x_0) \leq 0.$$
\end{proposition}

\begin{proof}
We only prove the first implication. Let us set $\xi=\frac{y_0-x_0}{|y_0-x_0|}$.  Then
\begin{multline*}
\limsup_{\rho\to 0}\frac{\LL^2(\{\sigma\cdot\xi\le 0\}\cap B_\rho(x_0))}{\pi\rho^2}\sigma(x_0)\cdot\xi  \leq  \\ \limsup_{\rho\to 0} \frac{1}{\pi\rho^2}\int_{\{\sigma\cdot\xi\le 0\}\cap B_\rho(x_0)} [\sigma(x_0)-\sigma(z)]\cdot \xi \, dz
\\ \le  \limsup_{\rho\to 0} \frac{1}{\pi\rho^2}\int_{B_\rho(x_0)}|\sigma(x_0)-\sigma(z)|\, dz=0.
\end{multline*}
Thus, since $\sigma(x_0)\cdot\xi>0$,
$$
\lim_{\rho\to 0}\frac{\LL^2(\{\sigma\cdot\xi\le 0\}\cap B_\rho(x_0))}{\pi\rho^2}=0,
$$
hence
$$
\lim_{\rho\to 0}\frac{\LL^2(\{\sigma\cdot\xi> 0\}\cap B_\rho(x_0))}{\pi\rho^2}=1.
$$
It shows that $x_0$ is a Lebesgue point of $\chi(\cdot,\xi)$ with $\chi(x_0,\xi)=1$. Note that the same argument shows that if $\sigma(x_0)\cdot
\xi<0$ then $x_0$ is also a Lebesgue point of $\chi(\cdot,\xi)$ with $\chi(x_0,\xi)=0$.

Consider the segment $S=[x_0,y_0]$ and let $U=\{z \in \O_p : \; \dist(z,S)<\e\}$ be a (connected) $\e$-neighborhood of $S$, where $\e>0$ is small enough so that $S \subset U \subset\subset \O_p$ (which is always possible thanks to (H)). Let $\{\eta_\e\}_{\e>0}$ be a standard family of mollifiers and set $\chi_\e:=\eta_\e * \chi(\cdot,\xi)$.  Because of the kinetic formulation  $D_\xi\chi(\cdot,\xi)=0$ in $\mathcal D'(\O_p)$, then
$D_\xi \chi_\e =0$ in $U$.  Thus $\chi_\e(z)=\chi_\e(z+\xi)$ for all $z \in U$ with $z+\xi \in U$.
Now, $x_0$ is a Lebesgue point of $\chi(\cdot,\xi)$, so that $\chi_\e(x_0) \to \chi(x_0,\xi)$. But $y_0=x_0+\xi$, so $\chi_\e(y_0) =\chi_\e(x_0+\xi)=\chi_\e(x_0)\to \chi(x_0,\xi)=1$ which implies that $\sigma(y_0)\cdot \xi \geq 0$. Indeed, if we had $\sigma(y_0)\cdot \xi<0$, then the
above argument would show that $y_0$ is a Lebesgue point of $\chi(\cdot,\xi)$ with $\chi(y_0,\xi)=0$, and thus $\chi_\e(y_0) \to \chi(y_0,\xi)=0$ which is a contradiction.
\end{proof}

Thanks to the previous ordering property, we will show that $\sigma$ is constant on every line $L_x$. The following result is an adaptation and an
improvement of \cite[Proposition 3.2]{JOP} (see also \cite[Proposition 4.8]{BP}).

\begin{proposition}\label{prop:caracteristic}
Assume that hypothesis  (H) holds. Let $x_0 \in \O_p$ be a {Lebesgue point of $\sigma$}. Then $ \sigma(x)=\sigma(x_0)$  for $\HH^1$-a.e. $x\in L_{x_0} \cap \O_p$.
\end{proposition}

\begin{proof}
Up to a change of coordinate system and to a scaling argument, we can assume without loss of generality that $x_0=(0,0)$, $L_{x_0}=\R e_2$, $L_{x_0} \cap \overline \O_p\supset\{0\} \times [-1,1]:=L$, that $\sigma(x_0)=e_1:=(1,0)$. Let us consider, for $\e\in(0,1)$, the triangle
$$T_\e=\left\{y=(y_1,y_2) \in \R^2 : \; 0 <y_2<1 \text{ and } 0<y_1<\frac\e{1-\e} y_2\right\}.$$
For all Lebesgue points $y \in T_\e$ of $\sigma$, we have
$$\sigma(0)\cdot y=y_1>0$$
which implies, according to Proposition \ref{prop:order}, that $\sigma(y)\cdot y \geq 0$. Thus
\begin{equation}\label{eq:eps}
\sigma_2(y)  \geq -\sigma_1(y)\frac{y_1}{y_2} \geq -\frac{\e}{1-\e}
\end{equation}
since $|\sigma_1(y)|\leq 1$ and $0<y_1/y_2 \leq \e/(1-\e)$.

Fix $\eta \in (0,1/4)$ and define $S_\eta:=\{0\}\times [\eta,1-\eta]$. Let $(0,x_2) \in S_\eta$ be a Lebesgue point of $\sigma$, and consider the half-ball centered at $(0,x_2)$ and radius $\e x_2$ with $\e\in (0,\eta)$, that is
$$B_\e^+(x_2)=\{y=(y_1,y_2) \in \R^2 : \; y_1>0 \text{ and } |y-(0,x_2)|<\e x_2\}.$$
It is immediately checked that $B_\e^+(x_2) \subset T_\e$ so that, in view of \eqref{eq:eps},
$$
\sigma_2(y)  \geq -\frac{\e}{1-\e} \; \mbox{ for a.e.  }y \in B_\e^+(x_2).
$$
Then,
$$-\frac{\e}{1-\e} - \sigma_2(0,x_2) \le \frac2{\pi\e^2x_2^2}\int_{B_\e^+(x_2)} |\sigma_2(y)-\sigma_2(0,x_2)|\, dy$$
so that, because $(0,x_2)$ is a Lebesgue point of $\sigma$ we can pass to
the $\e\searrow 0$ limit and we conclude that  $\sigma_2(0,x_2)\ge 0$.  A
similar argument would show that $ \sigma_2(0,x_2)\leq 0$,  hence {$\sigma_2(0,x_2)= 0$}. Recalling  Remark \ref{rem:h1pp}, we {get} that
$\sigma_2(0,x_2)=0$ for $\h$-a.e. $(0,x_2) \in S_\eta$. Since $\eta>0$ is arbitrary, we infer that
$$\sigma_2=0 \text{ for $\HH^1$-a.e. in } \{0\}\times(0,1).$$
The same type of argument would show that
$$\sigma_2=0 \text{ for $\HH^1$-a.e. in } \{0\}\times(-1,0).$$

As a consequence, since $|\sigma|=1$ $\HH^1$-a.e. in $L$, we deduce that
$$\sigma= e_1 {\bf 1}_A - e_1 {\bf 1}_{L \setminus A}$$
for some $\HH^1$-measurable set $A \subset L$.  According to Lemma \ref{lem:VMO}, the Sobolev space $H^{1/2}(L;\R^2)$ is continuously embedded into ${\rm VMO}(L;\R^2)$. Then, we deduce from Lemma \ref{lem:constant} that
either $\HH^1(A)=0$ or $\HH^1(L \setminus A)=0$. If $\HH^1(A)=0$, then $ \sigma=-e_1$ $\HH^1$-a.e. in $L$ so that $$-e_1 = \frac{1}{2\rho}\int_{L \cap B_\rho(x_0)}  \sigma(y)\, d\HH^1(y).$$
But according to Remark \ref{rem:Leb-pts},
$$\frac{1}{2\rho}\int_{L \cap B_\rho(x_0)}  \sigma(y)\, d\HH^1(y) \to \sigma(x_0)=e_1$$
which is impossible. As a consequence $\HH^1(L \setminus A)=0$ which implies that $\sigma=e_1$ $\HH^1$-a.e. in $L$ as required.
\end{proof}

We next show that two distinct characteristic lines cannot intersect inside {$\O_p$}.

\begin{proposition}\label{prop:vortex}
Assume that hypothesis  (H) holds. Let $x_0$ and $y_0 \in \O_p$ be two Lebesgue points of $\sigma$. Then either $L_{x_0}$ and $L_{y_0}$ are colinear, or $L_{x_0} \cap L_{y_0}=\{z_0\}$ where $z_0 \not\in \O_p$.
\end{proposition}

\begin{proof}
If $L_{x_0}$ and $L_{y_0}$ are not colinear, we have in particular that $\sigma(x_0) \neq \sigma(y_0)$, $\sigma(x_0) \neq -\sigma(y_0)$ and there exists a unique $z_0 \in L_{x_0} \cap L_{y_0}$. Assume that $z_0 \in \O_p$, and let $z_1 \in L_{x_0} \cap \O_p$ and $z_2 \in L_{y_0} \cap \O_p$ be
such that the triangle $T$ with vertices $z_0$, $z_1$ and $z_2$ satisfies
$\overline T \subset \O_p$. Since $\sigma \in H^1(T;\R^2),$  its trace $\sigma_{|_{\partial T}} \in H^{1/2}(\partial T;\R^2)$. Let us denote by $S_{x_0}:=\partial T \cap L_{x_0}$, $S_{y_0}:=\partial T \cap L_{y_0}$ and $\Gamma=S_{x_0}\cup S_{y_0}$ so that
$$\sigma_{|_{\partial T}} \in H^{1/2}(\Gamma;\R^2).$$
On the other hand, Proposition \ref{prop:caracteristic} implies that
$\sigma_{|_{\partial T}}=\sigma_{|_{S_{x_0}}}=\sigma(x_0)$ $\HH^1$-a.e. on $S_{x_0}$
and
$\sigma_{|_{\partial T}}=\sigma_{|_{S_{y_0}}}=\sigma(y_0)$ $\HH^1$-a.e. on $S_{y_0}$,
which is impossible in view of Lemma \ref{lem:notH1/2}. We thus deduce that $z_0 \not\in \O_p$.
\end{proof}

A concatenation of the previous results implies the announced local Lipschitz regularity for $\sigma$ (see \cite[Theorem 1.5]{BP}).

\begin{proof}[Proof of Theorem \ref{thm:lip}]
{\it Case I:} If $\sigma(x_0)=\sigma(y_0)$ then the result follows.

\medskip

{\it Case II:} We now prove that the case $\sigma(x_0)=-\sigma(y_0)$ cannot happen. If $\sigma(x_0)\cdot (y_0-x_0)>0$, according to Proposition \ref{prop:order}, we have $\sigma(y_0)\cdot (y_0-x_0) \geq 0$ hence $\sigma(x_0)\cdot (y_0-x_0) \leq 0$ which is impossible. A similar argument shows that $\sigma(x_0) \cdot (y_0-x_0)<0$ is impossible. It remains to consider the case where $\sigma(x_0)\cdot (y_0-x_0)=0$ which means that $y_0 \in L_{x_0}=:x_0+\R \sigma^\perp(x_0)$ or still  that $L_{y_0}= L_{x_0}$. Since $x_0$ and $y_0$ are Lebesgue points of $\sigma$, we obtain a contradiction with the result of   Proposition \ref{prop:caracteristic}.

\medskip

{\it Case III:} Assume now that $\sigma(x_0)$ and $\sigma(y_0)$ are not colinear. Then both lines $L_{x_0}=x_0+\R \sigma^\perp(x_0)$ and $L_{y_0}=y_0+\R \sigma^\perp(y_0)$ intersect at a single point $z_0 \not\in \O_p$ by Proposition \ref{prop:vortex}. Note that
$L_{x_0}$ (resp. $L_{y_0}$) is  colinear with $x_0-z_0$ (resp. $y_0-z_0$)
so that there is no loss of generality in assuming that {\it e.g.}
$$\sigma^\perp(x_0)=\frac{x_0-z_0}{|x_0-z_0|}, \quad \sigma^\perp(y_0)=\pm \frac{y_0-z_0}{|y_0-z_0|}.$$
We claim that actually
\begin{equation}\label{eq:claim}
\sigma^\perp(y_0)=\frac{y_0-z_0}{|y_0-z_0|}.
\end{equation}
Since $y_0 \not\in L_{x_0}$,
$$\sigma (x_0) \cdot \frac{y_0-x_0}{|y_0-x_0|}  \neq 0$$
so that Proposition \ref{prop:order} ensures that
\begin{equation}\label{eq:sign1}
{\rm sign}\left(\sigma (x_0) \cdot \frac{y_0-x_0}{|y_0-x_0|}\right)={\rm sign}\left(\sigma (y_0) \cdot \frac{y_0-x_0}{|y_0-x_0|}\right).
\end{equation}
Since $\frac{y_0-z_0}{|y_0-z_0|}$ belongs to the convex cone $C(\frac{x_0-z_0}{|x_0-z_0|},\frac{y_0-x_0}{|y_0-x_0|})$, there exist $\alpha > 0$ and $\beta > 0$ such that
$$\frac{y_0-z_0}{|y_0-z_0|}=\alpha\frac{x_0-z_0}{|x_0-z_0|}+\beta\frac{y_0-x_0}{|y_0-x_0|},$$
hence
$$\frac{y_0-z_0}{|y_0-z_0|} \cdot \frac{(y_0-x_0)^\perp}{|y_0-x_0|}=\alpha\frac{x_0-z_0}{|x_0-z_0|} \cdot \frac{(y_0-x_0)^\perp}{|y_0-x_0|}.$$
As a consequence
$$\pm \sigma^\perp(y_0) \cdot \frac{(y_0-x_0)^\perp}{|y_0-x_0|}=\alpha\sigma^\perp(x_0)\cdot \frac{(y_0-x_0)^\perp}{|y_0-x_0|}$$
or still
\begin{equation}\label{eq:sign2}
\pm \sigma(y_0) \cdot \frac{y_0-x_0}{|y_0-x_0|}=\alpha\sigma(x_0)\cdot \frac{y_0-x_0}{|y_0-x_0|}
\end{equation}
Gathering \eqref{eq:sign1} and \eqref{eq:sign2} yields \eqref{eq:claim}.

Since $\dist(z_0,\omega) \geq d$, then it follows that $|z_0-x_0|\geq d$ and $|z_0-y_0|\geq d$. Therefore the projections of $x_0-z_0$ and $y_0-z_0$ onto the closed ball $\overline B_d(z_0)$ are given, respectively, by
$d(x_0-z_0)/|x_0-z_0|$ and $d(y_0-z_0)/|y_0-z_0|$. Since the projection is $1$-Lipschitz, we deduce that
$$|\sigma^\perp(x_0)-\sigma^\perp(y_0)|=\left|\frac{x_0-z_0}{|x_0-z_0|}
-\frac{y_0-z_0}{|y_0-z_0|} \right| \leq \frac1d |x_0-y_0|,$$
and the conclusion follows.

{In the sequel, we will identify $\sigma$ with its locally Lipschitz representative. In particular, the conclusion of Proposition \ref{prop:order}
now holds {\it for all} $x_0$ and $y_0 \in \O_p$, while that of Proposition \ref{prop:caracteristic} states that {\it for all} $x_0 \in \O_p$, then $\sigma(x)=\sigma(x_0)$ {\it for all} $x \in L_{x_0} \cap \O_p$.}
\end{proof}

\subsection{Rigidity of the displacement}\label{sec:u}

We now demonstrate that the displacement field(s), like the stress field,
is (are) severely constrained by assumption (H) and conform(s) to what formal manipulations {of the hyperbolic equation \eqref{eq:hyperbolicu}} would entail, that is that the displacement must remain constant on (almost) every characteristic line in $\O_p$. Formally, the argument goes as follows. Compute the derivative of $u$ along the characteristics. The chain rule gives
$$\frac{d}{ds} u(x+s\sigma^\perp(x))=Du(x+s\sigma^\perp(x))\cdot {\sigma^\perp}(x).$$
Using that $\sigma$ is constant along the characteristics, we get
$$\frac{d}{ds} u(x+s\sigma^\perp(x))=Du(x+s\sigma^\perp(x))\cdot {\sigma^\perp}(x+s\sigma^\perp(x))=0,$$
since, thanks to the flow rule, $Du=\sigma(1+|p|)$ is colinear with $\sigma$. Unfortunately, this argument cannot be made rigorous for want of a
general chain rule formula for the composition of a $BV$ function with a (locally) Lipschitz function.

\medskip

We will first show that such a property actually holds locally, {\it i.e.}, for small values of $s$. Using a well-suited covering we then establish the global result in $\O_p$, resulting in the

\begin{theorem}\label{thm:cst-disp}
Assume that hypothesis  (H) holds. There exists an $\HH^1$-negligible set
$Z \subset \O_p$ with $\LL^2\left((\bigcup_{z \in Z}L_z)\cap \O_p\right)=0$ such that  $u$ is constant on $L_x\cap \O_p$ for all $x \in \O_p \setminus \left(\bigcup_{z \in Z} L_z\right)$.
\end{theorem}

The rest of the section is devoted to the proof of Theorem \ref{thm:cst-disp}. Let $\omega$ be a convex open set such that $\overline \omega \subset \O_p$. Let $x_0 \in \overline \omega$ and $R>0$ small enough so that $\overline{B_{2R}(x_0)} \subset \O_p$. Let us define the mapping {$\Phi_{x_0} : [-R,R] \times \R \to \R^2$} by
$$
\Phi_{x_0}(t,s):=x_0+t\sigma(x_0)+s\sigma^\perp(x_0+t\sigma(x_0)) \quad
\text{ for all }{(t,s) \in [-R,R]\times \R}.
$$
Clearly, $\Phi_{x_0}$ is locally Lipschitz in $[-R,R]\times \R$ as composition of locally Lipschitz mappings.

\begin{proposition}\label{prop:bi-Lip}
There exists $0<r<R$ and an open neighborhood $U_{x_0}$ of $x_0$ such that the function $\Phi_{x_0}$ is bi-Lipschitz from {$Q_r:=(-r,r)^2$} onto
$U_{x_0}$.
\end{proposition}

\begin{proof}
{We first observe that, since $\overline{B_{2R}(x_0)} \subset \O_p$ and both $\sigma(x_0)$ and $\sigma^\perp(x_0+t\sigma(x_0))$ are unit vectors,  $\Phi_{x_0}([-R,R]^2) \subset \O_p$. Moreover, using that $\sigma$ is locally Lipschitz in $\O_p$}, it follows that $t \in [-R,R]\mapsto \sigma^\perp(x_0+t\sigma(x_0))$ is Lipschitz. As a consequence of Rademacher's Theorem, it is differentiable almost everywhere and there exists $M>0$ such that
\begin{equation}\label{eq:M}
\left|\frac{d}{dt} \sigma^\perp(x_0+t\sigma(x_0))\right|\leq M \quad \text{ for a.e. }t \in [-R,R].
\end{equation}
Hence, $\Phi_{x_0}$ is Lipschitz in $\overline{Q}_R$ as  composition of Lipschitz functions.

\medskip

Further, $\Phi_{x_0}$ is injective on $\overline{Q_R}$. Indeed, if $(t,s)\ne (t',s')\in \overline{Q_R}$ are such that $\Phi_{x_0}(t,s)=\Phi_{x_0}(t',s')$, then
$$
z:=x_0+t\sigma(x_0)+s\sigma^\perp(x_0+t\sigma(x_0))=x_0+t'\sigma(x_0)+s'\sigma^\perp(x_0+t'\sigma(x_0)).
$$
Clearly $t\ne t'$ and thus $x_0+t\sigma(x_0) \neq x_0+t'\sigma(x_0)$. The
point $z$ would then belong to both $L_{x_0+t\sigma(x_0)}$ and $L_{x_0+t'\sigma(x_0)}$ so that, by Proposition \ref{prop:vortex}, we {would} have $L_{x_0+t\sigma(x_0)}\!\!=L_{x_0+t'\sigma(x_0)}$, hence, by Proposition \ref{prop:caracteristic}, $\sigma(x_0+t\sigma(x_0))\!\!=\sigma(x_0+t'\sigma(x_0))$. But then $\sigma(x_0)$ and $\sigma^\perp(x_0+t\sigma(x_0))$ are colinear, and, because they are both unit vectors,
\begin{equation}\label{eq:sig=sigperp}\sigma^\perp(x_0+t\sigma(x_0))=\pm \sigma(x_0).
\end{equation}
Consequently, we would have $x_0=x_0+t\sigma(x_0)\mp t\sigma^\perp(x_0+t\sigma(x_0))$, and  the characteristic line $L_{x_0+t\sigma(x_0)}$ would
intersect $L_{x_0}$ at the point $x_0$ which is impossible, owing again to Proposition \ref{prop:vortex}, unless $\sigma(x_0)=\pm \sigma(x_0+t\sigma(x_0))$. But this last relation contradicts \eqref{eq:sig=sigperp}.
The injectivity of $\Phi_{x_0}$ in $\overline{Q_R}$ is  established, and,
as a consequence of Brouwer's Invariance Domain Theorem (see \cite[Theorem 3.30]{FoGa}), $\Phi_{x_0}$ is a homeomorphism from the open square $Q_R$ onto its range which is open.

We  now compute the Jacobian determinant of $\Phi_{x_0}$. For a.e. $(t,s)
\in Q_R$, we have
\begin{eqnarray}\label{eq:det}
{\rm det} \nabla \Phi_{x_0}(t,s) &= &\sigma(x_0)\cdot\sigma(x_0+t\sigma(x_0))- s\sigma^\perp(x_0+t\sigma(x_0))\cdot \frac{d}{dt}\sigma(x_0+t\sigma(x_0))\nonumber\\
& :  = & a(t)+sb(t),
\end{eqnarray}
where, {for a.e.} $t \in (-R,R)$, we set
$$a(t):=\sigma(x_0)\cdot\sigma(x_0+t\sigma(x_0)), \quad b(t)=-\sigma^\perp(x_0+t\sigma(x_0))\cdot \frac{d}{dt}\sigma(x_0+t\sigma(x_0)).$$
Since $\sigma$ is Lipschitz in $\overline{B_{2R}(x_0)}$, there exists $K>0$ such that
\begin{equation}\label{eq:K}
|\sigma(x_0+t\sigma(x_0))-\sigma(x_0)|\leq K|t| \quad \text{ for all }t \in [-R,R].
\end{equation}
Thus, by \eqref{eq:M} and \eqref{eq:K}, for a.e. $t \in (-R,R)$,
\begin{equation}\label{eq:int-det}
a(t)\ge 1-K |t|,\quad |b(t)|\le M,
\end{equation}
and we can bound from below the right hand-side of \eqref{eq:det} by
$$
{\rm det} \nabla \Phi_{x_0}(t,s)\ge 1-K|t| - M|s|.
$$
Let $r>0$ be small enough so that {$1-(K+M)r > \frac12$}, then we get that
\begin{equation}\label{eq:det>0}
{\rm det} \nabla \Phi_{x_0}(t,s)> 1/2 \quad \text{ for a.e. } (t,s) \in Q_r.
\end{equation}
Denoting by $U_{x_0}:=\Phi_{x_0}(Q_r){\subset \O_p}$, we have so far established that $\Phi_{x_0} : Q_r \to U_{x_0}$ is a Lipschitz homeomorphism. With the help of \eqref{eq:det>0}, we now show that $\Phi_{x_0}^{-1}$ is Lipschitz in $U_{x_0}$. To that end, we prove in a  manner similar
to that in \cite[Theorem 6.1]{FoGa} that $\Phi_{x_0}^{-1} \in W^{1,\infty}(U_{x_0};\R^2)$. Since we already know that $\Phi_{x_0}^{-1}$ is continuous in $U_{x_0}$, it enough to prove that $\Phi_{x_0}^{-1}$ has a weak gradient in $L^\infty(U_{x_0};\mathbb M^{2 \times 2})$.

Let {$\varphi \in \mathcal C^\infty_c(U_{x_0})$} be an arbitrary test
function. Using the area formula with the one-to-one Lipschitz function $\Phi_{x_0}$, together with the chain rule formula $\nabla (\varphi \circ \Phi_{x_0})=(\nabla \Phi_{x_0})^T \nabla \varphi \circ \Phi_{x_0}$, we get that for $i=1,2$,
\begin{eqnarray*}
\int_{U_{x_0}} \Phi_{x_0}^{-1}(x)\, \partial_i \varphi(x)\, dx & = & \int_{Q_r}  \Phi_{x_0}^{-1}(\Phi_{x_0}(t,s)) \, \partial_i \varphi(\Phi_{x_0}(t,s)) \, {\rm det} \nabla \Phi_{x_0}(t,s)\, dt\, ds\\
 & = &   \int_{Q_r}   \big[ {\rm cof}(\nabla \Phi_{x_0}(t,s)) \nabla (\varphi \circ \Phi_{x_0})(t,s)\big]_i
\begin{pmatrix} t\\s\end{pmatrix}  dt\, ds.
\end{eqnarray*}
Let $(\Phi_n)_{n \in \N}$ be an approximating sequence in $\mathcal C^\infty(\overline{Q_r};\R^2)$ such that $\Phi_n \to \Phi_{x_0}$ uniformly in $\overline{Q_r}$ and also in $W^{1,p}(Q_r;\R^2)$ for all $p<\infty$, then
$$\int_{U_{x_0}} \Phi_{x_0}^{-1}(x)\, \partial_i \varphi(x)\, dx= \lim_{n \to +\infty}  \int_{Q_r}   \big[ {\rm cof}(\nabla \Phi_n(t,s)) \nabla (\varphi \circ \Phi_n)(t,s)\big]_i \begin{pmatrix} t\\s\end{pmatrix} dt\,
ds.$$
Integrating by parts, using that ${\rm div} \big({\rm cof} (\nabla \Phi_n)\big)=0$, as well as the area formula with the function $\Phi_{x_0}$ once more, we get
\begin{eqnarray*}
\int_{U_{x_0}} \Phi_{x_0}^{-1}(x) \,\partial_i \varphi(x)\, dx & = & - \lim_{n \to +\infty}\int_{Q_r} \big( {\rm cof}(\nabla \Phi_n)\big)^{(i)}
(\varphi \circ \Phi_n) \, dt\, ds\\
& = &- \int_{Q_r}  \big({\rm cof}(\nabla \Phi_{x_0})\big)^{(i)}  (\varphi \circ \Phi_{x_0}) \, dt\, ds\\
& = & - \int_{U_{x_0}} \frac{\big( {\rm cof}(\nabla \Phi_{x_0}(\Phi_{x_0}^{-1}))\big)^{(i)} }{{\rm det}\nabla \Phi_{x_0}(\Phi_{x_0}^{-1})}\varphi\, dx,
\end{eqnarray*}
where $A^{(i)}$ stands for the $i$-th row of the matrix $A \in \mathbb M^{2 \times 2}$. By definition of the weak gradient, \eqref{eq:det>0} and since $\Phi_{x_0}$ is  Lipschitz in {$\overline{Q}_r$}, we infer that
$$\nabla\Phi_{x_0}^{-1}= \frac{\big({\rm cof}(\nabla \Phi_{x_0}(\Phi_{x_0}^{-1}))\big)^T}{{\rm det}\nabla \Phi_{x_0}(\Phi_{x_0}^{-1})} \in  L^\infty(U_{x_0};\mathbb M^{2 \times 2}),$$
which completes the proof of the Proposition.
\end{proof}

The mapping {$\Phi_{x_0}$}  provides a convenient change of variables thanks to which we now deduce that the displacement is locally constant along characteristics.

\begin{proposition}\label{prop:cst-disp}
The function $(t,s) \in Q_r \mapsto u\circ \Phi_{x_0}(t,s) \in BV(Q_r)$ only depends on $t$.
\end{proposition}

\begin{proof}Let $(u_n)_{n \in \N}$ be a sequence in $\C^\infty(\O) \cap W^{1,1}(\O)$ be such that $u_n \to u$ in $L^1(\O)$ and $Du_n \wto Du$ weakly* in $\M(\O;\R^2)$. Let us define the functions
$$v=u \circ \Phi_{x_0}, \quad v_n=u_n\circ \Phi_{x_0}.$$
Note that $v_n \in W^{1,\infty}(Q_r)$. Because, up to a subsequence, $u_n
\to u$ a.e. in $\O$ and because $\Phi^{-1}_{x_0}$, being Lipschitz, maps sets of zero Lebesgue measure into sets of zero Lebesgue measure, $v_n \to v$ a.e. in $Q_r$. The area formula (applied to the function $\Phi_{x_0}$) together with \eqref{eq:det>0} implies that,
$$\int_{Q_r} |v_n-v|\, dt\, ds = \int_{U_{x_0}}  \frac{|u_n-u|}{{\rm det}\nabla \Phi_{x_0}(\Phi_{x_0}^{-1})}\, dx \leq 2 \int_{U_{x_0}} |u_n-u| \, dx\to 0.$$
Hence $v_n \to v$ in $L^1(Q_r)$. Since $\Phi_{x_0} \in W^{1,\infty}(Q_r;\R^2)$ and $\nabla v_n=(\nabla\Phi_{x_0})^T \nabla u_n \circ \Phi_{x_0}$, the same change of variable argument yields in turn
\begin{multline*}\int_{Q_r} |\nabla v_n|\, dt\, ds \leq  \|\nabla \Phi_{x_0}\|_{L^\infty(Q_r)} \int_{Q_r} |\nabla u_n(\Phi_{x_0})|\, dt\, ds\\ \leq 2 \|\nabla \Phi_{x_0}\|_{L^\infty(Q_r)} \!\!\!\int_{U_{x_0}}|\nabla u_n|\, dx \leq C
\end{multline*}
for some constant $C>0$ independent of $n$. Hence $\nabla v_n \LL^2 \wto Dv$ weakly* in $\M(Q_r;\R^2)$ and $v \in BV(Q_r)$.

\medskip

Because $\Phi_{x_0}$ is Lipschitz in $\overline{Q_r}$,   the function
$$(t,s) \in Q_r \mapsto g(t,s):={\rm det}\nabla \Phi_{x_0}(t,s)=a(t)+
sb(t)$$ defined in \eqref{eq:det} is in $L^\infty(Q_r)$
and it is affine with respect to $s$ for a.e. $t \in {(-r,r)}$. It follows from an integration by parts that for all $\varphi \in \C^\infty_c(Q_r)$
\begin{equation}\label{eq:lim1}
\int_{Q_r} (\partial_s v_n) \,\varphi\, g\, dt\, ds=-\int_{Q_r} \left((\partial_s \varphi) \,v_n \, g + \varphi \, v_n\, \partial_s g\right) dt\, ds.
\end{equation}
On the other hand, using that $\sigma$ is constant on each characteristic
line, {\it i.e.}, that $\sigma(x_0+t\sigma(x_0))=\sigma(\Phi_{x_0}(t,s))$ for all $(t,s) \in Q_r$, we get that
\begin{eqnarray}\label{eq:lim2}
&&\int_{Q_r} (\partial_s v_n) \,\varphi\, g\, dt\, ds\nonumber\\
& & \hspace{1cm}= \int_{Q_r} \!\!\nabla u_n(\Phi_{x_0}(t,s)) \cdot \sigma^\perp(x_0+t\sigma(x_0)) \, \varphi(t,s)\, g(t,s)\, dt ds\nonumber\\
& & \hspace{1cm}=\int_{Q_r} \nabla u_n(\Phi_{x_0}(t,s)) \cdot \sigma^\perp(\Phi_{x_0}(t,s))\, \varphi \circ \Phi_{x_0}^{-1}(\Phi_{x_0}(t,s))\, g(t,s)\, dtds\nonumber\\
& & \hspace{1cm}=\int_{U_{x_0}} \nabla u_n \cdot \sigma^\perp \varphi\circ \Phi_{x_0}^{-1}\, dx,
\end{eqnarray}
where we used once more the area formula with the function $\Phi_{x_0}$ in the last equality. Since $\Phi_{x_0}^{-1}$ is continuous, $\sigma$ is locally Lipschitz  and $\varphi \in \C_c(Q_r)$, it follows that the function $\sigma^\perp \varphi \circ \Phi_{x_0}^{-1}$ belongs to $\C_c(U_{x_0};{\R^2})$. Gathering \eqref{eq:lim1}, \eqref{eq:lim2}, and passing to the limit leads to
$$-\int_{Q_r} \left((\partial_s \varphi) \,v \, g + \varphi \, v\, \partial_s g\right) dt\, ds=\int_{U_{x_0}} \varphi \circ\Phi_{x_0}^{-1} \sigma^\perp  \cdot dDu.$$
Because of \eqref{eq:DU=}, the right-hand side of the previous equality vanishes and
\begin{equation}\label{eq:gv}
\int_{Q_r} \left((\partial_s \varphi) \,v \, g + \varphi \, v\, \partial_s g\right) dt\, ds=0 \quad \text{ for all }\varphi \in \C^\infty_c(Q_r).
\end{equation}
Since $v \in BV(Q_r)$, it follows from slicing properties of $BV$-functions (see \cite[Theorem 3.107]{AFP}) that $s \mapsto v_t(s)=v(t,s)$ belongs to $BV((-r,r))$ for $\LL^1$-a.e. $t \in (-r,r)$, and, using disintegration, that $D_s v=\LL^1_t \otimes Dv_t$  in $\mathcal M(Q_r)$, {\it i.e.},
$$\int_{Q_r} \bar \varphi \, dD_s v=\int_{-r}^r\left(\int_{-r}^{r} \bar \varphi(t,s)\, dDv_t(s)\right){dt} \quad \text{ for all } \bar \varphi\in \C_c(Q_r).$$
Since $s \mapsto g(t,s)$ is affine, we can thus localize \eqref{eq:gv} in $t$ and we conclude that, for $\LL^1$-a.e. $t \in (-r,r)$,
$$ \int_{-r}^{r} g(t,s)\phi(s) \, d Dv_t(s) =0 \quad \text{ for all }\phi \in \C_c((-r,r)),$$
which means that
$$(a(t)+sb(t)) Dv_t=0 \quad \text{ in } \mathcal M((-r,r)) \quad \text{
for  $\LL^1$-a.e. }t \in (-r,r).$$
Note that by our choice of $r$ and \eqref{eq:int-det} $a(t) \geq \frac12$
for all $t \in (-r,r)$. If $b(t)=0$, we have $|a(t)+sb(t)|=|a(t)| \geq \frac12$ for all $s \in (-r,r)$. On the other hand, if $b(t) \neq 0$, we get with \eqref{eq:int-det} that
$$\left|\frac{a(t)}{b(t)} \right| \geq \frac{1}{2M}.$$
Since our choice of $r$ ensures that $r<1/2M$, we conclude that $|a(t)+sb(t)|>0$ for $s\in(-r,r)$. We conclude that $Dv_t=0$ in $\mathcal M((-r,r))$ for $\LL^1$-a.e. $t \in (-r,r)$, which proves that $D_s v=0$ in $\M(Q_r)$.
\end{proof}

\medskip

We now proceed to the proof of Theorem \ref{thm:cst-disp}.

\begin{proof}[Proof of Theorem \ref{thm:cst-disp}]
Since $\{U_{x}\}_{x \in \overline \omega}$ is an open covering of the compact set $\overline \omega$, we can extract a finite sub-covering. We can
thus find finitely many points $x_1,\ldots,x_N \in \overline \omega$ such
that
$$\overline \omega \subset \bigcup_{i=1}^N U_{x_i}.$$
Let $Q_i =(-r_i,r_i)^2=\Phi^{-1}_{x_i}(U_{x_i})$ be the corresponding
open sets with $v_i:=u \circ \Phi_{x_i} \in BV(Q_i)$. By Proposition \ref{prop:cst-disp}, $D_s v_i=0$ in $\M(Q_i)$, which means that the function
$$(t,s) \in Q_i \mapsto u\big(x_i+t\sigma(x_i)+s\sigma^\perp(x_i+t\sigma(x_i))\big)$$
is an {$s$-independent $BV(Q_i)$-function}. Using slicing properties of $BV$ functions (see \cite[Theorem 3.107]{AFP}), there exists an $\LL^1$-negligible set $N_i \subset (-r_i,r_i)$ such that for all $t \in (-r_i,r_i)
\setminus N_i$,
\begin{equation}\label{eq:inter-ucst}
s \in (-r_i,r_i) \mapsto u\big(x_i+t\sigma(x_i)+s\sigma^\perp(x_i+t\sigma(x_i))\big)
\end{equation}
is constant.

{Set $\Gamma_i=\Phi_{x_i}((-r,r) \times \{0\})$ and $Z_i:=\Phi_{x_i}(N_i \times \{0\})$. Since $\LL^1(N_i)=0$ and $\Phi_{x_i}$ is Lipschitz in $Q_r$, $\HH^1(Z_i)=0$. Further, since $\O_p$ is bounded, there exists $T>0$ large enough so that $\bigcup_{z \in Z_i}L_z\cap\O_p \subset
\Phi_{x_i}(N_i \times (-T,T))$. Using again that $\LL^1(N_i)=0$, then $\LL^2(N_i \times (-T,T))=0$ and because $\Phi_{x_i}$ is a Lipschitz function on $[-R,R] \times [-T,T]$, so $\LL^2\left(\big(\bigcup_{z \in Z_i}L_z\big)\cap\O_p\right)=0$. From \eqref{eq:inter-ucst}, for all $x \in \Gamma_i \setminus Z_i$,
$$s \in (-r_i,r_i) \mapsto u(x+s\sigma^\perp(x))$$
is constant. By construction, any point $x \in  U_{x_i}\setminus (\bigcup_{z \in Z_i}L_z)$ has an associated characteristic line $L_x$ passing through $\Gamma_i$ and, according to Proposition \ref{prop:vortex}, $L_x \cap L_z \cap \O_p= \emptyset$ for all $z \in Z_i$. It thus follows that, for all $x \in U_{x_i}\setminus (\bigcup_{z \in Z_i}L_z)$, the function
$$s \in (-r_i,r_i) \mapsto u(x+s\sigma^\perp(x))$$
is constant, that is $u$ remains constant along $L_{x}\cap U_{x_i}$. }

In turn, $\{U_{x_i}\}_{1 \leq i \leq N}$ is a finite covering of $\overline \omega$, so that $u$ is constant on $L_x \cap \overline \omega$ for every $x\in \omega\setminus (\bigcup_{z \in Z_{\omega}} L_z)$ with $Z_{\omega}:=\bigcup_{i=1}^N Z_{i}$ satisfying $\HH^1(Z_\omega)=0$ and $$\mathcal L^2\left(\left(\bigcup_{z \in Z_{\omega}} L_z\right)\cap\O_p\right)=0.$$ Finally, consider an exhaustion $\{\omega_k\}_{k \in \N}$ of open sets with $$\omega_k=\left\{y \in \O_p : {\rm dist}(x,\partial\O_p)>\frac{1}{k}\right\}$$ for all $k \geq 1$. Set $Z:=\bigcup_{k \in \N} Z_{\omega_k}$ which satisfies
$$\HH^1(Z)=0, \quad \LL^2\left(\left(\bigcup_{z \in Z} L_z\right)\cap\O_p\right)=0.$$
In conclusion, for all $x \in \O_p \setminus \left(\bigcup_{z \in Z} L_z\right)$, the function $u$ is constant on $L_x \cap \O_p$.
\end{proof}

\section{Geometry of the solutions}\label{sec:rig}

In this section, we  show that assumption (H) constrains the geometric structure of the plastic zone $\O_p$, and of the solutions $(\sigma,u)$ in that set. In particular, subsets $\mathbf F_{\bar z}$ of $\O_p$ which are
bounded by characteristics intersecting at $\bar z \in \partial\O_p$ lead
to boundary fans where the stress behaves like a vortex centered at the apex $\bar z$ of the fan, and the displacement is a monotone function of the angle (see Theorem \ref{thm:fan-sigma}). The complementary of those fans is made of either isolated characteristic lines (see Proposition \ref{prop:emptyint}), or convex sets $\mathbf C$ with one or two characteristics on the boundary (see Theorem \ref{thm:geom-struct}). In any case, if the non characteristic boundary of either $\mathbf F_{\bar z}$, or $\mathbf C$ intersects the boundary of the domain $\O$ on a set of positive $\HH^1$ measure, then the propagation of the prescribed Dirichlet boundary datum $w$ through the characteristics inside $\mathbf F_{\bar z}$ or $\mathbf C$ provides a partial uniqueness property of the displacement (see Propositions \ref{prop:unique-u-fan} and \ref{prop:unique-C}). In the case of a connected component which is not a boundary fan,  possible geometries are offered in Subsection \ref{examples}. They consist of exterior fans
 (Proposition \ref{prop:ext-fan}), constant zones (Proposition \ref{prop:cst-u}), or smooth one-parameter families of lines (see Paragraph \ref{4.5.3}). We conjecture that those are the only possibilities.

\subsection{Boundary fans}\label{bdary-fans}

We set
$$\mathcal F:=\Big\{\bar z \in {\partial  \O_p}: \; \exists \, x, y \in
\O_p \text{ with }y \neq x \text{ such that } L_x \cap L_y= \{\bar z\}\Big\}.$$
It is the set of all  points on $\partial \O_p$ belonging to (at least) two distinct characteristic lines. The next result shows that {any such point is the apex} of a cone which is the unique intersection point of all other characteristic lines inside the cone.

\begin{lemma}\label{lem:cone-bdary}
Let $\bar z \in \mathcal F$ and $x$, $y \in \O_p$ with $y \neq x$ be such
that $L_x \cap L_y= \{\bar z\}$. Then for all $z \in {(\bar z+ C(x-\bar
z,y-\bar z))}\cap \O_p$,
$$L_z \cap L_{x}=L_z \cap L_{y}=L_x \cap L_y=\{\bar z\}.$$
\end{lemma}

\begin{proof}
Since from hypothesis  (H) the set $\O_p$ is assumed to be convex, there exist unique points $x' \in \partial \O_p \cap L_x$ and $y' \in \partial \O_p \cap L_y$ with $x' \neq \bar z$ and $y' \neq \bar z$ with the open segments $]\bar z,x'[$ and $]\bar z,y'[$  contained in $\O_p$.  Let us consider the (closed) triangle $T:=(x',y',\bar z) \subset \overline \O_p$.

If $z \in \mathring T$, the line $L_z$ passing through $z$ must intersect
at least one of the segments $[\bar z,x'[$ or $[\bar z,y'[$. Without loss
of generality, we suppose that $[\bar z,x'[ \,\cap L_z \neq \emptyset$ and let $\hat z \in [\bar z,x'[\, \cap L_z$. Let us assume  that $\hat z \neq \bar z$. Since $L_z \cap L_x=\{\hat z\}$, we get from Proposition \ref{prop:vortex} that $\hat z \not\in \O_p$ which contradicts the fact that $\hat z\in\, ]\bar z,x'[\ \subset\O_p$. Thus $\hat z=\bar z$ and $L_z$ passes through $\bar z$.

If $z \not\in \mathring T$, take any point $w$ on  $]\bar z,z[\ \cap\ \mathring T$. Then, according to the previous argument $L_w$ passes through the point $\bar z$, thus $[\bar z,w] \subset L_w$ and  $z \in L_w$. Therefore, Proposition \ref{prop:caracteristic} ensures that $L_w=L_z$, hence $L_z$ passes through $\bar z$ as well.
\end{proof}

For all $\bar z \in \mathcal F$, we consider the set of all points $z$ in
$\O_p$ such that the associated characteristic line $L_z$ passes through the point $\bar z$, {\it i.e.},
$$\mathbf F_{\bar z}:={\rm int} \{z \in \O_p: \; \bar z \in L_{z}\}.$$
In view of Lemma \ref{lem:cone-bdary}, it is a non empty open set of the form $\mathbf F_{\bar z}:= (\bar z+C(x-\bar z,y-\bar z))\cap\O_p$, for some $x,y \in \overline\O_p$, and we call that open subset of $\O_p$ a {\it boundary fan} (see Figure \ref{fig:fan}).
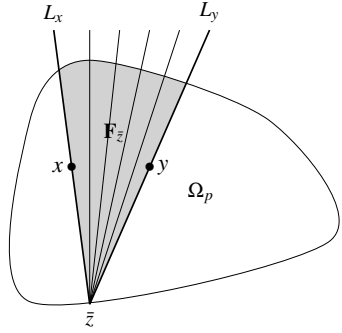
\begin{figure}[htbp]
\scalebox{.8}{\begin{tikzpicture}

\fill[color=gray!35]  plot[smooth cycle] coordinates {(-1,-1) (4,0) (3,2) (0,3) (-1,2)};
\fill[color=white]  plot coordinates { (0,-1.04) (-0.6,3.5) (-2.6,3.5)
(-2.6,-1.2) (0,-1.04)};
\fill[color=white]  plot coordinates { (0,-1.04)  (2,3.5) (5,3.5) (5,-1.04) };

\draw plot[smooth cycle] coordinates {(-1,-1) (4,0) (3,2) (0,3) (-1,2)};

\draw[style=thick] (0,-1.04) -- (-0.6,3.5) node[midway]{{\small $\bullet$}} node[midway,left]{$x$};
\draw (0,-1.04) -- (0,3.5) ;
\draw (0,-1.04) -- (0.5,3.5) ;
\draw (0,-1.04) -- (1,3.5) ;
\draw (0,-1.04) -- (1.5,3.5) ;
\draw[style=thick] (0,-1.04) -- (2,3.5) node[midway]{{\small $\bullet$}} node[midway,right]{$y$};

\draw (0.1,1.8) node[right]{\small $\mathbf F_{\bar z}$};
\draw (1.5,0.8) node[right]{\small $\Omega_p$};
\draw (0,-1.04) node[below]{\small $\bar z$};
\draw (-0.6,3.5) node[above]{\small $L_x$};
\draw (2,3.5) node[above]{\small $L_y$};
\end{tikzpicture}}
\caption{\small  A boundary fan with apex $\bar z$}
\label{fig:fan}
\end{figure}
Note that it might be the case that $[\bar z,x] \subset \partial \O_p$ or
 $[\bar z,y] \subset \partial \O_p$. In addition, if $\mathbf F_{\bar z}$
and $\mathbf F_{\bar z'}$ are two distinct boundary fans, for some $\bar z$ and $\bar z' \in \mathcal F$ with $\bar z \neq \bar z'$, then Proposition \ref{prop:vortex} ensures that ${\mathbf F_{\bar z}}\cap {\mathbf F_{\bar z'}}=\emptyset$ so that there are at most countably many boundary fans.

We now show a rigidity property of the Cauchy stress $\sigma$ in {$\bigcup_{\bar z\in \mathcal F} \mathbf F_{\bar z}$}: it is a vortex inside each boundary fan $\mathbf F_{\bar z}$. The corresponding displacement is constant along all characteristic lines. Note that the conclusion of Theorem \ref{thm:fan-sigma} below is stronger than that of Theorem \ref{thm:cst-disp}. Here we get a monotone function of
the angle (which {parameterizes} the characteristics in the case of a fan).

\begin{theorem}\label{thm:fan-sigma}
Let $\bar z \in \mathcal F$. Then there exists $\alpha \in \{-1,1\}$ such
that
$$\sigma(x)=\alpha \frac{(x-\bar z)^\perp}{|x-\bar z|} \quad \text{ for
all }x \in \overline{\mathbf F}_{\bar z}\cap\O_p \setminus \{\bar z\}$$
and
$$u(x)=\alpha h\left( \frac{(x-\bar z)_2}{(x-\bar z)_1}\right) \text{ for all }x \in \mathbf F_{\bar z},$$
for some nondecreasing function $h:\R \to \R$.
\end{theorem}

\begin{proof}
Assume for simplicity that $\bar z=(0,0)$. We  use the change of variables in polar coordinates $\Psi : (0,+\infty) \times (-\pi,\pi) \to \R^2 \setminus ((-\infty,0] \times \{0\})$ given by $\Psi(r,\theta)=(r\cos\theta,r\sin\theta)$. We set $e_r=(\cos\theta,\sin\theta)$, $e_\theta=(-\sin\theta,\cos\theta)$ and  $\sigma_r=(\sigma \circ \Psi) \cdot e_r$, $\sigma_\theta=(\sigma \circ \Psi) \cdot e_\theta$ so that $\sigma\circ
\Psi=\sigma_r e_r +\sigma_\theta e_\theta$. Since $\O_p$ is convex, so is $\mathbf F_{\bar z}$, thus
$$\mathbf F_{\bar z}=\{(r\cos \theta,r\sin\theta) \in \R^2: \; r>0 \text{ and }\theta_0< \theta< \theta_1\} \cap \O_p,$$
for some $-\frac{\pi}{2} \leq \theta_0 <\theta_1 < \frac{\pi}{2}$. We define
$$\tilde F:=\Psi^{-1}(\mathbf F_{\bar z})=\{(r,\theta) \in \R^2 : \; (r\cos \theta,r\sin\theta) \in \mathbf F_{\bar z}\}.$$

According to Proposition \ref{prop:caracteristic}, for all $z \in \mathbf
F_{\bar z}$, the vector field $\sigma$ is constant along all characteristic lines $L$ (which all pass through the origin),  and, further, it is orthogonal to the direction of $L$. Therefore, $\sigma(r\cos\theta,r\sin\theta):=f(\theta) e_\theta$ and, since $(r,\theta) \mapsto \sigma(r\cos\theta,r\sin\theta)$ is locally Lipschitz continuous,  $f:(\theta_0,\theta_1)\to \R$ is so as well. Using that $|\sigma|=1$ together with the expression of the divergence in polar coordinates, we conclude that
$$|f|=1 \text{ in } (\theta_0,\theta_1) \text{ and } f'=0  \text{ a.e. in } (\theta_0,\theta_1),$$
{\it i.e.},  $f \equiv 1$ or $f \equiv -1$. Therefore,  $\sigma\circ \Psi=\alpha e_\theta$, with $\alpha = 1$ or $-1$. Coming back to cartesian coordinates leads to the desired expression for $\sigma$ in $\mathbf F_{\bar z}$, hence to $\overline{\mathbf F}_{\bar z}\cap\O_p\setminus \{\bar z\}$ by continuity.

\medskip

Recalling \eqref{eq:DU=} $Du=\sigma\mu$  in  $\M(\O_p)$ with $\mu:=\LL^2+|p|$. Applying the ${\rm curl}$ to the previous equality yields
\begin{equation}\label{eq:curl1}
0={\rm curl }(\sigma\mu)=\Div (\sigma^\perp \mu) \quad \text{ in } \mathcal D'(\O_p).
\end{equation}
Let $\tilde \mu=\Psi^{-1}\, \# \, \mu \in \M(\tilde F)$ be the push-forward of $\mu$ by $\Psi^{-1}$. Since $\sigma^\perp(x)=-\alpha \frac{x}{|x|}$ for all $x \in {\mathbf F_{\bar z} \setminus \{\bar z\}}$, \eqref{eq:curl1} is easily seen to imply that $D_r \tilde \mu=0$ in $\mathcal D'(\tilde F)$ upon testing $D_r
\tilde \mu$ by smooth functions of the form $\varphi\circ \Psi=:\tilde \varphi  \in \C^\infty_c(\tilde F)$. This implies the existence of an orthoradial nonnegative measure $\eta \in \M((\theta_1,\theta_2))$ such that
$\tilde \mu= \LL^1_r \otimes \eta$, {\it i.e.},  for all $\varphi \in \mathcal C^\infty_c(\mathbf F_{\bar z})$,
$$\int_{\mathbf F_{\bar z}} \varphi \, d\mu=\int_{\tilde F} \tilde \varphi \, d\tilde \mu= \int_{\tilde F} \tilde \varphi(r,\theta)\, dr \, d\eta(\theta),$$
where, once again, $\varphi\circ \Psi=\tilde \varphi$. As a consequence, since $Du=\sigma\mu=\alpha \frac{\;\ x^\perp}{|x|\;} \mu$, we deduce that
\begin{equation}\label{eq:Du1}
\int_{\mathbf F_{\bar z}} \varphi \, dDu=\alpha \int_{\tilde F} \tilde \varphi(r,\theta) e_\theta \, dr \, d\eta(\theta).
\end{equation}
But,
\begin{eqnarray}\label{eq:Du2}
\int_{\mathbf F_{\bar z}} \varphi \, dDu & = & -\int_{\mathbf F_{\bar z}} u \nabla \varphi \, dx\nonumber\\
& = & -\int_{\tilde F}u(r\cos\theta,r\sin\theta) \nabla \varphi(r\cos \theta,r\sin\theta)\, r \, dr\, d\theta\nonumber\\
& = & -\int_{\tilde F}  \tilde u \left[D_r \tilde \varphi \, e_r+ \frac{1}{r} D_\theta \tilde \varphi\,  e_\theta\right] r \, dr \, d\theta,
\end{eqnarray}
where we set $\tilde u:=u \circ \Psi$. Since $\tilde F$ does not contain the exceptional line $(-\infty,0] \times \{0\}$, it follows that $\Psi$ defines a $\C^\infty$-diffeomorphism between $\tilde F$ and $\mathbf F_{\bar z}$ so that $\tilde u \in BV(\tilde F)$. Thus, since $\partial_r e_r =0$, $\partial_\theta
e_\theta =-e_r$, \eqref{eq:Du2} reads as
\begin{equation}\label{eq:Du3}
\int_{\mathbf F_{\bar z}} \varphi \, dDu=\langle (D_r(r\tilde u)-\tilde
u)e_r+(D_\theta \tilde u) e_\theta,\tilde \varphi\rangle=\langle ({r}D_r\tilde u)e_r+(D_\theta \tilde u) e_\theta,\tilde \varphi\rangle,
\end{equation}
and gathering \eqref{eq:Du1} and \eqref{eq:Du3} yields
$${r}D_r\tilde u =0, \quad D_\theta \tilde u=\alpha \LL^1_r \otimes \eta \quad \text{ in }\mathcal D'(\tilde F).$$
By the first equation, there exists an orthoradial function $\tilde h \in
BV((\theta_0,\theta_1))$ such that $\tilde u(r,\theta)=\alpha \tilde h(\theta)$ for a.e. $(r,\theta) \in \tilde F$. The second equation leads to
$D\tilde h=\eta \geq 0$ which implies that $\tilde h$ is nondecreasing.
Using that $\theta=\arctan\left(\frac{(x-\bar z)_2}{(x-\bar z)_1}\right)$ and setting $h=\tilde h \circ \arctan$, the result then follows coming back to cartesian coordinates.
\end{proof}

In the following result, we prove a partial uniqueness result for a fan for which a portion of  its  ``top" boundary coincides with that of $\O$, which of course may not happen.

\begin{proposition}\label{prop:unique-u-fan}
{Extend $u$ by $w$ outside $\Omega$. Let $\mathbf F_{\bar z}$ be a fan centered at $\bar z \in \mathcal F$. Then, $u^+=u^-$ $\HH^1$-a.e. on $\partial \mathbf F_{\bar z}\cap \partial\O_p$ and, in particular, $u=w$ $\HH^1$-a.e. on  $\partial \mathbf F_{\bar z}\cap \partial\O_p \cap \partial\Omega$.}
\end{proposition}

\begin{proof}
{Let $C_{\bar z}:=\bar z+C(x-\bar z,y-\bar z)$, for some $x$, $y\in\overline{\Omega}_p$, be the maximal open cone with vertex $\bar z$ such that $\mathbf F_{\bar z}=C_{\bar z}\cap \O_p$. Then the set $\partial\mathbf F_{\bar z} \cap C_{\bar z}=\partial \O_p \cap  C_{\bar z}$ is} open in the relative topology of $\partial \O_p$ and Lipschitz.  According to Remark \ref{rmk:jump-flow-rule}, $\sigma\cdot\nu=\pm1$
$\HH^1$-a.e. on $\partial \mathbf F_{\bar z}\cap C_{\bar z} \cap J_u$. Using the explicit expression of $\sigma$ on $\overline{\mathbf F}_{\bar z}
\setminus \{\bar z\}$ given by Theorem \ref{thm:fan-sigma}, we deduce that $\sigma\cdot \nu$ coincides $\HH^1$-a.e. with the usual scalar product of $\sigma$ and $\nu$ on $\partial \mathbf F_{\bar z}\cap C_{\bar z}$. Therefore, $\nu=\pm\sigma$ $\HH^1$-a.e. on $\partial \mathbf F_{\bar z}\cap C_{\bar z} \cap J_u$.

Assume that $\HH^1(\partial \mathbf F_{\bar z} \cap C_{\bar z}\cap J_u)>0$, since $\HH^1$ almost every point of $\partial \mathbf F_{\bar z} \cap
 C_{\bar z}$ is a differentiability point of the boundary, we can find some $x_0 \in \partial \mathbf F_{\bar z}   \cap C_{\bar z} \cap J_u$ such that, up to a change of sign, $\nu(x_0)=\sigma(x_0)$. Let us consider the characteristic line $L_{x_0}=x_0+\R\sigma^\perp(x_0)$ which passes through the point $\bar z$ (because $x_0 \in \overline{\mathbf F}_{\bar z}$), and let $H$  be the closed half plane such that $\partial H=L_{x_0}$ which does not contain $\sigma(x_0)$ (so that $\sigma(x_0)$ is the outer unit normal to $H$). Since $\sigma(x_0)$ is also the unit outer normal to $\O_p$ at $x_0$, it results from the convexity of $\O_p$ that $\overline \O_p \subset H$. Note that since $x_0 \in C_{\bar z}$ which is open, then the point $x_0$ does not belong to the boundary of the cone $C_{\bar z}$. Hence $C'_{\bar z}:=H \cap C_{\bar z}$ is a cone with vertex $\bar z$ strictly contained in $C_{\bar z}$ which satisfies
$$\mathbf F_{\bar z}=C'_{\bar z} \cap \O_p,$$
which implies that either $x$ or $y$ does not belong to $\overline\O_p$, a contradiction. This argument proves that $\HH^1(
\partial \mathbf F_{\bar z} \cap C_{\bar z}\cap J_u)=0$ and thus, that $u^+=u^-$ $\HH^1$-a.e. on $\partial \mathbf F_{\bar z}\cap C_{\bar z}$.
Since in particular $\HH^1((\partial \mathbf F_{\bar z}\cap \partial\O_p) \setminus (\partial \mathbf F_{\bar z}\cap C_{\bar z}))=0$ the result follows.

Finally, using the boundary condition, we get that $u=w$ $\HH^1$-a.e. on $\partial \mathbf F_{\bar z} \cap {\partial\O_p}\cap \partial\O$.
\end{proof}

\subsection{Outside the fans}

The complementary set to the boundary fans, {\it i.e.},
$$\mathscr C:= \O_p\setminus {\bigg(\bigcup_{\bar z\in \mathcal F} \mathbf F_{\bar z}\bigg)}$$
 is a closed set in the relative topology of $\O_p$.

\subsubsection{Topological structure of the complementary of the fans}

We first establish topological properties of the connected components of $\mathscr C$, which are closed in the relative topology of $\O_p$.

\begin{lemma}\label{lem:Cconvex}
Let $\mathbf C$ be a connected component of $\mathscr C$. Then $\mathbf C$ is convex and for all $x \in \mathbf C$,  $L_x \cap \O_p \subset \mathbf  C$.
\end{lemma}

\begin{proof}

The proof is divided into three steps.

\medskip

\noindent {\sf Step 1.} We show that, for all $x \in \mathscr C$, $L_x \cap \O_p \subset \mathscr C$.

Let $x \in \mathscr C$, and consider the characteristic line $L_x$ passing through $x$. Assume by contradiction that there exists $y \in L_x \cap \O_p$ such that $y \not\in \mathscr C$. Then $L_x=L_y$ according to {Theorem \ref{thm:lip}} and, since $y\notin \mathscr C$, there exists $\bar z \in \mathcal F$ such that $y\in \mathbf F_{\bar z}$. Thus $x \in L_x \cap \O_p=L_y \cap \O_p \subset \mathbf F_{\bar z}$
which is impossible.

\medskip

\noindent {\sf Step 2.} We show that, for all $x \in \mathbf C$, $L_x \cap \O_p \subset \mathbf  C$.

Consider the characteristic line $L_x$ passing through $x$. We already know from the previous step that $L_x \cap \O_p \subset \mathscr C$. We assume by contradiction  that there exists $y \in L_x \cap \O_p$ such that $y \not\in \mathbf C$. Let $C_y$ be the connected component of $\mathscr C$ which contains $y$. We distinguish two cases:
\begin{itemize}
\item If the segment $[x,y] \subset \mathscr C$, then the closed set $C'=\mathbf C \cup [x,y] \cup C_y$ is connected (because $[x,y]\cap \mathbf C
\neq \emptyset$ and $[x,y] \cap C_y \neq \emptyset$), $C' \subset \mathscr C$ and $x$, $y \in C'$. Then $C'$ is { a connected subset} of  $\mathscr C$ which strictly contains $\mathbf C$,  a contradiction.
\item If there is $z \in [x,y]$ such that $z \not\in \mathscr C$, then we
can find $\bar z \in \mathcal F$  such that $z \in \mathbf F_{\bar z}$, which implies, by definition of a fan, that the line $L_x=L_z$ must pass
through the vertex $\bar z$ of the fan $\mathbf F_{\bar z}$ and $L_x \cap
\O_p \subset \mathbf F_{\bar z}$. In particular the point $x$ belongs to $\mathbf F_{\bar z}$ which is again impossible since $x \in \mathscr C$.
\end{itemize}

\medskip

\noindent {\sf Step 3.} We show that $\mathbf C$ is convex.

Let $x$ and $y \in \mathbf C$, and let us consider a point $z \in [x,y]$.
Note that since $\O_p$ is convex, then $z \in \O_p$ and it makes sense to
consider its associated characteristic line $L_z$. If, for all $z' \in L_z \cap \O_p$, $z' \not\in \mathbf C$, it would then imply that $\mathbf C
\subset H^+ \cup H^-$, where $H^\pm$ are both open half planes separated by the line $L_z$. {Since both $x$ and $y$ belong to $\mathbf C$, the segment $[x,y]$ is not contained in $L_z$. Then, up to a permutation, $x \in
H^-$ and $y \in H^+$, so that $\mathbf C \cap H^- \neq \emptyset$ and $\mathbf C \cap H^+ \neq \emptyset$}. This last property contradicts the connectedness of $\mathbf C$. Therefore, there exists $z' \in L_z \cap \O_p$
such that $z' \in \mathbf C$, and, by Step 2, we have $L_z \cap \O_p = L_{z'} \cap \O_p \subset \mathbf C$ which implies that $z \in \mathbf C$.
The proof of the lemma is  complete.
\end{proof}

The following result shows that connected components of $\mathscr C$ with
empty interior are actually characteristic lines.

\begin{proposition}\label{prop:emptyint}
Let $\mathbf C$ be a connected component of $\mathscr C$ be such that $\mathring{\mathbf C}=\emptyset$, then $\mathbf C=L_x \cap \O_p$ for some $x \in \mathbf C$.
\end{proposition}

\begin{proof}
Since $\mathbf C \neq \emptyset$ there exists $x \in \mathbf C$ and, by Lemma \ref{lem:Cconvex}, $L_x \cap \O_p \subset \mathbf C$. By convexity of $\O_p$, the line $L_x$ must intersects $\partial \O_p$ at two distinct points denoted by $x'$ and $x''$, and thus, $]x',x''[\, =L_x \cap \O_p \subset \mathbf C$. Assume by contradiction that there is $y \in \mathbf C\setminus (L_x \cap \O_p)$. The same argument shows that $L_y \cap \O_p=\, ]y',y''[\, \subset \mathbf C$. Since $\mathbf C$ is closed in $\O_p$ and convex, and since $x'$, $x''$, $y'$ and $y'' \in \mathbf C$ then $\overline{\mathbf C}$ contains the closed and convex hull of $\{x',x'',y',y''\}$ denoted by ${\rm conv}\{x',x'',y',y''\}$. Thus $\mathring{\mathbf  C}
\supset {\rm int} \;{\rm conv}\{x',x'',y',y''\} \neq \emptyset$ which is against the hypothesis, and finally $\mathbf C=L_x \cap \O_p$.
\end{proof}

{
\begin{remark}
Note that there might be uncountably many such characteristic lines corresponding to connected components of $\mathscr C$ with empty interior.
\hfill\P
\end{remark}
}

We next focus on the (countably many) connected components $\mathbf C$ of
$\mathscr C$ with non empty interior.  Since $\sigma$ is constant on each
characteristic line inside $\mathbf C$, we can naturally extend $\sigma$ to the part of boundary of $\mathbf C$ which is reached by a characteristic coming from the interior of $\mathbf C$, {\it i.e.}, $\partial \mathbf
C \cap \bigcup_{z \in \mathbf C} L_{z}$.  Specifically, we set
$$\sigma(x)=\sigma(z), \quad L_x:=L_{z}  \quad \text{ if }x \in \partial \mathbf C \cap L_{z},$$
in such a way that $\sigma$ is constant on each $L_{z}\cap \overline{\mathbf C}$.
The value of $\sigma$ at $x$ is unambiguous because there cannot be more than one characteristic line $L_z$ coming from the interior of $\mathbf C$ and passing through $x$, lest $x$ be the apex of a boundary fan {contained in $\mathbf C$}.

\subsubsection{Characteristic boundary points}

We  introduce the set of characteristic boundary points which are the points in $\partial\mathbf C$  not crossed by a characteristic line coming from the interior of $\mathbf C$.

\begin{definition}
Let $\mathbf C$ be a connected component of $\mathscr C$ be such that $\mathring{\mathbf C}\neq \emptyset$. We say that $x \in \partial \mathbf C \cap \partial \O_p$ is a {\it characteristic boundary point} of $\mathbf C$ if $x \not\in L_z$ for all $z\in{\mathbf C}$. We denote by $\partial^c
\mathbf C$ the set of all characteristic boundary points.
\end{definition}

The following result formalizes the idea that the set $\partial^c \mathbf
C$ of all characteristic boundary points of $\mathbf C$ is made of points
where the stress is normal to $\partial\mathbf C$, or, equivalently, where the characteristics are tangential to $\partial\mathbf C$. Note that $\sigma\in H^1_{\rm loc}(\O;\R^2) \cap L^\infty(\O;\R^2)$ so that, through a sequence of smooth approximations,  $\sigma$ is easily seen to be well defined on $\O \cap \partial\O_p$ as an element of ${H^{1/2}_{\rm loc}(\O \cap \partial\O_p;\R^2) \cap} L^\infty(\O \cap \partial\O_p;\R^2)$ (see for example \cite[Lemma 2.3]{DMDSM}). On the other hand, strictly speaking,
$\sigma$ might not be defined on $\partial \O \cap \partial \O_p$. However, since $\sigma \in H(\Div,\O) \cap L^\infty(\O;\R^2)$, the normal trace
$\sigma\cdot\nu$ is well defined as an element of $H^{-1/2}(\partial\O;\R^2) \cap L^\infty(\partial\O;\R^2)$ through similar arguments.

\begin{proposition}\label{prop.bad-points}
Let $\mathbf C$ be a connected component of $\mathscr C$ be such that $\mathring{\mathbf C}\neq \emptyset$ and $x \in \partial \mathbf C \cap \partial \O_p$.
\begin{enumerate}
\item[(i)] If $x\in \partial^c \mathbf C$ is an accumulation point of  $(\partial \mathbf C \cap \partial \O_p)  \setminus \partial^c \mathbf C$,
denoting by $\Sigma(x)$ the set of all limits of sequences $(\sigma(x_n))_{n \in \mathbb N}$ where $(x_n)_{n \in \N} \subset \partial \mathbf C \cap \partial \O_p \setminus \partial^c \mathbf C$ is such that $x_n \to x$, then
 $$\Sigma(x) \subset  {\mathscr N}_{\mathbf C}(x) \cup (- {\mathscr N}_{\mathbf C}(x)).$$
\item[(ii)] Assume that  the relative interior $(\partial \mathbf C \cap \partial \O_p)^\circ$  of  $\partial \mathbf C \cap \partial \O_p$ is not
empty and take  $x$  in that set. If $x$  is not an accumulation point of
$(\partial \mathbf C \cap \partial \O_p)  \setminus \partial^c \mathbf C$, then $x \in \partial^c \mathbf C$ and there exists a maximal open set $U_x$ containing $x$ such that  $S_x:=(\partial \mathbf C \cap \partial \O_p) \cap U_x$ is an open line segment with $S_x \subset \partial^c \mathbf C$. In addition $\sigma\cdot\nu=1$ or $\sigma\cdot\nu=-1$ $\mathcal H^1$-a.e. on $S_x$.
\end{enumerate}
\end{proposition}

\begin{proof}
{\sf Step 1.} {We first assume that $x \in \partial^c \mathbf C$ is an accumulation point of $(\partial \mathbf C \cap \partial \O_p) \setminus \partial^c \mathbf C$ and consider $\xi \in \Sigma(x)$. It means that there exists a sequence $(x_n)_{n \in \mathbb N}$ in $(\partial \mathbf C \cap \partial \O_p)  \setminus \partial^c \mathbf C$ such that $x_n \to x$ and $\sigma(x_n) \to \xi$. By convexity of $\mathbf C$, for each $n \in \N$, there exists another point $y_n\neq x_n$ in $\partial \mathbf C \cap L_{x_n}$. Up to a subsequence, we can assume that $y_n \to y \in \partial \mathbf C$. Let us distinguish two possibilities:

{\it Case I}. Assume first that $x \neq y$ and that there exists $\delta>0$ such that for each $n \in \mathbb N$,
$$\max_{z \in [x_n,y_n]}{\rm dist}(z,\partial \mathbf C \cap \partial \O_p) \geq \delta>0.$$
Therefore, there is $z_n \in [x_n,y_n]$ such that ${\rm dist}(z_n,\partial \mathbf C \cap \partial \O_p) \geq \delta$ and, by compactness, up to a
further subsequence, we can assume that $z_n \to z$ for some $z \in \mathbf C$ with ${\rm dist}(z,\partial \mathbf C \cap \partial \O_p) \geq \delta$. Since $\sigma$ is constant along  $L_{x_n}$, we can write
\begin{equation}\label{eq:xn}
x_n=z_{n} +\theta_n \sigma^\perp(z_{n})
\end{equation}
for some $\theta_n \in \R$. Using that $|\sigma(z_n)|=1$, up to a further subsequence we can also assume that $\theta_n \to \theta$ so that, passing to the limit in \eqref{eq:xn} and using the continuity of $\sigma$ in $\O_p$ yields $x=z+\theta\sigma^\perp(z),$ hence $x \in L_z$ which is against  $x \in \partial^c \mathbf C$. It shows that this first possibility never occurs.

{\it Case II}. We next suppose that $x \neq y$ and that, for a subsequence
$$\max_{z \in [x_n,y_n]}{\rm dist}(z,\partial \mathbf C \cap \partial \O_p) \to 0.$$
Since $x_n \to x$ and $y_n \to y$, then $[x_n,y_n]\to [x,y]$ in the sense
of Hausdorff. Therefore, for all $z \in [x,y]$, there is $z_n \in [x_n,y_n]$ such that $z_n \to z$. In particular, ${\rm dist}(z_n,\partial \mathbf C \cap \partial \O_p) \to 0$ which leads to $z \in \partial \mathbf C \cap \partial \O_p$. This implies that $[x,y] \subset \partial \mathbf C \cap \partial \O_p$. Since $\sigma(x_n)$ is orthogonal to $[x_n,y_n]$ for each $n \in \mathbb N$, we deduce that $\xi$ is orthogonal to $[x,y]$, hence $\xi$ belongs to $ {\mathscr N}_{\mathbf C}(x) \cup (- {\mathscr N}_{\mathbf C}(x))$.

{\it Case III}. We finally assume that $x=y$.} Because of the convexity
of $\mathbf C$, its boundary $\partial\mathbf C$ is Lipschitz continuous and thus, there exist $r>0$ and a $L$-Lipschitz function $f:\R\to \R$ such that, in a suitable coordinate system,
$$\begin{cases}
\mathbf C \cap B_r(x)=\{(s_1,s_2) \in B_r(x) : \; s_2<f(s_1)\},\\
\partial \mathbf C \cap B_r(x)=\{(s_1,s_2) \in B_r(x) : \; s_2=f(s_1)\}.
\end{cases}$$
Since $\lim_n x_n=\lim_n y_n =x$, we can assume that $n$ is large enough so that $x_n$ and $y_n \in B_r(x)$, hence $x_n=(s_n,f(s_n))$ and $y_n=(t_n,f(t_n))$.
Let $H_n$ be a half plane such that $\partial H_n =L_{x_n}$ and {$H_n$ does not contain the portion of $\partial \mathbf C$ in between $x_n$ and
$y_n$. Let us denote by $\xi_n$ be the unit exterior normal to $H_n$} and
let us fix $y \in \mathbf C$ (see Figure \ref{fig:Hn}).
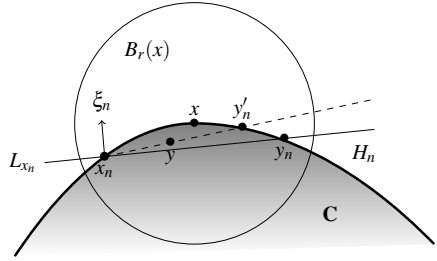
\begin{figure}[htbp]
\begin{center}
\scalebox{.8}{\begin{tikzpicture}
\begin{scope}
\shade[top color=gray, bottom color=white]
(-3,-2.2) parabola bend (0,0)  (4,-2);
\end{scope}
\draw[style=very thick] (-3,-2.2) parabola bend (0,0)  (4,-2);

\draw (0,0) node{{\small $\bullet$}};
\draw (0,0) node[above]{\small $x$};

\draw (-1.5,-0.55) node{$\bullet$};
\draw (-1.5,-0.55) node[below]{\small $x_n$};

\draw (1.5,-0.25) node{{\small $\bullet$}};
\draw (1.5,-0.25) node[below]{\small $y_n$};

\draw(-2.5,-0.65) -- (3,-0.1);
\draw (-2.5,-0.65)  node[left]{{\small $L_{x_n}$}};
\draw (2.5,-0.5) node[right]{{\small $H_n$}};

\draw (0.8,-0.07) node{{\small $\bullet$}};
\draw (0.8,-0.07) node[above]{\small $y'_n$};

\draw (-0.4,-0.32) node{\small{$\bullet$}};
\draw (-0.4,-0.32) node[below]{\small $y$};

\draw[dashed] (-1.5,-0.55) -- (3,0.4) ;

\draw (0,0) circle (2cm);

\draw (2,-1.5) node[right]{{\small $\mathbf C$}};

\draw (-1.3,1.2) node[right]{{\small $B_r(x)$}};

\draw[->] (-1.5,-0.55) -- (-1.55,0.05);
\draw (-1.55,0.05) node[above]{{\small $\xi_n$}};
\end{tikzpicture}}
\caption{{\small The case where $x \in \partial \mathbf C \cap \partial \O_p$ is an accumulation point of $(\partial \mathbf C \cap \partial \O_p)
 \setminus \partial^c \mathbf C$, and $x=\lim_n x_n=\lim_n y_n$ where
$x_n\not\in \partial^c \mathbf C$ and $x_n \neq y_n\in L_{x_n}\cap \partial \mathbf C$}}
\label{fig:Hn}
\end{center}
\end{figure}
Clearly, if $y \in \mathbf C \cap H_n$, then $\xi_n \cdot (y-x_n) \leq 0$. On the other hand, if $y \in \mathbf C \setminus H_n$ we consider the
point $y'_n \in \partial \mathbf C \cap {(x_n+ \R^*_+ (y-x_n))}$ which can be written as $y'_n=(t'_n,f(t'_n))$. Then, using the convexity of $\mathbf C$, we get that
\begin{eqnarray*}
\xi_n\cdot (y-x_n) & \leq & \xi_n\cdot (y'_n-x_n) \leq |y'_n - x_n| =|(t'_n,f(t'_n)) - (s_n,f(s_n))|\\
& \leq & \sqrt{1+L^2}|t'_n-s_n|\leq \sqrt{1+L^2}|x_n-y_n|\to 0.
\end{eqnarray*}
Thus, for all $y \in \mathbf C$, we have $\xi_n \cdot (y-x_n) \leq \sqrt{1+L^2}|x_n-y_n|$. Denoting by $\tilde \xi$ any accumulation point of the sequence $(\xi_n)_{n \in \mathbb N}$, we infer that
$$\tilde \xi \cdot (y-x)\leq 0 \quad \text{ for all }y \in \mathbf C,$$
which means that $\tilde \xi \in  {\mathscr N}_{\mathbf C}(x)$. Since $\sigma(x_n)=\pm \xi_n$, the previous argument shows that $\xi = \pm \tilde \xi$ and thus that $\Sigma(x) \subset  {\mathscr N}_{\mathbf C}(x) \cup (- {\mathscr N}_{\mathbf C}(x))$.

\medskip

{\sf Step 2.} If $x$ is not an accumulation point of $(\partial \mathbf C
\cap \partial \O_p)  \setminus \partial^c \mathbf C$,   for $R>0$ small
enough,  $\overline B_R(x) \cap (\partial \mathbf C \cap \partial \O_p)\setminus \{x\}  \subset \partial^c \mathbf C$. By convexity of $\mathbf C$
and $\O_p$ and since $x\in(\partial \mathbf C \cap \partial \O_p)^\circ$,
$R$ can also be chosen such that there exist two distinct points $x'$ and
$x''$ such that $ (\partial \mathbf C \cap \partial \O_p)  \cap \partial B_R(x)=\{x',x''\}$. Let $C_R$ be the convex set delimited by the segment $[x',x'']$ and $B_R(x) \cap (\partial \mathbf C \cap \partial \O_p)$ (see Figure \ref{fig:CR}). If $\overline B_R(x) \cap (\partial \mathbf C \cap \partial \O_p) \neq [x',x'']$, then $C_R$ has nonempty interior and,  for every $y \in \mathring C_R$, the characteristic line $L_y$ passing through $y$ must intersect $\overline B_R(x) \cap (\partial \mathbf C \cap \partial \O_p)$. Since, except for the point $x$, this set is contained in $\partial^c \mathbf C$, the only possible intersection point is $x$. But then $x$ would be the apex of a boundary fan {contained in $\mathbf C$}, which is impossible.

As a consequence, we must have that $\overline B_R(x) \cap (\partial \mathbf C \cap \partial \O_p) = [x',x'']$ is a {closed line} segment. Note that $x \in [x',x'']$ must belong to $\partial^c \mathbf C$. Indeed, if not, let $L_z$ be a characteristic line such that $x \in L_z$ for some $z \in \mathbf C$. Let us consider a triangle $T$ with apexes $x$, $a \in L_z$ and $b \in \,]x,x'[$ in such a way that $\mathring T \subset \mathring{\mathbf C}$ (see Figure \ref{fig:flat}). Then for any $y \in \mathring T$, the associated characteristic line $L_y$ must pass through the point $x$ (because $]x,a[\, \subset L_z$ and $]x,b[\, \subset\, ]x,x'[\subset \partial^c \mathbf C$). Then $x$ would be the apex of a boundary fan {contained in $\mathbf C$} which is impossible. We have thus established that the full segment $[x',x'']$ is contained in $\partial^c \mathbf C$. The maximal open set $U_x$ is obtained by taking the union of all open balls $B_R(x)$ such that $B_R(x) \cap (\partial^c \mathbf C \cap \partial \O_p)$ is a segment.

\begin{figure}[hbtp]
\begin{tabular}{cc}
\begin{minipage}[3cm]{.50\linewidth}
\scalebox{.8}{\begin{tikzpicture}

\begin{scope}
\shade[top color=gray, bottom color=white]
(-3,-2.2) parabola bend (0,0)  (3,-2);
\end{scope}

\draw[style=very thick] (-3,-2.2) parabola bend (0,0)  (3,-2);

\draw (0,0) node{{\small $\bullet$}};
\draw (0,0) node[above]{\small $x$};

\draw (-1.84,-0.81) node{$\bullet$};
\draw (-1.84,-0.81) node[below]{\small $x'$};

\draw (1.86,-0.78) node{{\small $\bullet$}};
\draw (1.86,-0.78) node[below]{\small $x''$};

\draw(-1.84,-0.81) -- (1.86,-0.78) ;

\draw (0,0) circle (2cm);

\draw (2,-1.5) node[right]{{\small $\mathbf C$}};

\draw (-0.2,-0.3) node[right]{{\small $C_R$}};

\draw (-1.3,1.2) node[right]{{\small $B_R(x)$}};
\end{tikzpicture}}
\caption{{\small The case where $\mathring{C}_R\neq\!\emptyset$ leading to a contradiction}}
\label{fig:CR}
\end{minipage}

&

\begin{minipage}[3cm]{.50\linewidth}
\scalebox{.8}{\begin{tikzpicture}

\begin{scope}
\shade[top color=gray, bottom color=white]
(-3,-2.2) parabola bend (0,0)  (2.8,-1);
\end{scope}

\draw[style=very thick] (-3,-2.2) parabola bend (0,0)  (2.8,-1);

\begin{scope}
\shade[top color=white, bottom color=white]
(-1.84,-0.81) parabola bend  (1.95,2) (1.95,-0.47);
\end{scope}

\draw (0,-0.65) node{{\small $\bullet$}};
\draw (0,-0.65) node[above]{\small $x$};

\draw (-1.84,-0.81) node{$\bullet$};
\draw (-1.84,-0.81) node[below]{\small $x'$};

\draw (1.95,-0.47) node{{\small $\bullet$}};
\draw (1.95,-0.47) node[below]{\small $x''$};

\draw[style=very thick] (-1.84,-0.81) -- (1.95,-0.47) ;

\draw (0,0) circle (2cm);

\draw (2,-1.5) node[right]{{\small $\mathbf C$}};

\draw (-1.3,1.2) node[right]{{\small $B_R(x)$}};

\draw(-2,-2.65) -- (2,1.35) ;
\draw (2,1.35)node[right]{{\small $L_z$}};
\draw (-1.4,-2.05) node{{\small $\bullet$}};
\draw  (-1.4,-2.05)  node[below]{\small $z$};

\draw (-0.5,-1.15) node{{\small $\bullet$}};
\draw  (-0.5,-1.15) node[right]{\small $a$};

\draw (-1,-0.73) node{{\small $\bullet$}};
\draw (-1,-0.73) node[above]{\small $b$};

\draw (-1,-0.73) -- (-0.5,-1.15) ;

\draw (-0.72,-0.87) node[right]{{\small $T$}};
\end{tikzpicture}}
\caption{{\small The case where $\mathring{C}_R= \emptyset$, hence $[x',x'']\subset \partial \mathbf C \cap \partial \O_p$}}
\label{fig:flat}
\end{minipage}

\end{tabular}
\end{figure}

In order to show that $\sigma\cdot\nu=1$ or $\sigma\cdot\nu=-1$ $\mathcal H^1$-a.e. on $S_x$, we assume without loss of generality that $S_x=\{0\} \times (-1,1)$ and that $\nu=e_1$ in such a way that $\mathbf C \subset \R^-\times\R$. Since $S_x \subset \partial^c \mathbf C$, then no characteristic line from a point in $\mathbf C$ intersects $S_x$. Thus, if
$\delta\in (0,1/2)$, $t \in (0,1)$ and $y \in (-1+\delta,1-\delta)$, then $(-t \e,y) \in \mathbf C$ for $\e>0$ small enough, and the characteristic line $L_{(-t\e,y)}$ passing through the point $(-t\e,y)$ intersects the vertical line $\{0\} \times \R$ at a point $(0,y^\e_0)$ with $|y^\e_0| \geq 1$, or else
$\sigma^\perp(-t \e,y)=\pm e_2$. Since $(-t\e,y)$ and $(0,y^\e_0) \in L_{(-t\e,y)}=(-t \e,y) + \R \sigma^\perp(-t\e,y)$, then
$$\sigma^\perp(-t \e,y)=\frac{\pm 1}{\sqrt{t^2\e^2 + (y^\e_0-y)^2}}
\begin{pmatrix}
t \e \\ y^\e_0-y
\end{pmatrix} \mbox{ or } \pm e_2\to \pm e_2$$
as $\e\to 0$.

If $\sigma^\perp(-t \e,y)=\pm e_2$ for some $y$, then $\sigma^\perp(-t \e,z)=\pm e_2$ for all $z\in (-1+\delta,1-\delta)$ because $(-t\e,z)\in L_{(-t\e,y)}$. Assume that there exist $y\neq z\in (-1+\delta,1-\delta)$ such that $\sigma^\perp(-t\e,y) \to e_2$ and $\sigma^\perp(-t\e,z) \to -e_2$ as $\e\to 0$. Then, denoting by $(0,y_0)$ (resp. $(0,z_0)$) the intersection point of $L_{(-t\e,y)}$ (resp. $L_{(-t\e,z)}$) with the vertical line $\{0\} \times \R$, it means that $y_0 \geq 1$ and $z_0 \leq -1$. We thus get that both characteristic lines $L_{(-t\e,y)}$ and $L_{(-t\e,z)}$ intersect at a
point $x=(x_1,x_2)$ with $-1+\delta \leq x_2 \leq 1-\delta$ and $-\frac{t \e}{\delta}\leq x_1\leq 0$ (see Figure \ref{fig:inter}).

\begin{figure}[ht]
\scalebox{.8}{\begin{tikzpicture}
\draw[style=very thick] plot[smooth cycle] coordinates {(0,-2) (-3,-1) (-4,0) (-3.5,1.5) (0,2)};
\fill[color=gray!35]  plot[smooth cycle] coordinates {(0,-2) (-3,-1) (-4,0) (-3.5,1.5) (0,2)};
\fill[color=white]  plot coordinates {(0,-2) (1,-2) (1,2) (0,2)};

\draw[style= thick] (0,-2) -- (0,2) ;
\draw[style=dashed] (0,2) -- (0,3) ;
\draw[style=dashed] (0,-2) -- (0,-3) ;

\draw (0,2.7) node{{\small $\bullet$}};
\draw (0,2.7) node[right]{{\small $y_0$}};

\draw (0,2) node{{\small $\bullet$}};
\draw (0,2) node[right]{{\small $(0,1)$}};

\draw (0,0) node{{\small $\bullet$}};
\draw (0,0) node[right]{{\small $(0,0)$}};

\draw (0,-2) node{{\small $\bullet$}};
\draw (0,-2) node[right]{{\small $(0,-1)$}};

\draw (0,-2.3) node{{\small $\bullet$}};
\draw (0,-2.3) node[right]{{\small $z_0$}};

\draw(0.05,3) -- (-1,-3) ;
\draw (-0.2,1.5) node{{\small $\bullet$}};
\draw (-0.1,1.5) node[right]{{\small $(-\e t,y)$}};
\draw (-1,-3) node[below]{{\small $L_{(-\e t,y)}$}};

\draw (0.11,-3) -- (-0.8,3) ;
\draw (-0.2,-1) node{{\small $\bullet$}};
\draw (-0.1,-1) node[right]{{\small $(-\e t,z)$}};
\draw (-0.8,3) node[above]{{\small $L_{(-\e t,z)}$}};

\draw (-0.4,0.35) node{{\small $\bullet$}};
\draw (-0.4,0.35) node[left]{\small $x$};

\draw (-3,0) node[right]{{\small $\mathbf C$}};

\end{tikzpicture}}
\caption{{\small }}
\label{fig:inter}
\end{figure}

Thus, for $\e>0$ small enough we infer that $x \in \mathbf C$ which is impossible according to Proposition \ref{prop:vortex}.

As a consequence, we always have that
\begin{equation*}\label{eq:sigma.to.e1}
\sigma((0,y) - \e t e_1)\cdot e_1 \to 1 \quad \text{ or}\quad \sigma((0,y) - \e t e_1)\cdot e_1 \to -1
\end{equation*}
for all $(t,y) \in (0,1) \times (-1+\delta,1-\delta)$. Thus in any case, according to the dominated convergence theorem,
\begin{multline}\label{eq:conv2}\lim_{\e \to 0} \int_0^1 \int_{-1+\delta}^{1-\delta} |\sigma((0,y) -\e t e_1)\cdot e_1 - 1|\, dy\, dt=0 \\
\text{ or }\quad \lim_{\e \to 0} \int_0^1 \int_{-1+\delta}^{1-\delta} |\sigma((0,y) -\e t e_1)\cdot e_1 + 1|\,  dy\, dt=0.
\end{multline}
On the other hand, using \eqref{eq:sigmanu}, we have for all $\varphi \in
\C^\infty_c((-1,1))$,
\begin{equation}\label{eq:conv1}
\int_0^1 \int_{-1}^1 [\sigma((0,y) -\e t e_1)\cdot e_1 - (\sigma\cdot\nu)(0,y)]\varphi(y)\, dy \to 0.
\end{equation}
Gathering \eqref{eq:conv2} and \eqref{eq:conv1} yields $\sigma\cdot\nu=1$ or  $\sigma\cdot\nu=-1$ $\HH^1$-a.e. in $S_x$.
\end{proof}

\subsubsection{Topological structure of the boundary of the connected components}

We now establish a result concerning the topological structure of $\partial\mathbf C$ which can be split as the disjoint union of characteristic lines and continuous curves.

\begin{proposition}\label{prop.boundary-inside}
Let $\mathbf C$ be a connected component of $\mathscr C$ be such that $\mathring{\mathbf C}\neq \emptyset$ and $\mathbf C\ne \Omega_p$. Then, $\partial \mathbf C\cap\Omega_p\ne\emptyset$ and there exist an (at most) countable set $J \subset \N$ and, for all $j \in J$, distinct points $\hat x_j \in \O_p$ such that
$$\partial \mathbf C \cap \O_p=\bigcup_{j \in  J}(L_{\hat x_j} \cap \O_p).$$
Moreover $\partial \mathbf C \cap \partial \O_p \neq \emptyset$ and there
exist a (possibly uncountable) set $I$ and, pairwise disjoint continuous curves $\{\Gamma_i\}_{i \in I}$ such that
$$\partial\mathbf C \cap \partial \O_p=\bigcup_{i \in I} \Gamma_i,$$
for some Lipschitz continuous mappings $\gamma_i:[0,1] \to \R^2$ with $\Gamma_i=\gamma_i([0,1])$.
\end{proposition}

\begin{proof}
Note that $\partial \mathbf C\cap\Omega_p\ne\emptyset$ otherwise $\Omega_p=\mathbf C$. If $x \in \partial \mathbf C \cap \O_p$, then there exists $r>0$ such that $B_r(x) \subset \O_p$. Moreover, since $\mathbf C$ is closed in $\O_p$, then $x \in \mathbf C$, hence $L_x \cap \O_p \subset \mathbf C$ by Lemma \ref{lem:Cconvex}.
If $L_x \cap \mathring{\mathbf C} \neq \emptyset$ then, by convexity of $\mathbf C$, $L_x$ is not contained in the tangential cone to $\partial \mathbf C$ at $x$ and thus
 $L_x \cap B_r(x) \setminus \mathbf C \neq \emptyset$, which is impossible according to Lemma \ref{lem:Cconvex}. Thus $L_x \cap \mathring{\mathbf C}=\emptyset$ and $L_x \cap \O_p \subset \partial \mathbf C$. Since $\mathbf C$ is bounded and convex,  its boundary is Lipschitz continuous (see Propositions 2.4.4 and 2.4.7 in \cite{HP}) and $\HH^1(\partial \mathbf C)<+\infty$. Using that $L_x \cap \O_p$ is a nonempty open segment,  $\HH^1(L_x \cap \O_p)>0$ and thus, there are at most countably many such characteristic lines. Therefore,
$$\partial \mathbf C \cap \O_p=\bigcup_{j \in J}(L_{\hat x_j} \cap \O_p),$$
for some nonempty (at most) countable set $J \subset \N$ and some distinct points $\hat x_j \in \O_p$, for all $j \in J$. Note that by convexity, for all $j \in J$, the set $\mathbf C$ is contained in one of the half planes delimited by $L_{\hat x_j}$ for each $j\in J$.

\medskip

Remark that $\partial \mathbf C\cap\partial\Omega_p\ne\emptyset$ because,
since there exists $x\in{\mathbf C}\ne \emptyset$, $L_x\cap\Omega_p\subset {\mathbf C}$ and $L_x$ intersects $\partial\Omega_p$ at points which belong to $\partial \mathbf C$. Consider the connected components $\{\Gamma_i\}_{i \in I}$ of $\partial \mathbf C\cap\partial\Omega_p$. Note that for all $i \in I$, $\Gamma_i$ is a  closed, connected set with finite $\HH^1$ measure, therefore, according to \cite[Proposition C-30.1]{david}, we
infer that $\Gamma_i$ is arcwise connected and that there exists a Lipschitz continuous mapping $\gamma_i:[0,1] \to \R^2$ such that $\Gamma_i=\gamma_i([0,1])$ (with possibly $\gamma_i(0)=\gamma_i(1)$).
\end{proof}

The previous result motivates the following definition.

\begin{definition}
Let $\mathbf C$ be a connected component of $\mathscr C$ be such that $\mathring{\mathbf C}\neq \emptyset$ and $\mathbf C\ne \Omega_p$. The set $\partial \mathbf C \cap \partial  \O_p$ is called the exterior boundary of
$\mathbf C$ and the set $\partial\mathbf C \cap \partial \O_p \setminus \partial^c \mathbf C$ is called the non-characteristic exterior boundary of $\mathbf C$.
\end{definition}

The next result shows that the topological structure of $\mathbf C$ is severely constrained.
\begin{figure}[htbp]
\begin{tabular}{cc}
\begin{minipage}[1cm]{.50\linewidth}
\scalebox{.8}{\begin{tikzpicture}
\fill[color=gray!35]  plot[smooth cycle] coordinates {(-1,-1) (3,0) (3,2) (0,3) (-1,2)};
\fill[color=white]  plot coordinates { (0,-2) (-1,3.5) (-2,2.5) (-2,-2)
};
\fill[color=white]  plot coordinates {(1.5,-2) (4.5,-2) (4.5,3.5) (2,3.5)};

\draw plot[smooth cycle] coordinates {(-1,-1) (3,0) (3,2) (0,3) (-1,2)};

\draw[style=thick] (0,-2) -- (-1,3.5)  ;
\draw (0.1,-2) -- (-0.3,3.5) ;
\draw (0.2,-2) -- (0.3,3.5) ;
\draw (0.3,-2) -- (0.6,3.5) ;

\draw (0.6,-2) -- (0.7,3.5) ;
\draw (1,-2) -- (0.8,3.5) ;
\draw (1.2,-2) -- (1.2,3.5) ;
\draw (1.4,-2) -- (1.7,3.5) ;

\draw[style=thick] (1.5,-2) -- (2,3.5) ;

\draw (0.5,1) node[right]{\small $\mathbf C$};
\draw (2.5,0.8) node[right]{\small $\Omega_p$};
\end{tikzpicture}}
\caption{\small The case $\#(I)=2$}
\label{fig:C1}
\end{minipage}

&

\begin{minipage}[2cm]{.50\linewidth}
\scalebox{.8}{\begin{tikzpicture}
\fill[color=gray!35]  plot[smooth cycle] coordinates {(-1,-1) (3,0) (3,2) (0,3) (-1,2)};
\fill[color=white]  plot coordinates {(0.45,-2) (6,-2) (6,3.5) (0.95,3.5)};

\draw plot[smooth cycle] coordinates {(-1,-1) (3,0) (3,2) (0,3) (-1,2)};

\draw[style=thick] (-1.04,-2) -- (-1.9,3.5)  ;
\draw (-0.95,-2) -- (-1.35,3.5) ;
\draw (-0.85,-2) -- (-0.75,3.5) ;
\draw (-0.75,-2) -- (-0.45,3.5) ;

\draw (-0.45,-2) -- (-0.35,3.5) ;
\draw (-0.05,-2) -- (-0.25,3.5) ;
\draw (0.15,-2) -- (0.15,3.5) ;
\draw (0.35,-2) -- (0.65,3.5) ;

\draw[style=thick] (0.45,-2) -- (0.95,3.5) ;

\draw (-0.5,0.8) node[right]{\small $\mathbf C$};
\draw (2.5,0.8) node[right]{\small $\Omega_p$};

\draw  (-1.27,-0.5) node{{\small $\bullet$}};
\draw (-1.31,-0.5)  node[left]{{\small $\partial^c \mathbf C \ni x$}};

\end{tikzpicture}}
\caption{\small  The case $\#(I)=1$}
\label{fig:C2}
\end{minipage}

\end{tabular}
\end{figure}In fact, its boundary $\partial\mathbf C$ contains at most two characteristic lines and $\partial\mathbf C \cap \partial \O_p$ has at most two connected components {(see Figures \ref{fig:C1} and \ref{fig:C2}
for illustration)}.

\begin{theorem} \label{thm:geom-struct}
Let $\mathbf C$ be a connected component of $\mathscr C$ be such that $\mathring{\mathbf C}\neq \emptyset$ and $\mathbf C \neq \O_p$. Then $\partial\mathbf C \cap \partial \O_p$ has at most two connected components and
\begin{itemize}
\item[i)] if $\partial \mathbf C\cap\partial \Omega_p$ has exactly two connected components, then  $\partial^c\mathbf C=\emptyset$ and all characteristic lines that intersect $\partial \mathbf C {\cap \partial\O_p}$ must intersect both connected components.
\item[ii)] if  $\partial \mathbf C\cap\partial \Omega_p$ is connected, then $\partial^c \mathbf C$ is either a point or a closed line segment contained in $\partial \mathbf C\cap\partial \Omega_p$,  $\partial \mathbf C\cap\partial \Omega_p \setminus \partial^c \mathbf C$ has two connected components, and all characteristic lines that intersect $\partial \mathbf C{\cap \partial\O_p}$ must intersect both connected components {of
$\partial \mathbf C\cap\partial \Omega_p \setminus \partial^c \mathbf C$}.
\end{itemize}
\end{theorem}

\begin{remark}
If $\mathbf C=\O_p$,   arguments similar to those used in the proof of Theorem \ref{thm:geom-struct} would show that

\noindent-- $\partial^c\O_p$ has two connected components which are both either a point or a closed line segment;

\noindent--  $\partial\O_p \setminus \partial^c\O_p$ has two connected components, and all characteristic lines that intersect $\partial \O_p$ must intersect both connected components {of $\partial\O_p \setminus \partial^c\O_p$}.\hfill\P
\end{remark}

The proof of Theorem \ref{thm:geom-struct} relies on several technical results. We first show that there can be at most two boundary characteristic line segments.

\begin{lemma}\label{lem:numberJ}
 $\#(J)= 1$ or $2$ in the notation of Proposition \ref{prop.boundary-inside}.
\end{lemma}

\begin{proof}
Since $\partial\mathbf C \cap \O_p \neq \emptyset$, it follows that $\#(J)\geq 1$.

\medskip

Assume now that $\#(J)\ge 3$ and let $L_1$, $L_2$, $L_3$ be three distinct interior boundary characteristic lines, that is such that $L_i\cap\O_p\subset\partial{\mathbf C}$. Note that $L_i\cap\partial\O_p=\{x_i,x'_i\}$, where both points $x_i$ and $x'_i$ lie in $\partial\O_p\cap\partial{\mathbf C}$, so that $L_i\cap\O_p=\, ]x_i,x'_i[$. Furthermore, because there can be no boundary or interior fans inside $\mathbf C$, the points $x_1$, $x'_1$, $x_2$, $x'_2$, $x_3$ and $x'_3$ are pairwise distinct and $L_i\cap L_j\cap\overline{\mathbf C}=\emptyset$ for $i\ne j$. By Lemma \ref{lem:Cconvex}, the middle points $y_i:={(x_i+x'_i)}/2$ belong to $L_i\cap\O_p\cap\mathbf C$ and consequently, $L_{y_i}=L_i$. Moreover, for $i \neq j$, the closed segments $[y_i,y_j] \subset \O_p$ cannot be contained in a characteristic line $L_x$, for some $x \in \O_p$, otherwise $L_x$
and $L_i$ (resp. $L_j$) would intersect at $y_i$ (resp. $y_j$), which is impossible in view of  Proposition \ref{prop:vortex}.

Now $\partial \mathbf C\, \setminus \, ]x_1,x'_1[$ is a closed, connected
set with finite $\HH^1$ measure, therefore, according to \cite[Proposition C-30.1]{david},  it is arcwise connected and there exists a Lipschitz continuous mapping $\gamma:[0,1] \to \R^2$ such that $\partial \mathbf C \, \setminus \, ]x_1,x'_1[\; =\gamma([0,1])$ with $\gamma(0)=x_1$ and
$\gamma(1)=x'_1$. Since $\partial \mathbf C=\gamma([0,1])\, \cup \, ]x_1,x'_1[$ with $]x_1,x'_1[ \subset L_1$, it follows from Proposition \ref{prop:vortex} that $L_2$ and $L_3$ intersect $\partial\mathbf C$ in $\gamma((0,1))$. Let us renumber the $L_i$ and exchange $x_i$ with $x'_i$ if necessary so that
$$ x_2=\gamma(s_2),\; x_3=\gamma(s_3),\; x'_3=\gamma(s'_3) \mbox{ with } 0<s_2<s_3<s'_3<1,$$
and consider $x'_2=\gamma(s'_2)$ for some $0<s'_2<1$.

If $s'_2 \in (s'_3,1)$, then the segment $]y_1,y_3[$ intersects $L_2$ at a point inside $\mathring{\mathbf C}$ (see Figure \ref{fig:3.1}), which contradicts the fact that $L_2\subset \partial{\mathbf C}$.

\begin{figure}[hbtp]
\begin{tabular}{cc}
\hskip-.5cm\begin{minipage}[1cm]{.5\linewidth}
\scalebox{.8}{\begin{tikzpicture}
\fill[color=gray!35]  plot[smooth cycle] coordinates {(-1,-1) (4,0) (3,2) (0,3) (-1,2)};
\fill[color=white]  plot coordinates {  (0,-2) (-1,3.5) (-3,3.5) (-3,-2)};
\fill[color=white]  plot coordinates {  (4.95,0.1) (0.2,3.5) (4.8,3.5) (5,-1) };

\draw plot[smooth cycle] coordinates {(-1,-1) (4,0) (3,2) (0,3) (-1,2)};
\draw[style=thick] (0,-2) -- (-1,3.5)  ;
\draw[style=thick] (2.09,-2) -- (0.09,3.5) ;
\draw[style=thick] (4.95,0.1) -- (0.2,3.5) ;
\draw[dashed](-0.45,0.6) -- (2.8,1.7);

\draw   (2.2,-2.2)      node{$L_2$};
\draw   (0.1,-2.2)      node{$L_1$};
\draw   (5.1,0)      node{$L_3$};

\draw  (-0.45,0.6) node{{\small $\bullet$}};
\draw   (-0.7,0.6)      node{$y_1$};
\draw  (2.75,1.7) node{{\small $\bullet$}};
\draw   (3.1,1.7)      node{$y_3$};
\draw  (-0.18,-1.02) node{{\small $\bullet$}};
\draw  (-0.4,-1.3) node{$x_1$};
\draw   (1.63,-0.75) node{{\small $\bullet$}};
\draw  (2,-0.9) node{$x_2$};
\draw   (4.06,0.75) node{{\small $\bullet$}};
\draw  (4.3,0.9) node{$x_3$};
\draw   (1.3,2.72) node{{\small $\bullet$}};
\draw  (1.3,3) node{$x'_3$};
\draw  (0.27,2.97)node{{\small $\bullet$}};
\draw  (-0.05,3.25) node{$x'_2$};
\draw  (-0.8,2.45) node{{\small $\bullet$}};
\draw  (-1.1,2.5) node{$x'_1$};
\draw (2,0.5) node[right]{\small $\mathbf C$};
\draw (-1.2,0) node[right]{\small $\Omega_p$};
\end{tikzpicture}}
\caption{\small  {The first case: $\!s'_2 \in (s'_3,1)$}}
\label{fig:3.1}
\end{minipage}

&

\begin{minipage}[1cm]{.5\linewidth}
\scalebox{.8}{\begin{tikzpicture}
\fill[color=gray!35]  plot[smooth cycle] coordinates {(-1,-1) (4,0) (3,2) (0,3) (-1,2)};
\fill[color=white]  plot coordinates {  (0,-2) (-1,3.5) (-3,3.5) (-3,-2)};
\fill[color=white]  plot coordinates {  (4.95,0.1) (0.2,3.5) (4.8,3.5) (5,-1) };
\fill[color=white]  plot coordinates {  (-1,-2) (4.95,.7) (4.8,3.5) (5,-2) };

\draw plot[smooth cycle] coordinates {(-1,-1) (4,0) (3,2) (0,3) (-1,2)};
\draw[style=thick] (0,-2) -- (-1,3.5)  ;
\draw[style=thick] (-1,-2) -- (4.95,.7) ;
\draw[style=thick] (4.95,0.1) -- (0.2,3.5) ;
\draw[dashed](-0.45,0.6) -- (2.8,1.7);

\draw   (0.7,3.5)      node{$L_3$};
\draw   (-1.2,-2)      node{$L_2$};
\draw   (-1.2,3.5)      node{$L_1$};
\draw  (-0.45,0.6) node{{\small $\bullet$}};
\draw   (-0.7,0.6)      node{$y_1$};
\draw  (2.75,1.7) node{{\small $\bullet$}};
\draw   (3.1,1.7)      node{$y_3$};

\draw  (-0.18,-1.02) node{{\small $\bullet$}};
\draw  (-0.4,-1.3) node{$x_1$};
\draw   (1.9,-0.7) node{{\small $\bullet$}};
\draw  (2,-0.9) node{$x_2$};
\draw   (4.06,0.75) node{{\small $\bullet$}};
\draw  (4.3,0.9) node{$x_3$};

\draw   (1.3,2.72) node{{\small $\bullet$}};
\draw  (1.3,3) node{$x'_3$};
\draw  (4.15,0.35) node{{\small $\bullet$}};
\draw  (4.38,0.2) node{$x'_2$};
\draw  (-0.8,2.45) node{{\small $\bullet$}};
\draw  (-1.1,2.5) node{$x'_1$};

\draw (1,0.5) node[right]{\small $\mathbf C$};
\draw (-1.2,0) node[right]{\small $\Omega_p$};
\end{tikzpicture}}
\caption{\small {The second case: $\!s'_2\!\in\!(s_2,s_3)$}}
\label{fig:3.2}
\end{minipage}
\end{tabular}
\end{figure}

Now $s'_2\notin (s_3,s'_3)$ because, {otherwise $L_2$ and $L_3$ would intersect at $\gamma(s'_2) \in \O_p$, which is impossible owing to Proposition \ref{prop:vortex}}.

So the other possibility is that $s'_2\in (s_2,s_3)$ (see Figure \ref{fig:3.2}). For any $t \in (0,1)$, define $y(t):=ty_3+(1-t)y_1 \in \; ]y_1,y_3[$. The intersection points of $L_{y(t)}$ with $\partial \O_p$, respectively denoted  by $\gamma(s_t)$ and $\gamma(s'_t)$, satisfy
$$s'_t\in (s'_3,1), \quad s_t \in (0,s_2) \cup(s'_2,s_3).$$
Define
\begin{eqnarray*}
\underline t& := & \sup\{t \in [0,1]: \; L_{y(t)}\cap\gamma([0,s_2])\neq \emptyset\},\\
 \bar t & := & \inf \{t \in [0,1]: \; L_{y(t)}\cap\gamma([s'_2,s_3]) \neq \emptyset\}.
\end{eqnarray*}
If, for all $y \in \; ]y_1,y_3[$, the characteristic line $L_y$ intersects $\gamma([s'_2,s_3])$, then, by continuity of $\sigma$ on $[y_1,y_3]$, we would have that $L_{y_1}\cap\gamma([s'_2,s_3])\neq \emptyset$ which is impossible since $L_{y_1}=L_1$ is disjoint from $\gamma([s'_2,s_3])$. Therefore the set $\{t \in [0,1]: \; L_{y(t)}\cap\gamma([0,s_2])\neq \emptyset\}$ is not empty and $\underline t>0$. A similar argument also shows that $\underline t<1$. Let $(t_n)_{n \in \N}$ be a maximizing sequence in
$[0,1]$ such that $L_{y(t_n)}\cap\gamma([0,s_2]) \neq \emptyset$ for all $n \in \N$ and $t_n \to \underline t$. Since $\gamma(s_{t_n}) \in L_{y(t_n)}=y(t_n)+\R \sigma^\perp(y(t_n))$, there exists $\theta_n \in \R$ such that
$$\gamma(s_{t_n})=y(t_n)+\theta_n \sigma^\perp(y(t_n)),$$
where $(\gamma(s_{t_n}))_{n \in \N}$ is a sequence in $\gamma([0,s_2])$ (hence bounded) and $(\theta_n)_{n \in \N}$ is a bounded sequence since $|\sigma^\perp(y(t_n))|=1$ for all $n\in \N$. Therefore, up to a further subsequence $\gamma(s_{t_n}) \to \underline x \in \gamma([0,s_2])$ and $\theta_n \to \theta$. Thus, using that  $\sigma$ is continuous in $[y_1,y_3]$ (because $[y_1,y_3] \subset \O_p$), it follows that
$$\underline x=y(\underline t)+\theta \sigma^\perp(y(\underline t)),$$
hence $\underline x \in L_{y(\underline t)} \cap \gamma([0,s_2])$. Note that since $\underline t \in (0,1)$, then $\underline y:=y(\underline t)
\in\; ]y_1,y_3[$. Moreover, $\underline x \neq \gamma(0)=x_1$ and $\underline x\neq \gamma(s_2)=x_2$ otherwise, $\underline x$ would be the apex of a boundary fan contained in $\mathbf C$, hence $\underline x \in L_{y(\underline t)} \cap \gamma((0,s_2))$. A similar argument shows that $\bar t \in (0,1)$, $\bar y:=y(\bar t) \in\, ]y_1,y_3[$ and $L_{\bar y} \cap \gamma((s'_2,s_3))\neq \emptyset$.

Since $s_2<s'_2$, then $\underline t < \bar t$, otherwise $L_{\underline y}$ and $L_{\bar y}$ would intersect inside $\mathring {\mathbf C}$ since
$L_{\underline y}$ intersects $\gamma((s_1,s_2))$ and  $L_{\bar y}$ intersects $\gamma((s'_2,s_3))$. But this is impossible by Proposition \ref{prop:vortex}.

Denote by $\underline H$ and $\overline H$ the open half-planes with boundary {$L_{\underline x}$} and $L_{\bar y}$ that do not contain the points $y_1$ and $y_3$ respectively. The  region $\mathbf C':=\mathbf C\cap \underline H\cap\overline H$ contains the characteristic lines $L_{y(t)}$ for all $t\in (\underline t,\bar t)$. Such a line cannot intersect {$L_{\underline x}$} and $L_{\bar y}$ {in $\mathbf C$} by Proposition \ref{prop:vortex}, and it cannot intersect the  connected boundaries $\gamma([s_1,s_2]) \cup \gamma([s'_2,s_3])$ by construction. The line $L_{y(t)}$ must therefore intersect $\gamma((s_2,s'_2))=\; ]x_2,x'_2[\;  \subset L_2$ (see Figure \ref{fig:3.2}), which is impossible according, once again, to Proposition \ref{prop:vortex}.
\end{proof}

\begin{remark}\label{rem:con.comp}
Lemma \ref{lem:numberJ} actually establishes, if $\mathbf D$ is any closed, convex subset  of $\mathscr C$ with $\mathring{\mathbf  D} \ne\emptyset$ and $\mathbf D\ne \O_p$ such that
$\partial{\mathbf D}\cap\O_p$ is a countable union of disjoint characteristic line segments, that is such that
$$
\partial \mathbf D \cap \O_p=\bigcup_{k \in  K}(L_{\hat x_k} \cap \O_p),
$$
for some countable set $K$ and distinct points $\hat x_k\in \O_p$, then {$\#(K)=1$ or $2$}.
\hfill\P\end{remark}

\medskip

Then, we show that $\partial\mathbf C \cap \partial \O_p$ has at most two
connected components as well, that those cannot reduce to a single point,
and that the extreme points lie on a characteristic line.

\begin{lemma}\label{lem:5.11}
$\#(I)=\#(J)\in \{1,2\}$ in the notation of Proposition \ref{prop.boundary-inside}. Moreover, for all $i \in I$, $\gamma_i(0) \neq \gamma_i(1)$.
 If $\#(J)=1$ then $\gamma_1(0)$ and $\gamma_1(1)$ belong to $L_1$ while if  $\#(J)=2$ then
$\gamma_1(0)$ and $\gamma_2(1)$ belong to one of the boundary characteristic lines, while $\gamma_1(1)$ and $\gamma_2(0)$ belong to the other one.
\end{lemma}

\begin{proof}
Assume first that $\#(J)=1$. Then the characteristic line $L_1:=L_{{\hat x}_1}$ intersects
$\partial\O_p$ at two points, and thus determine a unique connected component $\partial\mathbf C \setminus L_1=\partial\mathbf C\cap\partial\O_p$.

If $\#(J)=2$, then the two characteristic lines $L_1:=L_{{\hat x}_1}$
and $L_2:=L_{{\hat x}_2}$ are distinct. They cannot intersect at $a\in \partial\O_p$  otherwise $a$ would be the apex of a boundary fan contained in $\mathbf C$, which is not possible. They cannot intersect in $\mathring{\mathbf C}$ by virtue of Proposition \ref{prop:vortex}.
Thus they determine two disjoint connected components $\Gamma_1$ and $\Gamma_2$ for ${\partial\mathbf C \setminus (L_1\cup L_2)}=\partial\mathbf
C\cap\partial\O_p$.

The rest of the Lemma is a direct consequence of that geometry.
\end{proof}

We next show that, when $\#(J)=2$, no characteristic line can intersect
twice the same connected component of $\partial \mathbf C \cap \partial \O_p$.

\begin{lemma}\label{lem5.13}
If $\#(J)=2$, then all characteristic lines $L_x$ with $x \in \mathbf C$ intersect both $\Gamma_1$ and $\Gamma_2$.
\end{lemma}
\begin{proof}
Assume by contradiction that there is a connected component of $\partial \mathbf C \cap \partial \O_p$, say $\Gamma_1$, and a characteristic line $L$ that intersects $\Gamma_1$ at two distinct points, say $\bar a=\gamma_1(\bar s)$ and $\bar b=\gamma_1(\bar t)$ with $s<t$. First, $\bar s>0$ and $\bar t<1$, otherwise $\gamma_1(0)$ and/or $\gamma_1(1)$ would be the apex of a boundary fan {contained in $\mathbf C$}. Consider the closed hyperplane $H$ bounded by $L$ and containing both $\gamma(0)$ and $\gamma(1)$. Then $\mathbf C\cap H$ is a convex set in $\mathscr C$ which has three boundary characteristic line segments $L_1 \cap\O_p$, $L_2\cap\O_p$ and $L\cap\O_p$, in contradiction with Remark \ref{rem:con.comp}.
\end{proof}

Provided that $\partial \mathbf C \cap \partial \O_p$ possesses two connected components,  there are no characteristic boundary points.

\begin{lemma}\label{lem:5.14}
Assume that $\#(J)=2$, then $\partial^c{\mathbf C}=\emptyset$.
\end{lemma}

\begin{proof}
Assume by contradiction that there is $x \in \partial^c \mathbf C$. Then,
without loss of generality, $x \in \Gamma_1$ and, from Lemma \ref{lem:5.11}, $x$ cannot be an extreme point of $\Gamma_1$ so there exists $s \in (0,1)$ such that $x=\gamma_1(s)$. Let us distinguish two cases:

\medskip

\noindent {\it Case I:} If there is a sequence $(x_n)_{n \in \N}$ in $\Gamma_1 \setminus \partial^c \mathbf C$ such that $x_n\to x$, then, according to Lemma \ref{lem5.13}, the characteristic line $L_{x_n}$ intersects $\partial \mathbf C$ at another point $y_n$ in $\Gamma_2$. There is $\delta>0$ and a sequence of points $z_n\in \mathbf C$ with ${\rm dist}(z_n, \partial \mathbf C \cap \partial \O_p) \geq \delta$ such that $x_n = z_n + {\theta_n} \sigma^\perp(z_n)$ for some {$\theta_n \in \R$}. Then, up to
a subsequence, $z_n \to z \in \mathbf C$, {$\theta_n \to \theta$} so that
 $x=z+{\theta}\sigma^\perp(z) \in L_{z}$, a contradiction.

\medskip

\noindent {\it Case II:}  If there exists $R>0$ such that $B_R(x) \cap \Gamma_1 \subset \partial^c \mathbf C$, according to Proposition \ref{prop.bad-points} (ii), there exists a maximal open set $U_x$ containing $x$ such that $\Gamma_1 \cap U_x=\, ]a,b[$ is a segment contained in $\partial^c \mathbf C$. Consider {\it e.g.} the point $a$ and note that $a\neq \gamma_1(0)$ and $b \neq \gamma_1(1)$ otherwise, $a$ (resp. $b$) would be the apex of a boundary fan according to Lemma \ref{lem:5.11}. Indeed, one could consider a small open triangle $T$ with vertices $a=\gamma_1(0)$ (resp. $b=\gamma_1(1)$), a point in $]a,b[$ and a point on $L_{\gamma_1(0)}$ (resp $L_{\gamma_1(1)}$) so that ${\mathring T}\subset \mathring {\mathbf C}$. But then, any point $y\in \mathring T$ would be such that
$L_y$ contains $\gamma_1(0)$ (resp. $\gamma_1(1)$) which would therefore be a boundary fan contained in $\mathbf C$.
Thus, $a=\gamma_1(s)$ for some $s \in (0,1)$.

We claim that there exists a sequence $(s_n)_{n \in \mathbb N}$ in $(0, s)$ such that $x_n=\gamma_1(s_n)\in \Gamma_1\setminus\partial^c \mathbf C$ for each $n \in \mathbb N$ and $x_n \to  a$. Otherwise, $a$, which is in $(\partial \mathbf C \cap \partial \O_p)^\circ$, is not an accumulation point of
$(\partial \mathbf C \cap \partial \O_p)  \setminus \partial^c \mathbf C$ so that, according to  Proposition \ref{prop.bad-points}(ii), the point $a$ would be in a maximal open segment $]a',a''[$ contained in $\Gamma_1 \cap \partial^c
\mathbf C$ {for some $a'=\gamma_1(s') \in \Gamma_1$ and $a''=\gamma_1(s'') \in \Gamma_1$ with $0<s'<s<s''\le 1$}. Then $]a',a''[$ and $]a,b[$ must be aligned, $\partial \mathbf C\cap\partial\O_p$ being Lipschitz, and we reach a contradiction because $]a',b[$ must strictly contain
$]a,b[$ which cannot thus be maximal.

According to the previous claim and Lemma \ref{lem5.13}, the characteristic line $L_{x_n}$ intersects $\partial \mathbf C$ at another point $y_n \in \Gamma_2$. Since ${\rm dist\;}(\Gamma_2,\Gamma_1)>0$, there exist $\delta>0$ and a sequence of points $(z_n)_{n \in \N}$ in $\mathbf C$ with ${\rm dist}(z_n, \partial \mathbf C \cap \partial \O_p) \geq \delta$ and such that $x_n = z_n + \theta_n \sigma^\perp(z_n)$ for some $\theta_n \in
\R$. Then, up to a subsequence, $y_n \to y$, $z_n \to z \in \mathbf C$ and $\theta_n \to \theta$ so that, using the continuity of $\sigma$ in $\O_p$,
$$a=z+\theta \sigma^\perp(z) \in L_{z},$$
hence $ a \not\in \partial^c \mathbf C$.  Considering a small open triangle $T$ with apexes $a$, $(a+b)/2$ and a point living on $L_z$ in such a way that $ T \subset \mathring{\mathbf C}$, then any point $w \in   T$ has
a characteristic line $L_w$ which must pass through the point $a$, leading to a boundary fan contained in $\mathbf C$ and centered at $a$, a contradiction.
\end{proof}

{\begin{remark}\label{rem:con.comp.bis}
As was the case for Remark \ref{rem:con.comp} and with the same notation,
Lemma \ref{lem:5.14} actually establishes that if
$\#(K)= 2$, then $\partial^c {\mathbf D}=\emptyset$.
\hfill\P\end{remark}

We next focus on the case where $\partial \mathbf C \cap \partial \O_p$ is connected, and show that $\partial^c \mathbf C$ is connected set which separates $\partial \mathbf C \cap \partial \O_p$ into two connected components.

\begin{lemma}\label{lem:516}
If $\#(J)=1$, then the set $\partial^c \mathbf C$ is either a single point or a {closed} line segment and $\partial\mathbf C \cap \partial\O_p \setminus \partial^c \mathbf C$ has two connected components $\Gamma'_1$ and $\Gamma'_2$. Further, all characteristic lines that intersect $\partial \mathbf C \setminus \partial^c \mathbf C$ must intersect both connected
components $\Gamma'_1$ and $\Gamma'_2$.
\end{lemma}

\begin{proof}
Since $\#(I)=1$, there exists a unique characteristic line $L=L_1$ such that $\mathbf C= \O_p \cap H$, where $H$ is an closed hyperplane such that $\partial H=L$. Moreover,
$$\Gamma:=\partial\mathbf C \cap \partial \O_p=\gamma([0,1])$$
for some Lipschitz continuous mapping $\gamma:[0,1] \to \R^2$.

\medskip

{\sf Step 1. } Let us first prove that $\partial^c\mathbf C \neq \emptyset$. To this aim, assume by contradiction that $\partial^c\mathbf C = \emptyset$.

We first consider the characteristic line passing through $\gamma(0)$ and
$\gamma(1)$ {and set $(s_0,t_0)=(0,1)$}. Assume that $(s_n,t_n)$ are known so that $L_{\gamma(s_n)} \cap \overline \O_p=[\gamma(s_n),\gamma(t_n)]$. Let $V_n:[s_n,t_n] \to \R^+$ be the continuous function defined by
$$V_n(t)=\HH^1(\gamma([s_n,t])) =\int_{s_n}^t |\dot \gamma(s)|\, ds\quad \text{ for all }t \in [s_n,t_n]$$
which satisfies $V_n(s_n)=0$ and $V_n(t_n)=\HH^1(\gamma([s_n,t_n]))$.
According to the intermediate valued Theorem, there exists $s_{n+1/2} \in
(s_n,t_n)$ such that
$$V_n(s_{n+1/2})= \frac{V_n(t_n)}{2}=\frac{\HH^1(\gamma([s_n,t_n])}{2},$$
hus $\gamma([s_n,t_n])$ splits into two curves $\gamma([s_n,s_{n+1/2}])$ and $\gamma([s_{n+1/2},t_n])$ of equal length with
$$\HH^1(\gamma([s_n,s_{n+1/2}]))=\HH^1(\gamma([s_{n+1/2},t_n]))=\frac12 \HH^1(\gamma([s_n,t_n])).$$
Since $\O_p$ is convex, the resulting triangle $T_n$ with vertices $\gamma(s_n)$, $\gamma(t_n)$ and $\gamma(s_{n+1/2})$ satisfies $\mathring T_n \subset \O_p$. Moreover, $[\gamma(s_n),\gamma(s_{n+1/2})] \cup [\gamma(t_n),\gamma(s_{n+1/2})] \not\subset\partial\mathbf C$, otherwise it would also be lying in $\partial\mathbf C \cap \partial \O_p$ because that latter
set has only one connected component by assumption. Then the characteristic line passing through $\gamma(s_{n+1/2})$ would necessarily intersect  the open segment  $]\gamma(s_n),\gamma(t_n)[ \, \subset L_{\gamma(s_n)} \cap \O_p$, which is impossible by Proposition \ref{prop:vortex}. We then define  $t_{n+1/2}$ such that $\gamma(t_{n+1/2})$ is the other intersection point of $L_{\gamma(s_{n+1/2})}$ with $\partial \O_p$. Note that $t_{n+1/2} \in (s_n,t_n)$ otherwise $\gamma(s_n)$ (resp. $\gamma(t_n)$) would form a boundary fan contained in $\mathbf C$. Then if $s_{n+1/2} < t_{n+1/2}$, we define ${(s_{n+1},t_{n+1})}:=(s_{n+1/2},t_{n+1/2})$, while if $s_{n+1/2} > t_{n+1/2}$, we define ${(s_{n+1},t_{n+1})}:=(t_{n+1/2},s_{n+1/2})$ (see Figure \ref{fig:isoceles}). The sequence $(s_n)_{n \in \N}$ is increasing while $(t_n)_{n \in \N}$ is decreasing with $0 < s_n<s_{n+1} <t_{n+1} < t_n < 1$ for all $n \in \N$. Therefore, $s_n\to \bar s$ and $t_n\to \bar t$ for some $0 < \bar s \leq \bar t < 1$, and the line segment $[s_n,t_n]$ converges in the sense of Hausdorff to $[\bar s, \bar t]$. Since
\begin{multline*}
|\gamma(t_{n+1})- \gamma(s_{n+1})| \leq \HH^1(\gamma([s_{n+1},t_{n+1}]))
\leq  \HH^1([\gamma(s_{n+1/2},t_n])\\ =\HH^1(\gamma([s_n,s_{n+1/2}])) =
 \frac{\HH^1(\gamma([s_n,t_n])}{2} \leq \frac{\HH^1(\gamma([0,1])}{2^{n+1}} \to 0,
\end{multline*}
we deduce that $\gamma(\bar s)=\gamma(\bar t)=:x$. Since $x \not\in \partial^c \mathbf C$ (because $\partial^c \mathbf C$ is empty), it lies on a characteristic line $L_x$ which must intersect $[\gamma(s_n),\gamma(t_n)]$, a contradiction with Proposition \ref{prop:vortex}.

\begin{figure}[htbp]
\scalebox{.8}{\begin{tikzpicture}
\fill[color=gray!35]  plot[smooth cycle] coordinates {(-1,-1) (4,0) (3,2) (0,3) (-1,2)};

\fill[color=white]  plot coordinates {(2.42,-0.55) (4.5,-0.55) (4.5,1.37) (3.65,1.37)  };

\draw plot[smooth cycle] coordinates {(-1,-1) (4,0) (3,2) (0,3) (-1,2)};
\draw (1.8,-1.5) -- (4,2) ;

\draw (-2,1) -- (5,2.5) ; \draw(5,2.5) node[right]{\small $L_{\gamma(s_n)}$};
\draw  (-1.18,1.19) node{{\small $\bullet$}};
\draw (-0.8,1.3) node[below]{{\small $\gamma(t_n)$}};

\draw  (2.95,2.05) node{{\small $\bullet$}};
\draw (2.95,2.05) node[above]{{\small $\gamma(s_n)$}};

\draw[dashed] (-1.18,1.19) -- (0.6,2.9) ;
\draw[dashed] (2.95,2.05) -- (0.6,2.9) ;
\draw  (0.6,2.9) node{{\small $\bullet$}};
\draw (0.72,2.96) node[right]{{\small $\gamma(s_{n+1/2})=\gamma(s_{n+1})$}};

\draw (-2,1.3) -- (1.5,3.45) ; \draw (1.5,3.45) node[right]{\small $L_{\gamma(s_{n+1})}$};
\draw  (-1.03,1.87) node{{\small $\bullet$}};
\draw  (-1.03,1.95) node[left]{{\small $\gamma(t_{n+1})$}};

\draw (1,-1.1) node[right]{\small $\Gamma$};
\draw (2.8,0.2) node[right]{\small $L$};
\draw (1,0.8) node[right]{\small $\mathbf C$};
\draw  (2.42,-0.55) node{{\small $\bullet$}};
\draw (2.7,-0.55)  node[below]{{\small $\gamma(1)$}};
\draw  (3.65,1.37) node{{\small $\bullet$}};
\draw  (3.65,1.37)  node[right]{{\small $\gamma(0)$}};
\end{tikzpicture}}
\caption{\small  Construction of the sequences $(s_n)_{n \in \N}$ and $(t_n)_{n \in \N}$}
\label{fig:isoceles}
\end{figure}
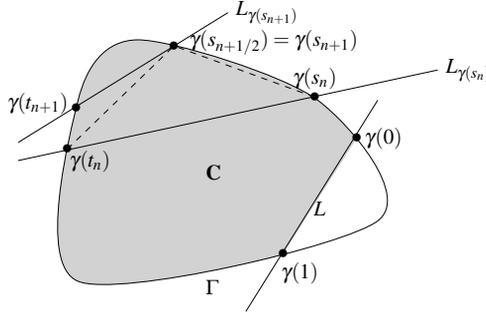

\medskip

{\sf Step 2.} Assume now that $ \partial^c\mathbf C$ is not connected and
consider two distinct connected components $S$ and $S'$ with
$$t_0:=\inf\{t \in [0,1] :\; \gamma(t)\in S'\}>\sup\{s \in [0,1] :\ \gamma(s)\in S\}=:s_0.$$
Then there exists $\bar s\in [s_0,t_0]$ such that $x:=\gamma(\bar s)\notin \partial^c\mathbf C$. Consider the characteristic line $L_x$ passing through $x$. It must intersect $\Gamma$ at some other point $x'$ which cannot coincide with $\gamma(0)$ or $\gamma(1)$, otherwise it would form with $L$ a boundary fan {contained in $\mathbf C$} with apex at one of these points. Let $H$ be the closed half-plane with $L_x$ as boundary containing both $\gamma(0)$ and $\gamma(1)$. Then the region ${\mathbf D}:=\mathbf C \cap H$ would be a convex relatively closed subset of $\O_p$ such that, on the one hand, its boundary contains {two characteristic lines $L$ and $L_x$}, and, on the other hand, {$\partial\mathbf D \cap\partial\O_p$ has two connected components}. {Since one of $S$ or $S'$ is contained in $\partial^c \mathbf D$,   $\partial^c \mathbf D \neq \emptyset$}, a  contradiction according to Remark \ref{rem:con.comp.bis}.
So $ \partial^c\mathbf C$ is connected and $\partial\mathbf C \cap \partial\O_p \setminus \partial^c \mathbf C$ has  two connected components $\Gamma'_1$ and $\Gamma'_2$.

\medskip

{\sf Step 3.} If $\partial^c \mathbf C$ is not reduced to a point, it must contain $\gamma((s',t'))$ for some $0<s'<t'<1$ so that its relative interior in $\partial\mathbf C$, denoted by $(\partial^c \mathbf C)^\circ$, is not empty. Then, any point $x \in (\partial^c \mathbf C)^\circ$ is not an accumulation point of $(\partial\mathbf C \cap \partial\O_p) \setminus \partial^c \mathbf C$ so that Proposition \ref{prop.bad-points}-(ii) ensures that $x \in \, ]a_x,b_x[\,  \subset \partial^c \mathbf C$ for
some maximal open segment $]a_x,b_x[$ containing $x$. If $]a_x,b_x[ \, \neq (\partial^c \mathbf C)^\circ$, then $a_x$ is not an accumulation of
$(\partial\mathbf C \cap \partial\O_p) \setminus \partial^c \mathbf C$, and a new application of  Proposition \ref{prop.bad-points}-(ii) shows that $a_x$ is contained in another open segment $S\subset \partial^c \mathbf
C$. Then, $]a_x,b_x[$ and $S$, being on the Lipschitz boundary $\partial \mathbf  C\cap\partial\O$, must be aligned so that $]a_x,b_x[\, \cup S \subset \partial^c\mathbf C$ is an open {line} segment containing $x$ which is strictly larger that $]a_x,b_x[$, a contradiction to the maximality of $]a_x,b_x[$. Therefore we must have that $]a_x,b_x[\, =(\partial^c \mathbf C)^\circ$, which shows that $\partial^c \mathbf C$ is a segment. If $a_x \not\in \partial^c \mathbf C$, then there exists $z \in \mathbf C$ such that $a_x \in L_z$ and $a_x$ would be the apex of a boundary fan contained in $\mathbf C$ by an argument identical to that used at the end of Case II in the proof of Lemma \ref{lem:5.14}. Therefore $a_x \in \partial^c \mathbf C$ and the same argument shows that $b_x \in \partial^c \mathbf C$. This proves that $\partial^c \mathbf C$ is a closed line segment.

\medskip

{\sf Step 4.} Assume that there is $z \in \mathbf C$ such that the characteristic line $L_z$ intersects say $\Gamma'_1$ at two distinct points $\gamma(s)$ and $\gamma(t) \in \Gamma'_1$. But then the same construction as
in Step 1 would establish the existence of some $\check s\in [s,t]$ such that $\gamma(\check s)\in \partial^c\mathbf C \cap \Gamma'_1$, which is impossible since $\Gamma'_1 \subset \partial\mathbf C \setminus \partial^c
\mathbf C$.  Alternatively, the closed convex connected subdomain $\mathbf D: =H\cap \mathbf C$ of $\mathbf C$, where $H$ is the closed half plane with boundary $L_z$ containing $\partial^c\mathbf C$, would  be such that, in the notation of Remark \ref{rem:con.comp}, $\#(K)=2$ while $\partial^c\mathbf D= \partial^c\mathbf C\ne\emptyset$ and Remark \ref{rem:con.comp.bis} would lead to a contradiction. Thus all characteristic lines that intersect $\Gamma'_1$ must
intersect $\Gamma'_2$ as well.
\end{proof}

\begin{remark}\label{rem.flat.bdary}
Remark that  in Lemma \ref{lem:516}, for any $x\in \Gamma'_1$, $y\in \Gamma'_2$, the line segment $[x,y]\not\subset \partial \mathbf C\cap\partial\O_p$. In other words, the boundary $\partial \mathbf C\cap\partial\O_p$ cannot be flat around $\partial^c \mathbf C$. Assuming otherwise, consider the Lipschitz {parameterization} $\gamma$ of $\partial \mathbf C$ and assume, for example, that $x=\gamma(s)$, $y=\gamma(t)$ with $s<t$. Take $L_x$ (resp. $L_y$) be the associated characteristic lines. Then, $L_x$
must  intersect $\Gamma'_2$ at $\gamma(t')$ with $t'>t$ while $L_y$ must
intersect $\Gamma'_1$ at $\gamma(s')$ with $s'<s$. But then, $L_x$ and $L_y$ must intersect in $\mathbf C$, which is impossible by Proposition \ref{prop:vortex}.
\hfill\P\end{remark}

\begin{remark}\label{rem:Gamma12}
Theorem \ref{thm:geom-struct} can be rephrased as follows:  the non-characteristic exterior boundary $\partial\mathbf C \cap \partial \O_p \setminus \partial^c \mathbf C$ has always exactly two connected components, denoted hereafter by $\Gamma_1$ and $\Gamma_2$ and all characteristic lines that intersect $\partial \mathbf C$ must intersect both $\Gamma_1$ and $\Gamma_2$.
\hfill\P
\end{remark}

\medskip

{The particular geometrical structure of  the connected components of $\mathscr C$ allows one to improve the continuity of the Cauchy stress $\sigma$ up to the (non characteristic) boundary.}

\begin{theorem}\label{thm:cont-sigma}
Let $\mathbf C$ be a connected component of $\mathscr C$ with nonempty interior. Then $\sigma$ is continuous in $\overline{\mathbf C}\setminus \partial^c \mathbf C$.
\end{theorem}

\begin{proof}
We already know that $\sigma$ is locally Lipschitz continuous on $\mathbf
C$. It thus remains to prove that $\sigma$ is continuous on $\partial\mathbf C\cap\partial\O_p\setminus \partial^c \mathbf C$. According to Theorem \ref{thm:geom-struct} and Remark \ref{rem:Gamma12}, we know that $\partial \mathbf C \cap \partial \O_p \setminus \partial^c \mathbf C$ has two connected components $\Gamma_1$ and $\Gamma_2$.

Let $x \in \partial\mathbf C\cap\partial\O_p\setminus \partial^c \mathbf C$ and $(x_n)_{n \in \N}$ be a sequence in $\overline{\mathbf C}\setminus
\partial^c \mathbf C$ such that $x_n \to x$. {Without loss of generality,
we can assume that $x \in \Gamma_1$. Since by Theorem \ref{thm:geom-struct} $L_{x_n}$ must intersect both $\Gamma_1$ and $\Gamma_2$, there exists $a_n \in L_{x_n} \cap \Gamma_1 $ and $b_n \in L_{x_n} \cap \Gamma_2$ which, up to a subsequence, satisfy $a_n \to a \in \overline\Gamma_1$ and $b_n \to b \in \overline\Gamma_2$. As a consequence, the closed segment $S_n:=L_{x_n}\cap \overline \O_p$ converges in the sense of Hausdorff to the segment $S=[a,b]$ with $x=a$}.

\medskip

{\it Case I:} Assume first that there exists $\delta>0$ such that for all
$n \in \N$,
$$\max_{z \in S_n} {\rm dist}(z,\partial\mathbf C\cap\partial\O_p) \geq \delta,$$
then there exists $z_n \in S_n$ such that ${\rm dist}(z_n,\partial\mathbf
C\cap\partial\O_p) \geq \delta$ and, up to a subsequence, $z_n \to z$ for
some $z \in S$ with ${\rm dist}(z,\partial\mathbf C\cap\partial\O_p) \geq
\delta$. Since $x_n \in L_{z_n}$, there exists $\theta_n \in \R$ such that
$$x_n=z_n+\theta_n\sigma^\perp(z_n).$$
Note that, up to a further subsequence, $\theta_n \to \theta \in \R$ and thus, by continuity of $\sigma$ in $\O_p$, we have $x=z+\theta\sigma^\perp(z)$ which ensures that $x \in L_z$. Thus, using that $\sigma$ is constant along characteristics and, once again, the continuity of $\sigma$ in
$\O_p$, we get that
$$\sigma(x_n)=\sigma(z_n) \to \sigma(z)=\sigma(x).$$

\medskip

{\it Case II:} Assume next that, up to a subsequence
$${\max_{z \in S_n} {\rm dist}(z,\partial\mathbf C\cap\partial\O_p) \to 0.}$$
By Hausdorff convergence, for all $z \in S$, there exists a sequence $(z_n)_{n \in \N}$ with $z_n \in S_n$ and $z_n \to z$. Thus, ${\rm dist}(z_n,\partial\mathbf C\cap\partial\O_p) \to 0$ which ensures that $S \subset \partial \mathbf C{\cap\partial\O_p}$. Then {$S=[a,b]=[x,b]$} is a line segment contained in $\partial\mathbf C\cap\partial\O_p$ {which contains  $\partial^c\mathbf C$. Further,  $ \partial^c\mathbf C=[c,d]$ for some points $c,d\in \partial\mathbf C\cap\partial\O_p$ with $c\in \overline\Gamma_1$, $d\in \overline\Gamma_2$ and maybe $c=d$. Since $x\notin \partial^c\mathbf C$, $x=a\ne c$, $[a,c[\cap\partial^c\mathbf C=\emptyset$, and there exists a point $x'\in [a,c[$ such that $x'\notin \partial^c\mathbf C$. But then, the characteristic line segment $L_{x'}\cap\overline\O_p$ must intersect $S_n$ since $S_n$ Hausdorff-converges to $[a,b]$, which is impossible according to Proposition \ref{prop:vortex}.  }
Thus, this second case never occurs.
\end{proof}

The following result is analogous to Proposition \ref{prop:unique-u-fan}.
It gives a uniqueness result for the displacement in some portion of $\mathbf C$ whose boundary intersects the boundary of the domain $\O$.

\begin{proposition}\label{prop:unique-C}
Let $\mathbf C$ be a connected component of $\mathscr C$ with nonempty interior. Extend $u$ by $w$ outside $\Omega$. Then, $u^+=u^-$ $\HH^1$-a.e. on $\partial \mathbf C \cap   \partial \O_p \setminus \partial^c \mathbf C$ and, in particular, $u=w$ $\HH^1$-a.e. on  $\partial \mathbf C \cap   \partial \O_p \cap \partial\Omega \setminus \partial^c \mathbf C$.
\end{proposition}

\begin{proof}
In the current setting,  $\partial \mathbf C \cap   \partial \O_p \setminus \partial^c \mathbf C$ has two connected components $\Gamma_1$ and $\Gamma_2$ which are {open in the relative topology of $\partial\O_p$ and
Lipschitz}. According to Remark \ref{rmk:jump-flow-rule}, $\sigma\cdot\nu=\pm1$ $\HH^1$-a.e. on $\partial \mathbf C \cap \partial\O_p \cap J_u\setminus \partial^c \mathbf C$. Theorem \ref{thm:cont-sigma} ensures that
$\sigma$ is continuous in $\overline{\mathbf C}\setminus \partial^c \mathbf C$. We thus deduce that $\sigma\cdot \nu$ coincides $\HH^1$-a.e. with the usual scalar product of $\sigma$ and $\nu$ on $\partial \mathbf C\cap
\partial\O_p\setminus \partial^c \mathbf C$. Therefore, $\nu=\pm\sigma$
$\HH^1$-a.e. on $\partial \mathbf C \cap  \partial\O_p\cap J_u\setminus \partial^c \mathbf C$.

Assume that $\HH^1(\partial \mathbf C \cap  \partial\O_p \cap J_u\setminus \partial^c \mathbf C)>0$, we can find a point $x_0 \in \partial \mathbf
C \cap  \partial\O_p \cap J_u\setminus \partial^c \mathbf C$ such that, up to a change of sign, $\nu(x_0)=\sigma(x_0)$. It thus follows that $L_{x_0}$ is tangent to $\mathbf C$ at $x_0$ which is in contradiction with the fact that $x_0 \not\in \partial^c \mathbf C$. Therefore, $\HH^1(\partial \mathbf C \cap  \partial\O_p \cap J_u\setminus \partial^c \mathbf C)=0$ and thus,  $u^+=u^-$ $\HH^1$-a.e. on $\partial \mathbf C \cap  \partial\O_p \setminus \partial^c \mathbf C$.

Using the Dirichlet boundary condition, we get that $u=w$ $\HH^1$-a.e. on $\partial \mathbf C \cap  \partial\O_p\cap \partial\O\setminus \partial^c \mathbf C$.
\end{proof}

\begin{remark}\label{rem:strict-conv}
If $\O$ was a convex domain and if $\sigma$ was such that $|\sigma|\equiv
1$ in $\O$, then, $\O=\O_p$ would be the disjoint union of characteristic lines $L_{x_\lambda}$ (see Proposition \ref{prop:emptyint}),  of countably many boundary fans $\mathbf F_{\bar z_i}$, and of countably many connected components $\mathbf C_j$ (with non empty interior) such that every
point in $\partial\mathbf C_j \cap \partial\O \setminus \partial^c \mathbf C_j$ is traversed by a characteristic line (see Theorem \ref{thm:geom-struct}). Since all characteristic lines cut $\partial\O$ into two distinct points by convexity of $\O$, then $u$ is entirely determined by $w$ in each fan $\mathbf F_{\bar z_j}$ according Proposition \ref{prop:unique-u-fan}. Similarly, by virtue of Proposition \ref{prop:unique-C}, $u(x)=w(x)$ for $\HH^1$-a.e. $x \in \partial\mathbf C_j \cap \partial\O \setminus
\partial^c \mathbf C_j$ and, finally, by Theorem \ref{thm:cst-disp}, $u$ is constant along $\mathcal H^1$-a.e. associated characteristic line $L_x$.

Since it is generically not so that the values of $w$ at both points of $L_x \cap \partial \O$ should match, this situation is not possible. Therefore, in the scalar case with a generic Dirichlet boundary condition over the entire boundary, a convex body cannot be entirely plastified unless it is made only of boundary fans.

However, as seen in the example in the introduction, the situation is different in the case of mixed boundary conditions.
\hfill\P\end{remark}

\subsection{Miscellanea}\label{examples}
\subsubsection{Exterior fans}

In contrast with the case of  boundary fans, it might happen that two characteristic lines which intersect outside $\overline \O_p$ will not generate an exterior fan. In particular, there is no analogue to Lemma \ref{lem:cone-bdary}. We thus need to force this property by introducing the (possibly empty) set
\begin{eqnarray*}
\mathcal F^{\rm ext} & := & \Big\{\bar z \in \R^2\setminus \overline \O_p: \; \exists \, x, y \in \O_p \text{ with }y \neq x \text{ such that } L_x \cap L_y= \{\bar z\}\nonumber\\
&& \hspace{2cm}\text{ and } \bar z \in L_z \text{ for all }z \in {(\bar z+C(x-\bar z,y-\bar z))\cap \O_p}\Big\}.
\end{eqnarray*}

\begin{figure}[htbp]
\begin{tabular}{cc}
\begin{minipage}[1cm]{.50\linewidth}
\scalebox{.8}{\begin{tikzpicture}
\fill[color=gray!35]  plot[smooth cycle] coordinates {(-1,-1) (3,0) (3,2) (0,3) (-1,2)};
\fill[color=white]  plot coordinates {(0,-2) (-0.6,3.5) (-2,3.5) (-2,-2)};
\fill[color=white]  plot coordinates {(0,-2) (2,3.5) (5,3.5) (5,-2)};

\draw plot[smooth cycle] coordinates {(-1,-1) (3,0) (3,2) (0,3) (-1,2)};
\draw[style=thick] (0,-2) -- (-0.6,3.5)  ;
\draw (0,-2) -- (0,3.5) ;
\draw (0,-2) -- (0.5,3.5) ;
\draw (0,-2) -- (1,3.5) ;
\draw (0,-2) -- (1.5,3.5) ;
\draw[style=thick] (0,-2) -- (2,3.5) ;

\draw (0.1,1.8) node[right]{\small $\mathbf C$};
\draw (1.5,0.8) node[right]{\small $\Omega_p$};
\draw (0,-2) node[below]{\small $\bar z$};
\end{tikzpicture}}
\caption{\small  An exterior fan with apex $\bar z$ in the case $\#(I)=2$}
\label{fig:extfan1}
\end{minipage}

&

\begin{minipage}[1cm]{.50\linewidth}
\scalebox{.8}{\begin{tikzpicture}
\fill[color=gray!35]  plot[smooth cycle] coordinates {(-1,-1) (3,0) (3,2) (0,3) (-1,2)};
\fill[color=white]  plot coordinates {  (-0.85,-2)  (0,3.5)  (5,3.5) (5,-2)};
\draw plot[smooth cycle] coordinates {(-1,-1) (3,0) (3,2) (0,3) (-1,2)};

\draw[style=thick] (-0.85,-2) -- (-2.45,3.5)  ;

\draw  (-1.22,-0.7) node{{\small $\bullet$}};
\draw (-1.25,-0.7)  node[left]{{\small $\partial^c \mathbf C \ni x$}};

\draw (-0.85,-2) -- (-2,3.5) ;
\draw (-0.85,-2) -- (-0.5,3.5) ;
\draw (-0.85,-2) -- (-1,3.5) ;
\draw (-0.85,-2) -- (-1.5,3.5) ;
\draw[style=thick] (-0.85,-2) -- (0,3.5) ;

\draw (-1,1) node[right]{\small $\mathbf C$};
\draw (1.5,0.8) node[right]{\small $\Omega_p$};
\draw (-0.85,-2) node[below]{\small $\bar z$};
\end{tikzpicture}}
\caption{\small  An exterior fan with apex $\bar z$ in the case $\#(I)=1$}
\label{fig:extfan2}
\end{minipage}

\end{tabular}
\end{figure}
For all $\bar z \in \mathcal F^{\rm ext}$, there exists a maximal open set ${(\bar z+C(x\!-\!\bar z,y\!-\!\bar z)) \!\cap \O_p}$, for some $x$ and
$y \in \overline \O_p$, which is denoted by $\mathbf F_{\bar z}$, and which is called an {\it exterior fan}. Arguing exactly as in the proof of Theorem \ref{thm:fan-sigma}, we obtain the following rigidity result inside
exterior fans.
\begin{proposition}\label{prop:ext-fan}
If $\mathbf C$ is a connected component of $\mathscr C$ with nonempty interior such that $\mathbf C=\overline{\mathbf F}_{\bar z}\cap \O_p$, for
some  $\bar z \in \mathcal F^{\rm ext}$, then there exists $\alpha \in \{-1,1\}$ such that
$$\sigma(x)=\alpha \frac{(x-\bar z)^\perp}{|x-\bar z|} \quad \text{ for
all }x \in \overline{\mathbf F}_{\bar z} {\cap \O_p}$$
and
$$u(x)=\alpha h\left( \frac{(x-\bar z)_2}{(x-\bar z)_1}\right) \text{ for all }x \in \mathbf F_{\bar z},$$
for some nondecreasing function $h:\R \to \R$.
\end{proposition}
\begin{remark}\label{rem:ext-fan}
The analogue of Proposition \ref{prop:unique-u-fan} also holds true for exterior fans and there can be no jumps on $\partial \mathbf F_{\bar z}\cap C_{\bar z}$, {where $C_{\bar z}$ is the maximal open cone such that $\mathbf F_{\bar z}= C_{\bar z}\cap \O_p$}. Also, $u=w$ $\HH^1$-a.e. on
$\partial \mathbf F_{\bar z}\cap C_{\bar z}\cap \partial\Omega$.
\hfill\P\end{remark}

\subsubsection{Constant zones}

Constant zones have a special structure.

\begin{proposition}\label{prop:cst-u}
If $\mathbf C$ is a connected component of $\mathscr C$ with nonempty interior on which $\sigma=\bar \sigma$ is constant, then $u(x)=\bar \sigma \cdot x+v(\bar \sigma \cdot x)$ for all $x \in \mathring{\mathbf C}$ for some non decreasing function $v$,  and, if $\partial \mathbf C \cap \partial \O_p$ has two connected components, then both characteristic lines
$L_1$ and $L_2$ lying on $\partial\mathbf C \cap \O_p$ are parallel to $\R \bar \sigma^\perp$.
\end{proposition}

\begin{proof}
Assume for simplicity that $\bar\sigma=e_1=(1,0)$. Addressing the last part of the Proposition, if one of both characteristic lines, say $L_1$, lying on $\partial\mathbf C\cap \O_p$ is not parallel to $e_2=(0,1)$,
then there exists $x \in \mathring{\mathbf C}$ such that $L_x$ intersects
$L_1$ inside $\O_p$ which is absurd in view of Proposition \ref{prop:vortex}.

Further, with the help of \eqref{eq:DU=},
$$D_2|p|= \Div (\bar \sigma^\perp(\mathcal L^2+|p|))=0 \quad \text{ in }\mathcal D'(\mathring{\mathbf C}),$$
which means that $|p|$ is independent of $x_2$. As a consequence, $|p|=
\pi_{x_1}\otimes \LL^1_{x_2} $ for some nonnegative measure $\pi=\pi_{x_1} \in \mathcal M(\R)$, {\it i.e.},
$$\int_{\mathring{\mathbf C}}\varphi\, d|p|=\int_{\mathring{\mathbf C}}
\varphi(x_1,x_2)\, d\pi(x_1)\, dx_2 \quad \text{ for all }\varphi\in \C_c(\mathring{\mathbf C}).$$
This implies, in view of \eqref{eq:DU=},  that $Du=e_1 (1+ \pi_{x_1}\otimes \LL^1_{x_2})$ on $\mathring{\mathbf C}$. Hence $u(x)=x_1+v(x_1)$
for a.e. $x \in \mathring{\mathbf C}$ for some $v\in BV(\R)$ with $Dv=\pi\geq 0$.
\end{proof}

\begin{figure}[htbp]
\begin{tabular}{cc}
\hskip-.5cm\begin{minipage}[5cm]{.50\linewidth}

\scalebox{.8}{\begin{tikzpicture}
\fill[color=gray!35]  plot[smooth cycle] coordinates {(-1,-1) (3,0) (3,2) (0,3) (-1,2)};
\fill[color=white]  plot coordinates {  (0,-2) (-1,3.5) (-3,3.5) (-3,-2)};
\fill[color=white]  plot coordinates {  (1.81,-2) (0.81,3.5) (5,3.5) (5,-2) };

\draw plot[smooth cycle] coordinates {(-1,-1) (3,0) (3,2) (0,3) (-1,2)};
\draw[style=thick] (0,-2) -- (-1,3.5)  ;
\draw (0.36,-2) -- (-0.64,3.5) ;
\draw (0.73,-2) -- (-0.27,3.5) ;
\draw (1.09,-2) -- (0.09,3.5) ;
\draw (1.45,-2) -- (0.45,3.5) ;
\draw[style=thick]  (1.81,-2) -- (0.81,3.5) ;

\draw (0.1,1) node[right]{\small $\mathbf C$};
\draw (1.5,0.8) node[right]{\small $\Omega_p$};
\end{tikzpicture}}
\caption{\small  A constant zone in the case $\#(I)=2$}
\label{fig:const1}
\end{minipage}

&

\hskip.5cm\begin{minipage}[5cm]{.50\linewidth}
\scalebox{.8}{\begin{tikzpicture}
\fill[color=gray!35]  plot[smooth cycle] coordinates {(-1,-1) (3,0) (3,2) (0,3) (-1,2)};
\fill[color=white]  plot coordinates {  (0.76,-2) (-0.24,3.5) (5,3.5) (5,-2)};

\draw plot[smooth cycle] coordinates {(-1,-1) (3,0) (3,2) (0,3) (-1,2)};

\draw[style=thick] (-1,-2) -- (-2,3.5)  ;
\draw (-0.66,-2) -- (-1.66,3.5) ;
\draw (-0.3,-2) -- (-1.3,3.5) ;
\draw (0.03,-2) -- (-0.95,3.5) ;
\draw (0.40,-2) -- (-0.60,3.5) ;
\draw[style=thick]  (0.76,-2) -- (-0.24,3.5) ;

\draw (-0.8,0.7) node[right]{\small $\mathbf C$};
\draw (1.5,0.8) node[right]{\small $\Omega_p$};
\draw  (-1.26,-0.55) node{{\small $\bullet$}};
\draw (-1.3,-0.55) node[left]{{\small $\partial^c \mathbf C \ni x$}};

\end{tikzpicture}}
\caption{\small  A constant zone in the case $\#(I)=1$}
\label{fig:const2}
\end{minipage}

\end{tabular}
\end{figure}

\subsubsection{An example of a non-fan, non-constant structure}\label{4.5.3}

Assume that a connected component $\mathbf C$ of $\O_p$ is such that
$\mathbf C \subset (0,R)^2,$
for some $R>0$, and that the characteristic lines are given by the one-parameter family
$$L_t=\left\{(x,y) \in \R^2 : \; y=\frac{x}{t}-t\right\}, \quad t>0.$$
Note that this family of lines does not intersect at a single point so that it does not define an exterior fan (see Figure \ref{fig:x/t-t}).
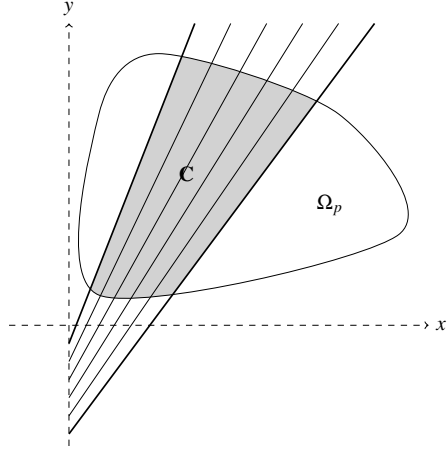
\begin{figure}[htbp]
\scalebox{.8}{\begin{tikzpicture}
\fill[color=gray!35]  plot[smooth cycle] coordinates  {(0.5,0.5) (5.5,1.5) (4.5,3.5) (1.5,4.5) (0.5,3.5)};
\fill[color=white]  plot coordinates {  (0,-0.3)  (2.1,5) (0,5) (0,-0.3)};
\fill[color=white]  plot coordinates { (0,-1.8)  (5.1,5) (7,5) (7,0)};

\draw plot[smooth cycle] coordinates {(0.5,0.5) (5.5,1.5) (4.5,3.5) (1.5,4.5) (0.5,3.5)};

\draw[dashed,->] (0,-2) -- (0,5)  ;
\draw[dashed, ->]  (-1,0) -- (6,0);

\draw[style=thick] (0,-0.3) -- (2.1,5) ;
\draw (0,-0.6) -- (2.7,5) ;
\draw (0,-0.9) -- (3.3,5) ;
\draw (0,-1.2) -- (3.9,5) ;
\draw (0,-1.5) -- (4.5,5) ;
\draw[style=thick] (0,-1.8) -- (5.1,5) ;

\draw(0,5) node[above]{\small$y$};
\draw(6,0) node[right]{\small$x$};
\draw (1.7,2.5) node[right]{\small $\mathbf C$};
\draw (4,2) node[right]{\small $\Omega_p$};
\end{tikzpicture}}
\caption{\small  An example of a non-fan, non-constant structure}
\label{fig:x/t-t}
\end{figure}
However, we are going to construct an example of smooth solutions $(u,\sigma,p)$ of
$$\begin{cases}
{\rm div\, }\sigma=0, \quad |\sigma|=1\quad \text{ in }\mathbf C,\\
\nabla u=\sigma+p, \quad p=\sigma |p| \quad \text{ in }\mathbf C,
\end{cases}$$
such that $u$ and $\sigma$ are constant along $L_t \cap \mathbf C$ for all $t>0$. Doing so, we will have  satisfied all equations \eqref{eq:plast}
on $\mathbf C$, ignoring the boundary condition on $\partial \O$ and recalling Remark \ref{rmk:strong-fr}.

\medskip

Let us start by constructing a (unique) stress $\sigma:\mathbf C \to \R^2$. Indeed, since $\sigma(x,y)$ must be an unimodular vector orthogonal to
$L_t$ at $(x,y=\frac{x}{t}-t)$, then
$$\sigma(x,y)=\frac{1}{\sqrt{1+t^2}}
\begin{pmatrix} -1 \\ t \end{pmatrix} \quad \text{ with }\quad t=\frac{-y+\sqrt{y^2+4x}}{2},$$
hence
$$\sigma(x,y):= \frac{2}{\sqrt{4+\left(\sqrt{y^2+4x}-y\right)^2}}\begin{pmatrix}
-1\\\frac{\sqrt{y^2+4x}-y}{2}
\end{pmatrix}.
$$
Then $\sigma \in \C^\infty(\mathbf C;\R^2)$ and, by construction, $|\sigma|=1$ in $\mathbf C$. Further a lengthy computation would show that ${\rm div\,}\sigma=0$ in $\mathbf C$.

\medskip

We now construct the displacement $u:\mathbf C\to\R$ as
$u(x,y):=f(\sqrt{y^2+4x}-y) \; \text{ for all }(x,y) \in \mathbf C,$
for some smooth decreasing function $f:\R^+ \to \R$ such that
\begin{equation}\label{eq:f'}
f'(t)\le -\frac{R+t}{\sqrt{4+t^2}}\quad \text{ for all }t>0.
\end{equation}
By construction, $u$ is constant along all lines $L_t$ for all
$t>0$.
We can immediately compute $\nabla u$ on $\mathbf C$. We  get
\begin{eqnarray*}
\nabla u(x,y)  & = & -\frac{2}{{\sqrt{y^2+4x}}} \begin{pmatrix}
-1\\ \frac{\sqrt{y^2+4x}-y}2
\end{pmatrix} f'(\sqrt{y^2+4x}-y)\\
& = & -\left(\frac{4+\left(\sqrt{y^2+4x}-y\right)^2}{y^2+4x}\right)^{1/2} \sigma(x,y)\, f'(\sqrt{y^2+4x}-y)  .
\end{eqnarray*}

\medskip

We then define the plastic strain $p: \mathbf C \to \R^2$ by
\begin{eqnarray*}
p(x,y) & := & \nabla u(x,y)-\sigma(x,y)\\
& = & -\left(\frac{\sqrt{4+\left(\sqrt{y^2+4x}-y\right)^2}}{\sqrt{y^2+4x}} f'(\sqrt{y^2+4x}-y)+1\right) \sigma(x,y).
\end{eqnarray*}
Note that, in view of \eqref{eq:f'},
$$-f'(\sqrt{y^2+4x}-y) \frac{\sqrt{4+\left(\sqrt{y^2+4x}-y\right)^2}}{\sqrt{y^2+4x}}-1\ge 0 $$
for all $(x,y) \in \mathbf C,$ so that
$$|p(x,y)|=-\left(\frac{\sqrt{4+\left(\sqrt{y^2+4x}-y\right)^2}}{\sqrt{y^2+4x}} f'(\sqrt{y^2+4x}-y)+1\right)$$
for all $(x,y) \in \mathbf C,$
and thus $p=\sigma|p|$. We have thus constructed smooth functions $u$ and $p$ satisfying $\nabla u=\sigma+p$ and $p=\sigma |p|$ in $\mathbf C$. Remark that there are many choices of functions $f$ satisfying \eqref{eq:f'}, even considering only  non-increasing $BV$-functions.



\appendix

\section{}

Let $I \subset \R$ be a bounded open interval. If $x_0 \in I$ and $\rho>0$, we denote by {$I_{x_0,\rho}=(x_0-\rho,x_0+\rho)$}. We recall that a function $f \in L^1(I)$ has {\it vanishing mean oscillation} in $I$, and we write $f \in {\rm VMO}(I)$, if
$$\frac{1}{\LL^1(I \cap I_{x_0,\rho})}\int_{I \cap I_{x_0,\rho}}|f(x)-f_{x_0,\rho}|\, dx \to 0, \quad \text{ as }\rho \to 0$$
uniformly with respect to $x_0 \in I$, where $f_{x_0,\rho}:=\frac{1}{\LL^1(I \cap I_{x_0,\rho})} \int_{I_{x_0,\rho}}f(y)\, dy$ is the average of
$f$ on $I \cap I_{x_0,\rho}$.

We first show a (well known) result (see \cite[Example 2]{BN}).

\begin{lemma}\label{lem:VMO}
If $f \in H^{1/2}(I)$, then $f \in {\rm VMO}(I)$.
\end{lemma}

\begin{proof}
Since $f \in H^{1/2}(I)$, then by definition
$$\iint_{I \times I}\frac{|f(x)-f(y)|^2}{|x-y|^2}\, dx\, dy <+\infty.$$
Therefore, by absolute continuity of the integral, for each $\e>0$ there exists $\delta>0$ such that for every Lebesgue measurable set $A \subset I \times I$ with $\LL^2(A) \leq \delta$, then
$$\iint_{A}\frac{|f(x)-f(y)|^2}{|x-y|^2}\, dx\, dy \leq  \e^2.$$

Let $x_0 \in I$ and $\rho \leq \sqrt{\delta}/2$ so that $\LL^2(I_{x_0,\rho}\times I_{x_0,\rho}) =4\rho^2 \leq \delta$. According to the triangle
and Cauchy-Schwarz inequalities, we have
\begin{multline*}
\frac{1}{\LL^1(I \cap I_{x_0,\rho})} \int_{I \cap I_{x_0,\rho}}|f(x)-f_{x_0,\rho}|\, dx   \\
\leq \frac{1}{\LL^1(I \cap I_{x_0,\rho})^2}\iint_{(I \cap I_{x_0,\rho})^2}|f(x)-f(y)|\, dx\, dy
  \\ \leq \left(\frac{1}{\LL^1(I \cap I_{x_0,\rho})^2}\iint_{(I \cap I_{x_0,\rho})^2}|f(x)-f(y)|^2\, dx\, dy\right)^{1/2}.
\end{multline*}
{If $x$ and $y \in I \cap I_{x_0,\rho}$, then,  $|x-y|\leq \LL^1(I \cap I_{x_0,\rho})$, hence }
\begin{multline*}
\frac{1}{\LL^1(I \cap I_{x_0,\rho})}\int_{I_{x_0,\rho}}|f-f_{x_0,\rho}|\,
dx\\
  \leq \left(\iint_{(I \cap I_{x_0,\rho})^2}\frac{|f(x)-f(y)|^2}{|x-y|^2}\, dx\, dy \right)^{1/2} \leq \e,
  \end{multline*}
which proves that $f \in {\rm VMO}(I)$.
\end{proof}

We now show that a characteristic function which belongs to $H^{1/2}(I)$ must be constant.

\begin{lemma}\label{lem:constant}
Let $\chi={\bf 1}_A$ for some Lebesgue measurable set $A \subset I$. If
$\chi \in H^{1/2}(I)$ then either $\LL^1(A)=0$ or $\LL^1(I \setminus A)=0$.
\end{lemma}

\begin{proof}
For $\LL^1$-a.e. $y\in I \cap I_{x,\rho}$ we have that $\dist(\chi_{x,\rho},\{0,1\})\le |\chi_{x,\rho}-\chi(y)|$.  Lemma \ref{lem:VMO} implies that  $\chi \in {\rm VMO}(I)$. Therefore, integrating over $I \cap I_{x,\rho}$, we get
$$\dist(\chi_{x,\rho},\{0,1\}) \leq \frac{1}{\LL^1(I \cap I_{x,\rho})}\int_{I \cap I_{x,\rho}} |\chi-\chi_{x,\rho}|\, dy \to 0$$
as $\rho \to 0$, uniformly with respect to $x \in I$. In particular, there exists $\rho_0>0$ such that $\dist(\chi_{x,\rho},\{0,1\})<1/3$ for all $0 < \rho<\rho_0$ and all $x \in I$. Let $\rho<\rho_0$. Since the function $x \in I \mapsto \chi_{x,\rho}$ is continuous, then, either $\chi_{x,\rho}\le 1/3$ for all $x \in I$, or $\chi_{x,\rho}\ge 2/3$ for all $x \in I$.
But, since $\chi_{x,\rho} \to \chi(x)$ as $\rho\to 0$ for every Lebesgue point $x \in I$ of $\chi$, then, either $\chi=0$ a.e. in $I$, or $\chi=1$ a.e. in $I$.
\end{proof}

We next prove that a piecewise constant function on a Lipschitz graph cannot belong to $H^{1/2}$.

\begin{lemma}\label{lem:notH1/2}
Let $\Gamma=\Gamma^- \cup \Gamma^+$ where $\Gamma^\pm$ are segments in $\R^2$ such that $\overline \Gamma^+ \cap \overline \Gamma^-$ is a single
point.  Let $f:\Gamma\to \R$ be such that $f=\alpha$ on $\Gamma^-$ and $f=\beta$ on $\Gamma^+$ where $\alpha \neq \beta$. Then $f \not\in H^{1/2}(\Gamma)$.
\end{lemma}

\begin{proof}
Without loss of generality, we suppose that $\Gamma \supset\{(x_1,x_2) \in \R^2 : \; -L < x_1 <L,\; x_2=a|x_1|\}$ where $L>0$, $a>0$, $\Gamma^-=\{(x_1,x_2) \in \R^2 : \; -L < x_1 <0,\; x_2=-a x_1\}$ and $\Gamma^+=\{(x_1,x_2) \in \R^2 : \; 0< x_1 <L,\; x_2=a x_1\}$. Let us show that
$$\iint_{\Gamma \times \Gamma}\frac{|f(x)-f(y)|^2}{|x-y|^2}\, d\HH^1(x)\,
d\HH^1(y)=+\infty.$$
Indeed, using the definition of a curvilinear integral and a change of variables
\begin{eqnarray*}
&&\iint_{\Gamma \times \Gamma}\frac{|f(x)-f(y)|^2}{|x-y|^2}\, d\HH^1(x)\,
d\HH^1(y)\\
&&\hspace{2cm}= 2|\alpha-\beta|^2 \iint_{\Gamma^- \times \Gamma^+}\frac{d\HH^1(x)\, d\HH^1(y) }{|x-y|^2}\\
&&\hspace{2cm}= 2(1+a^2)|\alpha-\beta|^2 \iint_{(-L,0) \times (0,L)}\frac{dx_1\, dy_1 }{(x_1-y_1)^2+a^2(x_1+y_1)^2}\\
&&\hspace{2cm}= 2(1+a^2)|\alpha-\beta|^2 \iint_{ (0,L)^2}\frac{ds\, dt }{(s+t)^2+a^2(s-t)^2}.
\end{eqnarray*}
Since $(s+t)^2+a^2(s-t)^2 \leq 2(1+a^2)(s^2+t^2)$, we deduce that
\begin{eqnarray*}
\iint_{\Gamma \times \Gamma}\frac{|f(x)-f(y)|^2}{|x-y|^2}\, d\HH^1(x)\, d\HH^1(y)&  \geq &
|\alpha-\beta|^2 \iint_{ (0,L)^2}\frac{ds\, dt }{s^2+t^2}\\
& \geq& |\alpha-\beta|^2 \iint_{ D_L}\frac{ds\, dt }{s^2+t^2},
\end{eqnarray*}
where $D_L:=(0,L)^2 \cap B_L(0)$ is the  right upper quarter of a disk of radius $L$. Using a change of variables in polar coordinates, we get that
$$\iint_{ D_L}\frac{ds\, dt }{s^2+t^2}=\frac{\pi}{2} \int_0^L \frac{dr}{r}=+\infty,$$
which completes the proof of the result.
\end{proof}



\ack
{\scriptsize This research was supported by the National Science Fundation Grants  DMS-1615839 and DMS/DMREF-1922371 and was completed in part while the first author was visiting the Courant Institute in April 2019 with the support of Grant DMS-1615839.  {The first author's research  was also partially supported by a
public grant as part of the Investissement d'avenir project, reference ANR-11-LABX-0056-LMH, LabEx LMH.} The second author also wishes to thank the Flatiron Institute of the Simons Foundation for its hospitality.}


\frenchspacing
\bibliographystyle{cpam}

\begin{thebibliography}{99}








\bibitem{Amb}
 Ambrosio, L. Transport equation and Cauchy problem for BV vector fields.
\textit {Inventiones Math.} \textbf{158} (2004), 227--260.

\bibitem{AFP}
 Ambrosio, L., Fusco, N. and  Pallara, D. \textit {Functions of Bounded Variation and Free Discontinuity Problems}. Oxford University Press, Oxford, 2000.

\bibitem{A}
Anzellotti, G. Pairings between measures and bounded functions and compensated compactness. \textit{Ann. Mat. Pura Appl.} \textbf{135}(1983), no.
4, 293--318.

\bibitem{A2}
Anzellotti, G. On the existence of the rates of stress and displacements for Prandtl-Reuss plasticity. \textit{Quart. Appl. Math.} \textbf{41} (1984), 181--208.

\bibitem{anzellotti84}
Anzellotti,  G. On the extremal stress and displacement in Hencky plasticity. \textit{Duke Math. J.} \textbf{51}(1984), no.1, 133--147 .

\bibitem{AnzLuc1987}
Anzellotti, G. and Luckhaus, S. Dynamical evolution of elasto-perfectly plastic bodies. \textit {Appl. Math. Optim.} \textbf{15} (1987), no. 2, 121--140.

\bibitem{BabFr}
Babadjian, J.-F. and Francfort, G. A.  A note on the derivation of rigid-plastic models. \textit{Nonlinear Differential Equations Appl.} \textbf {23} (2016), no. 3, 23--37.

\bibitem{BMi}
Babadjian, J.-F.  and Mifsud, C. Hyperbolic structure for a simplified model of dynamical perfect plasticity. \textit{Arch. Rational Mech. Anal.} \textbf{223} (2017), no. 2, 761--815.

\bibitem{BM}
Babadjian, J.-F.  and Mora, M. G.  Stress regularity in quasi-static perfect plasticity with a pressure dependent yield criterion. \textit{J. Differential Eq.} \textbf{264} (2018), no. 8, 5109--5151.

\bibitem{Ball}
 Ball, J.M. Convexity conditions and existence theorems in nonlinear elasticity. \textit{Arch. Rat. Mech. Anal.} \textbf{63} (1977), no. 4, 337--403.

\bibitem{BF}
 Bensoussan, A. Frehse, J. Asymptotic behaviour of Norton-Hoff's law in plasticity theory and $H^1$ regularity. \textit{Boundary value problems for partial differential equations and applications, RMA Res. Notes Appl. Math.} \textbf{29}  (1993), 3--25.

 \bibitem{BP}
 Bochard, P. and Pegon, P. Kinetic selection principle for curl-free vector fields of unit norm. \textit{Comm. Partial Diff. Eq.} \textbf{42} (2017), 1375--1402.

\bibitem{BN}
Brezis, H. and  Nirenberg, L.  Degree theory and BMO. I. Compact manifolds without boundaries. \textit{Selecta Math. (N.S.)} {\bf 1} (1995), no. 2,197--263.

\bibitem{CF}
Chen, G.-Q. and Frid, H. Divergence-measure fields and hyperbolic conservation laws. \textit{Arch. Ration. Mech. Anal.} \textbf{147} (1999), no. 2. 89--118.

\bibitem{CD}
 G. Crippa, G. and De Lellis, C. Estimates for transport equations and regularity of the DiPerna-Lions flow. \textit{J. Reine Angew. Math.} \textbf{616} (2008), 15--46.

\bibitem{DMDSM}
 Dal Maso, G.,De Simone, A. and Mora, M.G.   Quasistatic evolution problems for linearly elastic perfectly plastic materials. \textit{Arch. Rat. Mech. Anal.} \textbf {181} (2006), no. 2, 237--291.

\bibitem{david}
 David, G. \textit{Singular sets of minimizers for the Mumford-Shah functional},  Progress in Mathematics (Volume 233), Birkha\"user Verlag, Basel, 2005.

\bibitem{DLS}
De Lellis, C. and  Sz\'ekelyhidi Jr., L.  The Euler equations as a differential inclusion.    \textit {Annals Math.} \textbf{170} (2009), no. 3, 1417--1436.

\bibitem{demyanov}
Demyanov, A. Regularity of stresses in Prandtl-Reuss perfect plasticity. \textit{ Calc. Var. Partial Differential Equations} \textbf{34} (2009), no. 1, 23--72.

\bibitem{DKMO}
De Simone, A., Kohn, R. V., M\"uller, S. and Otto, F. A compactness result in the gradient theory of phase transitions. \textit{Proc. Roy. Soc. Edinburgh} \textbf{131} (2001), 833--844.

\bibitem{DPL}
Di Perna, R. J.  and  Lions P. L. Ordinary differential equations, transport theory and Sobolev spaces. \textit{Invent. Math.} \textbf{98} (1989),
511--547.

\bibitem{EG}
Evans, L.C.  and Gariepy R.F.  \textit{Measure Theory and Fine Properties
of Functions}. CRC Press, Boca Raton, 1992.

\bibitem{FoGa}
Fonseca, I. and Gangbo, W. \textit{Degree theory in analysis and applications}, Oxford lecture series in mathematics and its applications 2. Clarendon Press, Oxford, 1995.

\bibitem{FG}
Francfort,  G. A.  and Giacomini, A. Small-strain heterogeneous elastoplasticity revisited. \textit{Comm. Pure Appl. Math.}
 \textbf{65} (2012), no. 9, 1185--1241.

\bibitem{FGM}
Francfort,  G. A., Giacomini, A. and Marigo, J.-J.  The taming of plastic
slips in von Mises elasto-plasticity. \textit{Interfaces Free Bound.} \textbf{17} (2015), 497--516.

\bibitem{FGM16}
Francfort,  G. A., Giacomini, A. and Marigo, J.-J. A case study for uniqueness of elasto-plastic evolutions: The bi-axial test. \textit{J. Math. Pures Appl.} \textbf{105} (2016), 198--227.

\bibitem{FGM17}
Francfort,  G. A., Giacomini, A. and Marigo, J.-J. The elasto-plastic exquisite corpse: A Suquet legacy. \textit{J. Mech. Phys. Sol.} \textbf{97} (2016), 125--139.

\bibitem{HP}
Henrot, A. and Pierre, M. \textit{Shape variation and optimization. A geometrical analysis}, EMS Tracts in Mathematics, 28. European Mathematical Society (EMS), Z\"urich, 2018.

\bibitem{I}
Ignat, R. Two-dimensional unit-length vector fields of vanishing divergence, \textit{ J. Funct. Anal.} \textbf{262} (2012), no. 8, 3465--3494.

\bibitem{JOP}
Jabin, P.-E. ,   Otto, F. and Perthame B.  Line-energy Ginzburg-Landau models: zero-energy states. \textit{Ann. Sc. Norm. Super. Pisa Cl. Sci. (5)} \textbf{1} (2002),  no. 1, 187--202.

\bibitem{kohn.temam}
Kohn, R.V. and Temam R. Dual spaces of stresses and strains, with applications to Hencky plasticity. \textit{Appl. Math. Optim.} \textbf{10} (1983), no. 1, 1--35.

\bibitem{Lub}
J. Lubliner, J. \textit{Plasticity Theory}. Macmillan Publishing Company,
New York, 1990.

\bibitem{mielke05}
Mielke A.  Evolution of rate-independent systems. \textit{Evolutionary equations. Vol. II}, Handb. Differ. Equ., 461--559.  Elsevier/North-Holland, Amsterdam, 2005.

\bibitem{Mo}
 Mora, M.G. Relaxation of the Hencky model in perfect plasticity. \textit{J. Math. Pures Appl.} \textbf{9-106} (2016), no. 4, 725--743.

\bibitem{MT}
 Murat, F. and Trombetti, C. A chain rule formula for the composition of a vector-valued function by a piecewise smooth function \textit{Boll. U. Mat. Ital., Serie 8} \textbf{6-B} (2003), no. 3, 581--595.

\bibitem{S}
Seregin, G.A.
On the differentiability of extremals of variational problems of the mechanics of ideally elastoplastic media, (Russian).
\textit{Differentsial'nye Uravneniya} \textbf{23} (1987), no. 11, 1981--1991.

\bibitem{suquet81}
 Suquet, P.-M. Sur les \'equations de la plasticit\'e: existence et r\'egularit\'e des solutions \textit{J. M\'ecanique} \textbf{20} (1981), no. 1, 3--39.

\bibitem{T}
Temam, R.  \textit{Probl\`emes Math\'ematiques en Plasticit\'e}. Gauthier-Villars, Paris, 1983.


\end{thebibliography}

\end{document}